\documentclass[trsc,nonblindrev]{informs3}

\OneAndAHalfSpacedXI

\usepackage{hyperref}
\usepackage{natbib}
\bibpunct[, ]{(}{)}{,}{a}{}{,}
\def\bibfont{\small}

\usepackage{amsmath,amssymb,mathtools}
\usepackage{tabularx,longtable,booktabs,array,multirow,makecell}
\usepackage{appendix}
\usepackage{algorithm,algpseudocode}
\usepackage{subcaption}
\usepackage{tikz}
\usetikzlibrary{arrows.meta, positioning, calc}
\usepackage{siunitx}
\usepackage{rotating}
\newcolumntype{P}[1]{>{\centering\arraybackslash}p{#1}}

\DeclareMathOperator{\CVaR}{CVaR}
\DeclareMathOperator{\CE}{CE}
\DeclareMathOperator{\VaR}{VaR}

\newcommand{\R}{\mathbb{R}}
\newcommand{\E}{\mathbb{E}}
\newcommand{\Prob}{\mathbb{P}}

\renewcommand{\bibfont}{\small}
\TheoremsNumberedThrough
\EquationsNumberedThrough

\begin{document}

\RUNAUTHOR{Bao, Vogiatzis, and Kontou}
\RUNTITLE{Risk-Averse Stochastic Assignment on Uncertain Networks}
\TITLE{Risk-Averse Stochastic User Equilibrium on Uncertain Transportation Networks}

\ARTICLEAUTHORS{%
 \AUTHOR{Wencheng Bao, M.Sc.}
 \AFF{Department of Civil and Environmental Engineering, University of Illinois Urbana-Champaign}
 \AUTHOR{Chrysafis Vogiatzis, Ph.D.}
 \AFF{Department of Industrial and Enterprise Systems Engineering, University of Illinois Urbana-Champaign}
 \AUTHOR{Eleftheria Kontou, Ph.D.*}
 \AFF{Department of Civil and Environmental Engineering, University of Illinois Urbana-Champaign, Corresponding author: \EMAIL{kontou@illinois.edu}}
}

\ABSTRACT{


Travel times in transportation networks are uncertain and exhibit right-skewed distributions with long tails under weather and incident disruptions. The current challenge is that expected cost traffic assignment underestimates rare but severe delays, while path-based logit stochastic user equilibrium (SUE) models assign positive flow to every prespecified route, and estimated scenario probabilities may be misspecified. This paper develops a risk- and ambiguity-aware truncated SUE framework integrating upper-tail risk, endogenous route activation, and distributional uncertainty. We use an origin–destination reservation generalized cost to endogenously exclude excessively costly routes from the active choice set. A mean-CVaR criterion balances average travel time and upper-tail delay. The framework produces two complementary models. Model A applies mean-CVaR to each route's random travel time. It produces route-specific tail-tilted scenario laws and a variational-inequality formulation. Model B applies mean-CVaR to the Beckmann congestion potential. It evaluates all routes under one common tail-tilted law and gives a convex program with a unique path flow solution. We establish common-tail conditions under which the two models agree on perceived and reservation costs, active routes, and equilibrium flows. We extend Model B with a 1-Wasserstein ambiguity set that selects an adverse law before CVaR re-weights its upper tail. Braess and Sioux Falls networks experiments confirm this agreement when a common severe-state block fills the CVaR tail. In the Braess network upgrade case, expectation-only evaluation rejects the reliability upgrade. Every tested case with a positive CVaR weight favors it. Jointly modeling route activation, tail risk, and law uncertainty can therefore alter equilibrium routing and infrastructure decisions.
}



\maketitle

\section{Introduction}
\label{sec:introduction}
Hazards induced by natural phenomena or anthropogenic activities often disrupt infrastructure systems and their operations, including transportation networks. Flooding affecting major roads can cause limited and delayed access to employment, educational functions, and health and emergency services \citep{lightbody_2017}. In a modeled 100-year urban flood for 2,564 cities in 177 countries across the world, approximately 15\% of road kilometers were inundated. Among trips that remain feasible, the average route becomes about 1.5 km longer and 3 minutes slower relative to baseline conditions \citep{HeRentschlerAvner2022}. 

Weather-related hazards measurably degrade mobility and road safety. \citet{loren_2024} estimates that 24\% of weather-related crashes occur on snowy, slushy, or icy pavement, with another 15\% during active snowfall or sleet. Rain, snow, high winds, fog, and hurricanes all elevate crash risk and reduce travel speeds \citep{6991566}. In Florida, moderate rainfall decreases average vehicle speeds by roughly 10 mph \citep{collins_pietrzyk_2001}. These disruptions are consequential for accessibility, safety, and economic activity, and are expected to intensify as road infrastructure ages and extreme events become more frequent \citep{gangani2024adapting}. 


Traffic assignment models are widely used to evaluate network
performance under static hazard scenarios. \citet{bucar_hayeri_2020} and \citet{twumasi-boakye_sobanjo_2018} use user equilibrium (UE) assignments to assess flooding and bridge-damage impacts, respectively, while \citet{kurmankhojayev_tan_chen_2025} develops a stochastic user equilibrium framework for link criticality analysis under disruption. These models can capture flow redistribution that responds to congestion.

Such disruptions typically result in right-skewed, long-tailed travel time distributions \citep{van_lint_van_zuylen_tu_2008, zang_xu_qu_chen_chen_2022}. Risk-neutral assignment, based on expected travel times, can underestimate rare, but severe, delays that can lead to precautionary behavior and risk-averse choices. To that extent, empirical evidence has consistently pointed to the fact that travelers tend to choose reliable routes rather than routes that are on average faster, but more unreliable \citep{liu_recker_chen_2004, carrion_levinson_2012}. Hence, capturing the upper tail through Conditional Value at Risk (CVaR) constructs is essential when hazards can lead to low-probability but high-severity outcomes.

We model travel time as a random variable induced by uncertain network conditions, without imposing a specific parametric distribution. Travelers are assumed to be risk-averse, so route choices depend not only on expected travel time but also on upper-tail risk, such as the likelihood and severity of severe delays. At the same time, the estimated travel-time distribution may be misspecified or shifted relative to the true one as in Figure~\ref{fig:into}. Such discrepancies may arise from infrastructure deterioration or changes in hazard frequency and severity \citep{chandra_2004,liu_wang_wang_koks_2023}. For example, an outcome around 15 minutes may be estimated to occur with probability $6.5\%$, while its true probability is only $5.5\%$, with the missing mass shifted toward more severe delay outcomes. This motivates a distributionally robust formulation that protects against such distributional errors.




\begin{figure}[htbp]
\centering
\includegraphics[width=0.6\textwidth]{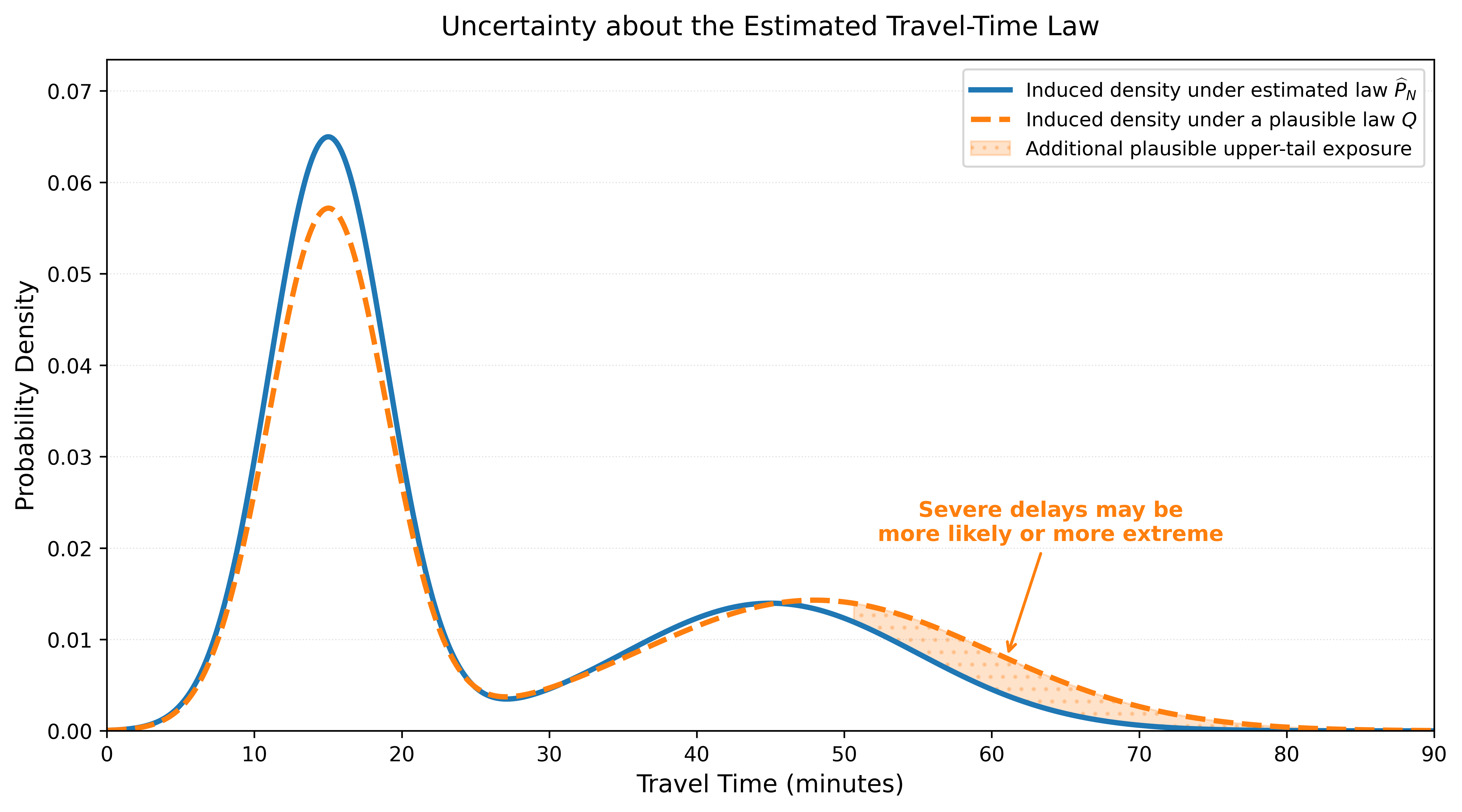}
\caption{Illustration of random travel time, upper-tail delay risk, and
distributional shift. The solid curve denotes the estimated travel time
density, and the dashed curve denotes a plausible true density.}
\label{fig:into}
\end{figure}

\section{Literature overview}
\label{sec:literature}

\subsection{Travel-time reliability and risk-averse traffic assignment}
\label{sec:lit_reliability_assignment}

Stochastic User Equilibrium (SUE) uses a logit choice mechanism to map perceived link costs into path selection probabilities. How each link's perceived travel time is quantified determines link use probabilities and equilibrium flows. When travelers are risk-averse, equilibrium can still be established under broad conditions \citep{palma_picard_2006, lianeas_nikolova_stier-moses_2018, ordonez2010wardrop}. In hazard-prone and affected networks, link travel times become random variables. The model must compress a travel time distribution into a behaviorally meaningful perceived time that reflects aversion to unreliability.

The contemporary literature develops several of such reliability measures. These range from scheduling-based reliability premia to additive variability measures, including their interaction with congestion \citep{batley_2007, fosgerau2010value, fosgerau_engelson_2011, xiao_coulombel_palma_2017b}. These constructs shape perceived link costs and logit probabilities, which in turn determine the equilibrium flow distribution. \cite{chen_zhou_2010, chen_zhou_lam_2011} define and extend an $\alpha$-reliable mean excess traffic equilibrium model combining buffer time and tardy time concepts, embedding tail delay sensitivity directly into the equilibrium mechanism. \cite{nikolova_stier-moses_2014} characterize equilibrium under mean-standard deviation objectives and bound efficiency losses for risk-averse commuters. \citet{xu_li_xie_chen_liu_2026} extend the mean-standard deviation traffic-assignment framework to a continuously distributed risk-aversion factor, deriving a finite-dimensional cumulative path cost characterization of Wardrop equilibrium for infinitely many risk preference classes.

In hazard-prone networks, the upper tail of the travel time distribution is fundamental. Rare but severe delays dominate perceived disutility and shape route choice. \cite{xu_chen_cheng_hong_kam_lo_2014} study an $\alpha$-reliable mean excess equilibrium model integrating buffer time and tardy time beyond a reliability threshold. A parallel stream adopts parametric representations, such as lognormal or right-skewed families, to capture tail behavior \citep{wang_osaragi_2024, taylor_susilawati_2012, kim_mahmassani_2015}. \cite{zang_batley_xu_david_z.w._wang_2024} formalize the tail through an unreliability area measuring upper tail mass and severity. Their framework derives the value of distribution tail and the value of travel time reliability, quantifying willingness to pay to reduce upper tail risk.

\subsection{Endogenous route-choice sets in stochastic user equilibrium}
\label{sec:lit_choice_sets}

The conventional path-based logit SUE determines path flows endogenously. It remains conditional on a prespecified candidate set and assigns positive probability to every included path. Endogenous flow allocation must therefore be distinguished from endogenous identification of behaviorally active paths. Equilibrated and restricted choice set models address this by jointly determining active routes and equilibrium flows. Dedicated algorithms have been developed for the logit formulations \citep{watling_rasmussen_prato_nielsen_2015, rasmussen_watling_prato_nielsen_2015}. Bounded-choice SUE allows sufficiently unattractive alternatives to remain unused \citep{watling_rasmussen_prato_nielsen_2018}. Recursive link-based models avoid finite path enumeration through unrestricted choice sets \citep{fosgerau_frejinger_karlstrom_2013}. \cite{tan_xu_chen_2024} formulate a convex SUE model with truncated path choice that endogenously distinguishes active from inactive paths. We adopt their central endogenization device: the OD-specific maximum acceptable generalized cost is the multiplier of demand conservation, rather than a separately calibrated safety budget. The present study retains this truncated choice structure while introducing hazard-dependent stochastic supply and risk- and ambiguity-aware perceived costs.

\subsection{Optimization under uncertainty and distributionally robust transportation}
\label{sec:lit_optimization_dro}

To solve problems under uncertainty, stochastic programming assumes a prescribed probability law or scenario distribution. Robust optimization instead protects against adverse realizations within an uncertainty set \citep{bertsimas_sim_2004, gabrel_murat_thiele_2014}. Distributionally robust optimization (DRO) takes a middle ground. It assumes that the probability law belongs to an ambiguity set rather than being fully specified. Moment-based and general convex ambiguity sets provide tractable representations of incomplete distributional information \citep{delage_ye_2010, wiesemann_kuhn_sim_2014}. Wasserstein sets centered on empirical distributions are especially useful in data-driven settings. They quantify distributional discrepancies and support tractable reformulations with statistical guarantees \citep{mohajerin_esfahani_kuhn_2017, gao_kleywegt_2022, gao_2022}.

In transportation applications, these include robust road network design under demand uncertainty and robust Wardrop equilibrium under stochastic demand and supply \citep{yin_madanat_lu_2009}. \cite{ahipasaoglu_arikan_natarajan_2019} develop a distributionally robust Markovian traffic equilibrium when marginal link-utility distributions are known but their dependence structure is ambiguous. \cite{wang_you_song_zhang_2020} formulate a Wasserstein distributionally robust shortest-path problem using worst-case $\alpha$-reliable mean-excess travel time. The former relies on sequential link choices rather than a path-based truncated logit model. The latter optimizes a single route without endogenous congestion feedback. These studies motivate, but do not encompass, a Wasserstein DRO formulation of risk-sensitive SUE.



\subsection{Research gap}
\label{sec:research_gap}

The first gap concerns risk-adjusted costs and their equilibrium properties. Quantile-based $\alpha$-reliable costs can be nonsmooth
and need not yield a convex equilibrium formulation, so uniqueness and algorithm convergence require additional conditions. Convexity supports global optimality checks, and strict convexity can ensure a unique flow solution. A reliability quantile also provides only a threshold and does not describe the severity of delays beyond it. Mean-excess and CVaR-based equilibrium models capture these upper-tail delays \citep{chen_zhou_2010, chen_zhou_lam_2011,
xu_chen_cheng_hong_kam_lo_2014}, but pure CVaR cannot distinguish routes with identical upper-tail performance but different normal travel times. Risk evaluation therefore needs to retain information about both normal and severe delay conditions, while its implications for equilibrium structure must be established.

A second gap concerns how these risk-adjusted costs determine path use. Conventional path-based logit SUE assigns positive probability to every finite cost candidate path, while bounded
and truncated choice models allow unattractive paths to remain unused \citep{watling_rasmussen_prato_nielsen_2018,
tan_xu_chen_2024}. These models provide a route choice mechanism, but incorporating disruption risk requires more than replacing deterministic travel times with risk-adjusted costs. Avoiding
a risky route shifts demand and changes congestion on other routes, so risk costs, OD reservation costs, active paths, and flows must be determined together. Different routes may also face different adverse states, and evaluating their risks separately need not preserve the potential structure of the original truncated model. The conditions under which a common-state potential formulation agrees with route-specific risk evaluation therefore need to be established.

A third gap concerns the network-state distribution used to form risk-adjusted costs. Conventional random utility SUE models include error terms to represent unobserved factors affecting route choice \citep{daganzo_sheffi_1977}. However, choice probabilities remain conditional on the modeled costs and the assumed joint distribution of these errors. Just including an error term does not address errors in the physical network-state probabilities used to calculate those costs. Even when the route-choice mechanism is appropriate, an inaccurate network-state distribution $P$ can still produce misleading risk-adjusted costs and equilibrium predictions. This is also different from risk aversion: placing extra weight on delays from a disruption with an estimated probability of $10\%$ reflects risk aversion, while questioning
whether that probability is accurate reflects distributional uncertainty. Neither stochastic dispersion in route choice nor greater risk aversion under a fixed $P$ automatically addresses
the latter. Existing robust models consider related but different settings. The distributionally robust shortest-path model of
\citet{wang_you_song_zhang_2020} considers uncertain travel time distributions without congestion feedback. The Markovian
equilibrium of \citet{ahipasaoglu_arikan_natarajan_2019}
includes congestion and uncertain dependence among link utilities, but uses sequential link choices with known marginal utility distributions. These studies do not directly address how uncertainty in the joint distribution of physical network
conditions changes risk-adjusted costs and endogenous path use in a truncated equilibrium with congestion feedback.


\subsection{Our contributions}
\label{sec:contribution}

This paper compares two ways to include risk aversion in a truncated stochastic user equilibrium (TSUE). We study how they differ, when they give the same equilibrium, and how to extend the common-state model when the probability distribution is uncertain. We use the truncated choice model of \citet{tan_xu_chen_2024}. The OD reservation generalized cost remains the multiplier of the demand conservation constraint, and the active paths are determined within a fixed candidate set. We make the following contributions.

We develop two mean-CVaR equilibrium models under a given joint distribution of network states. Model A applies mean-CVaR to each route's random travel time. Each route has its own tail-tilted scenario distribution, and the model is written as a variational inequality. Model B applies mean-CVaR to the canonically normalized Beckmann congestion potential. Its selected marginal costs are expected route travel times under one common tail-tilted distribution. Under the stated assumptions, Model B is a convex program with a unique path flow solution. We explain the common state-weighting assumption behind Model B and its user equilibrium interpretation. It is not a system optimal assignment or a general reformulation of Model A.

We also establish sufficient common-tail conditions for the two models to give the same equilibrium. At every feasible flow vector, suppose one CVaR selector is valid for every path travel time and for the congestion potential. The two models then have the same equilibrium flow. With this common selector, their generalized costs, OD reservation costs, active paths, and route-choice probabilities also agree. We give two sufficient cases: a common scenario ordering across all links, and a common severe-state block that fills the CVaR tail. These results show when the convex model preserves route-specific mean-CVaR evaluation. When these conditions do not hold, equivalence is not guaranteed, and the models can give different route flows.

In addition, we extend Model B to a $1$-Wasserstein distributionally robust model for an uncertain joint distribution of network states. The distribution can vary within an ambiguity set centered on the empirical distribution, but all paths still use a common risk-adjusted distribution. We explain two effects in perceived costs: the selection of a worst-case distribution within the ambiguity set and mean-CVaR tail weighting under that distribution. This separates uncertainty about the distribution from risk aversion under a given distribution. Under the stated assumptions, the robust model is convex and has a unique path-flow solution. We derive exact minimax and Wasserstein-dual reformulations and an exchange algorithm. The algorithm gives lower and upper bounds when the restricted main problem and all separation problems are solved globally.



\subsection{Our assumption}

All models describe an ex ante single-stage assignment. Travelers choose from a prespecified candidate path set before the network condition is realized. They cannot condition their route choices on the realized state, and rerouting within trip or recourse is not modeled. The behaviorally active subset of candidate paths is instead determined endogenously at equilibrium.

Models A and B require a specified probability distribution for the joint network conditions. For example, the estimated distribution may indicate a $10\%$ chance that a trip will take approximately 50 minutes. Model A evaluates this uncertainty separately for each route, whereas Model B evaluates all routes under one common state-salience rule. Travelers are not assumed to compute the equilibrium themselves; they respond to equilibrium-consistent generalized costs learned from experience or supplied by a common information system.

The distributionally robust extension relaxes the assumption that the estimated distribution is exact. Historical data may suggest that a severe trip takes approximately 50 minutes, whereas more recent experience may indicate a duration closer to 55 minutes or a higher frequency of such delays. The analyst therefore considers a set of nearby plausible distributions, and travelers respond to the resulting robust generalized costs.

The analysis assumes that travelers form equilibrium-consistent generalized costs from a common information set. Heterogeneous beliefs, private signals, real-time information, and adaptive rerouting are beyond the scope of this paper.

\section{Problem setting and truncated stochastic user equilibrium}
\label{sec:problem}

\subsection{Network, joint uncertainty, and feasible flows}

Let $G=(V,A)$ be a directed network and let $\mathcal O$ denote the set of origin--destination (OD) pairs. For each $(o,d)\in\mathcal O$, $q_{od}>0$ is fixed demand and $K_{od}$ is a finite nonempty candidate path set. The model determines the active subset within $K_{od}$; candidate path generation remains exogenous. Let $f_k^{od}$ and $x_a$ denote path and link flows, respectively, and let $\delta_{ak}^{od}\in\{0,1\}$ indicate whether path $k$ uses link $a$. The physical flow constraints common to all formulations are defined once as
\begin{subequations}
\label{eq:flow_feasibility}
\begin{align}
&\sum_{k\in K_{od}} f_k^{od}=q_{od},
&&\forall (o,d)\in\mathcal O,
\label{eq:demand_conservation_new}\\
&x_a=\sum_{(o,d)\in\mathcal O}\sum_{k\in K_{od}}
\delta_{ak}^{od}f_k^{od},
&&\forall a\in A,
\label{eq:link_flow_new}\\
&f_k^{od}\ge0,
&&\forall (o,d)\in\mathcal O,\ k\in K_{od},
\qquad
x_a\ge0,
&&\forall a\in A.
\label{eq:nonnegativity_new}
\end{align}
\end{subequations}
Define the feasible flow-pair set and its path flow projection by
\begin{equation}
\mathcal F_{fx}
:=\left\{(f,x):\eqref{eq:flow_feasibility}\text{ holds}\right\},
\qquad
\mathcal F
:=\left\{f:\exists x\text{ such that }(f,x)\in\mathcal F_{fx}\right\}.
\label{eq:feasible_set}
\end{equation}
Because \eqref{eq:link_flow_new} uniquely determines $x$ from $f$, write
\begin{equation}
x_a(f)
:=\sum_{(o,d)\in\mathcal O}\sum_{k\in K_{od}}
\delta_{ak}^{od}f_k^{od},
\qquad a\in A.
\label{eq:link_flow_reduced}
\end{equation}
All subsequent formulations invoke $(f,x)\in\mathcal F_{fx}$ rather than restating \eqref{eq:flow_feasibility}. When convenient, the theoretical analysis substitutes $x=x(f)$ and works on the reduced path-flow set $\mathcal F$.

Uncertainty is represented by a joint network state $\xi=(\xi_a)_{a\in A}\in\Xi$ with probability law $P$. Thus, dependence among link conditions is part of the model. For an arbitrary probability law $Q\in\mathcal P(\Xi)$ and any $Q$-integrable function $h$, expectation means integration with respect to that law:
\begin{equation}
\E_Q[h(\xi)]
=\int_{\Xi}h(\zeta)\,Q(\mathrm d\zeta).
\label{eq:general_expectation_integral}
\end{equation}
For the finite-scenario stochastic program, the joint realizations are $\xi^1,\ldots,\xi^S$ with probabilities $p_s>0$ and $\sum_{s=1}^S p_s=1$, so that $P=\sum_{s=1}^S p_s\delta_{\xi^s}$. Substituting this discrete law into \eqref{eq:general_expectation_integral} gives
\begin{equation}
\begin{aligned}
\E_P[h(\xi)] =\int_{\Xi}h(\zeta)\,P(\mathrm d\zeta) =\int_{\Xi}h(\zeta)\left(\sum_{s=1}^S p_s\delta_{\xi^s}\right)(\mathrm d\zeta) = \sum_{s=1}^S p_s h(\xi^s).
\end{aligned}
\label{eq:finite_expectation}
\end{equation}
Link-wise marginal probabilities are derived from this joint law. To make every probabilistic calculation transparent, all subsequent expectations are expanded again at their point of use rather than being left only in operator notation.

The joint state represents physical supply-side uncertainty, such as weather, incidents, or capacity conditions. It is distinct from the auxiliary Type-I extreme-value perturbations used below to motivate one algebraic form of the truncation kernel. The former generates random path travel times and tail risk; the latter is not part of the hazard distribution.



\subsection{Scenario-dependent link travel time}

For link $a$ and state $\xi$, we use the extended BPR form
\begin{equation}
t_a(x_a,\xi)
=t_a^0(\xi)\left[
1+\beta_a(\xi)
\left(\frac{x_a}{c_a(\xi)}\right)^{\gamma_a(\xi)}
\right]
+\Delta_a(\xi),
\label{eq:bpr_new}
\end{equation}
where $t_a^0(\xi)>0$ is free flow time, $c_a(\xi)>0$ is effective capacity, $\beta_a(\xi)\ge0$ is the BPR coefficient, $\gamma_a(\xi)\ge1$ is the BPR exponent, and $\Delta_a(\xi)\ge0$ is incident delay. These symbols avoid using the CVaR confidence level $\alpha$ as a BPR parameter. For scenario $s$, write $t_a^s(x_a):=t_a(x_a,\xi^s)$.

Under the stated parameter restrictions, $t_a^s$ is continuous and nondecreasing in $x_a$, and its integral is convex. The theory below assumes finite link costs and positive capacities. Therefore, complete roadway closure is not part of the core model; it requires a separate recourse or unserved-demand mechanism rather than setting $c_a=0$ in \eqref{eq:bpr_new}.

For path $k\in K_{od}$, scenario-$s$ travel time at link-flow vector $x$ is
$
T_k^{od,s}(x)
:=\sum_{a\in A}\delta_{ak}^{od}t_a^s(x_a).
$
After substituting \eqref{eq:link_flow_reduced}, we use the reduced notation
$
T_k^{od,s}(f)
:=T_k^{od,s}(x(f))
=\sum_{a\in A}\delta_{ak}^{od}t_a^s(x_a(f)).
$
The canonically normalized scenario congestion potential is
\begin{equation}
L^s(x)
:=\sum_{a\in A}\int_0^{x_a}t_a^s(\omega)\,\mathrm d\omega,
\qquad
L^s(0)=0,
\label{eq:canonical_potential_x}
\end{equation}
and its reduced form is $L^s(f):=L^s(x(f))$. We write $T_k^{od}(f,\xi)$ and $L(f,\xi)$ for the random variables satisfying
\begin{equation}
T_k^{od}(f,\xi^s)=T_k^{od,s}(f),
\qquad
L(f,\xi^s)=L^s(f),
\qquad
s=1,\ldots,S.
\label{eq:random_path_potential_values}
\end{equation}
The affine link flow definition then gives
\begin{equation}
\frac{\partial L^s(f)}{\partial f_k^{od}}
=T_k^{od,s}(f).
\label{eq:potential_gradient}
\end{equation}

The quantity $L^s(f)$ is a Beckmann potential, not the total system travel time. Its role is to integrate the scenario-$s$ path cost field: the derivative in \eqref{eq:potential_gradient} is the marginal path travel time for path $k$. Consequently, minimizing a convex combination or a coherent risk transform of these potentials generates a user equilibrium cost field through first-order conditions.

\subsection{Truncated kernel from convex regularization}
\label{subsec:truncated_choice}

Let $c_k^{od}$ be a deterministic generalized path cost. It may include travel time, tolls, and a deterministic path-overlap correction, provided all components use a common generalized-cost scale. Fix $\theta>0$ and a reference-flow scale $r_0>0$, and define
\begin{equation}
\psi_{r_0}(y)
:=(y+r_0)\ln\left(\frac{y+r_0}{r_0}\right)-y,
\qquad y\ge0, \qquad \Psi_{r_0}(f)
:=\sum_{(o,d)\in\mathcal O}
\sum_{k\in K_{od}}
\psi_{r_0}(f_k^{od}).
\label{eq:entropy_new}
\end{equation}
Its first two derivatives are
\begin{equation}
\psi_{r_0}'(y)=\ln\left(1+\frac{y}{r_0}\right),
\qquad
\psi_{r_0}''(y)=\frac{1}{y+r_0}>0.
\label{eq:entropy_derivatives}
\end{equation}
For one OD pair, the KKT conditions of the regularized problem
\begin{equation}
\min_{f^{od}\ge0}
\left\{
\sum_{k\in K_{od}}c_k^{od}f_k^{od}
+\frac{1}{\theta}
\sum_{k\in K_{od}}\psi_{r_0}(f_k^{od})
:
\sum_{k\in K_{od}}f_k^{od}=q_{od}
\right\}
\label{eq:regularized_od_problem}
\end{equation}
give the generic truncated-flow relation
\begin{equation}
f_k^{od}
=r_0
\left[
\exp\{\theta(\pi_{od}-c_k^{od})\}-1
\right]_+,
\label{eq:generic_truncated_flow}
\end{equation}
where $\pi_{od}$ is the multiplier of demand conservation and is uniquely determined by
\begin{equation}
r_0
\sum_{k\in K_{od}}
\left[
\exp\{\theta(\pi_{od}-c_k^{od})\}-1
\right]_+
=q_{od}.
\label{eq:generic_pi_balance}
\end{equation}
Consequently, the path-share kernel is
\begin{equation}
P_k^{od}
=
\frac{
\left[
\exp\{\theta(\pi_{od}-c_k^{od})\}-1
\right]_+
}{
\displaystyle
\sum_{\ell\in K_{od}}
\left[
\exp\{\theta(\pi_{od}-c_\ell^{od})\}-1
\right]_+
}.
\label{eq:truncated_probability_new}
\end{equation}
A path is active exactly when $c_k^{od}<\pi_{od}$. Thus, the active set and reservation generalized cost are jointly determined by the regularized equilibrium rather than by an externally imposed cutoff.

A reservation-cost benchmark construction provides a compact motivation for the positive part term in \eqref{eq:truncated_probability_new}. Introduce auxiliary perturbed scores
$
U_k=-\theta c_k^{od}+\varepsilon_k,
$
and
$
U_0=-\theta\pi_{od}+\varepsilon_0,
$
where the perturbations are independent Type-I extreme-value variables with a common scale. The binary comparison gives
\begin{equation}
B_k^{od}
=\Prob(U_k\ge U_0)
=
\frac{1}{
1+\exp\{-\theta(\pi_{od}-c_k^{od})\}
},
\label{eq:binary_auxiliary}
\end{equation}
and the transformation
$A_k^{od} :=\left[2B_k^{od}-1\right]_+$ has odds
$
\frac{A_k^{od}}{1-A_k^{od}}
=
\frac12
\left[
\exp\{\theta(\pi_{od}-c_k^{od})\}-1
\right]_+.
$
Normalizing these auxiliary odds produces \eqref{eq:truncated_probability_new}. This binary construction motivates the numerator, but the final multipath kernel is the equilibrium kernel generated by the convex regularizer; it is not claimed to be the multinomial probability of a standard additive random utility model.

The expression reduces to $(f+1)\ln(f+1)-f$ when flow is measured in units for which $r_0=1$. The value of $r_0$ must be reported and transformed consistently under any change of flow units. The pair $(\theta,r_0)$ controls the curvature and scale of the regularizer, while $\pi_{od}$ adjusts to satisfy the fixed demand. Hence $\pi_{od}$ is an endogenous reservation generalized cost rather than an exogenous parameter.

\subsection{Risk-neutral truncated stochastic user equilibrium}

Define the full scenario TSUE potential
$
Z^{\mathrm{TSUE}}(f,\xi)
:=L(f,\xi)+\frac{1}{\theta}\Psi_{r_0}(f),
$
where the entropy is scenario invariant. The nominal expectation-based TSUE is
\begin{equation}
\begin{aligned}
\min_{(f,x)\in\mathcal F_{fx}}\quad
&\E_P\!\left[Z^{\mathrm{TSUE}}(f,\xi)\right]
=\min_{(f,x)\in\mathcal F_{fx}} \sum_{s=1}^S p_s Z^{\mathrm{TSUE}}(f,\xi^s) =\min_{(f,x)\in\mathcal F_{fx}}
\left\{\sum_{s=1}^S p_s\sum_{a\in A}\int_0^{x_a(f)}t_a^s(\omega)\,\mathrm d\omega
+\frac{1}{\theta}\Psi_{r_0}(f)\right\}\\
& \text{s.t.} \eqref{eq:demand_conservation_new} - \eqref{eq:nonnegativity_new}
\end{aligned}
\label{eq:base_tsue_system}
\end{equation}
where the penultimate equality uses $\sum_{s=1}^S p_s=1$. Thus, the objective combines the explicitly weighted scenario Beckmann congestion potentials with the entropy regularizer. Membership in $\mathcal F_{fx}$ invokes the demand conservation, path-to-link aggregation, and nonnegativity requirements in \eqref{eq:flow_feasibility}; therefore, these constraints are not repeated in the program. Although the objective is expressed using the reduced notation $L^s(f)$, retaining $x$ and imposing $(f,x)\in\mathcal F_{fx}$ keeps the path--link relation explicit. When \eqref{eq:flow_feasibility} is expanded, the multiplier associated with \eqref{eq:demand_conservation_new} is the endogenous reservation generalized cost $\pi_{od}$.

\begin{proposition}
\label{prop:base_tsue}
Suppose every $t_a^s$ is continuous and nondecreasing and $\theta,r_0>0$. Problem~\eqref{eq:base_tsue_system} has a unique path-flow solution $f^\star$ and, through \eqref{eq:link_flow_new}, a unique associated link-flow vector $x^\star$. Define
\begin{equation}
\begin{aligned}
c_k^{od}(f^\star)
&:=\E_P\!\left[T_k^{od}(f^\star,\xi)\right]=\sum_{s=1}^S p_sT_k^{od}(f^\star,\xi^s) =\sum_{s=1}^S p_s\sum_{a\in A}\delta_{ak}^{od}
 t_a^s\!\left(
 \sum_{(o',d')\in\mathcal O}\sum_{\ell\in K_{o'd'}}
 \delta_{a\ell}^{o'd'}f_{\ell}^{o'd',\star}
 \right).
\end{aligned}
\label{eq:base_expected_path_cost}
\end{equation}
For every OD pair there is a unique multiplier $\pi_{od}^\star$ such that the generic relations \eqref{eq:generic_truncated_flow}--\eqref{eq:generic_pi_balance} hold with $c_k^{od}=c_k^{od}(f^\star)$ and $\pi_{od}=\pi_{od}^\star$. Consequently, $P_k^{od}=f_k^{od,\star}/q_{od}$ equals \eqref{eq:truncated_probability_new}, and path $k$ is active if and only if $c_k^{od}(f^\star)<\pi_{od}^\star$.
\end{proposition}

\noindent\textit{Proof and provenance.}
A complete proof is given in Appendix~\ref{app:proof_base_tsue}. The entropy/KKT mechanism follows the truncated-choice construction of \citet{tan_xu_chen_2024}; existence, strict convexity, and multiplier uniqueness are proved here.

The proof and the full KKT derivation are given in Appendix~\ref{app:truncated_derivation}. When the generalized cost consists only of the travel time, $\pi_{od}$ has time units; otherwise it is more precisely an endogenous reservation generalized cost.

\begin{remark}[Scope of path-set endogeneity]
The active path set is endogenous only within the finite candidate set $K_{od}$. Candidate-path generation and duplicate-path control remain implementation choices. A deterministic path-size or overlap correction can be included in $c_k^{od}$ without changing the truncation derivation, but the present theory does not claim an unrestricted route set.
\end{remark}

\section{CVaR and the mean-CVaR compromise}

The figure~\ref{fig:risk_qua} provide the general idea, where is the expected travel time, value-at-risk (VaR), and CVaR for an upper-tail probability of $10\%$, corresponding to a confidence level of $\alpha=0.90$.

\begin{figure}[htbp]
\centering
\includegraphics[width=0.6\textwidth]{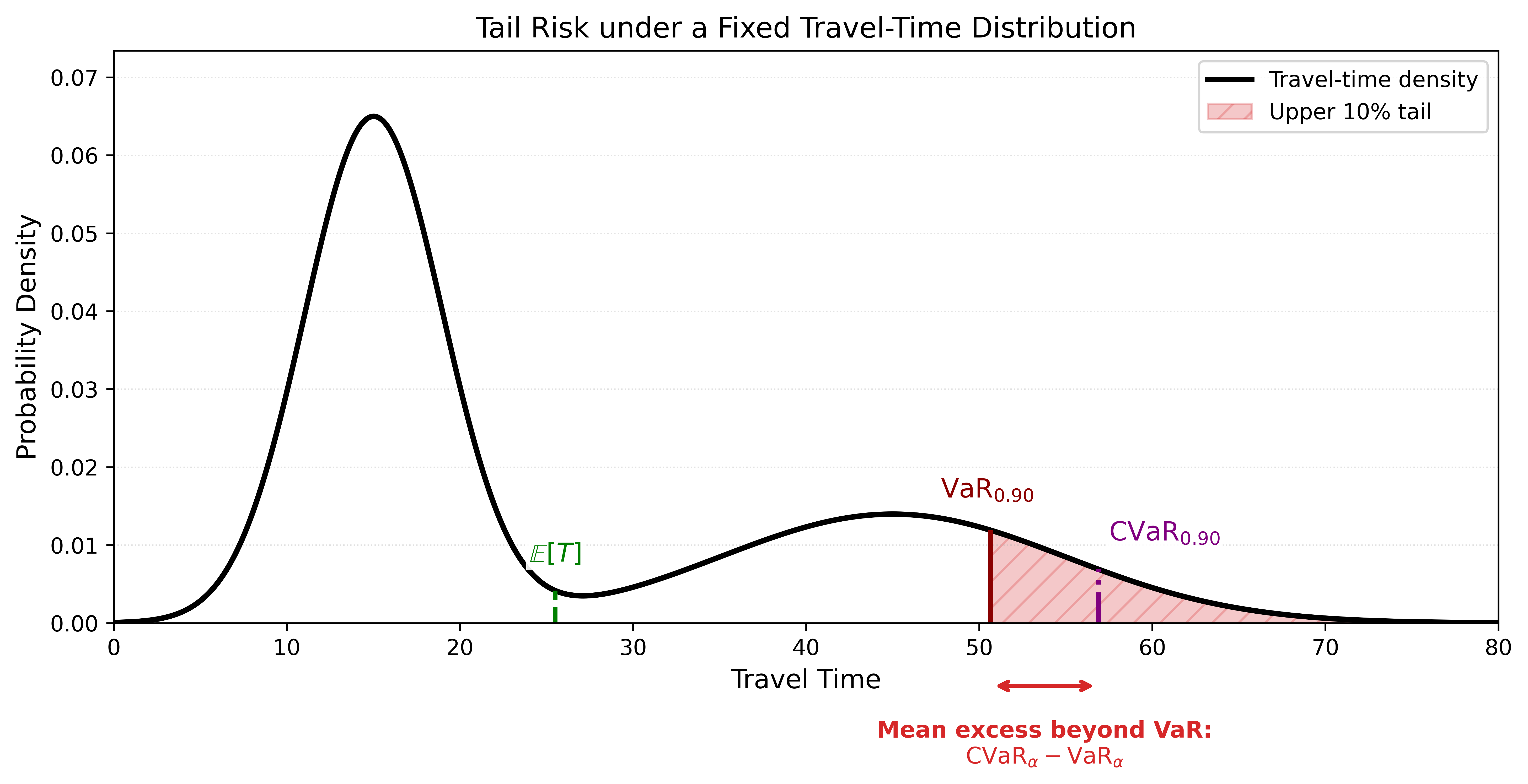}
\caption{Schematic illustration of mean, VaR, and CVaR for a travel time distribution. The shaded region is the upper \(10\%\) tail, and \(\mathrm{CVaR}_{0.90}-\mathrm{VaR}_{0.90}\) measures the mean excess travel time beyond the reliability threshold.}
\label{fig:risk_qua}
\end{figure}

For an integrable random cost $Y$ with law $P$ and confidence level $\alpha\in(0,1)$, we use the Rockafellar--Uryasev definition. Writing $Y(\xi^s)=y_s$ in the nominal finite-scenario model gives the complete chain
\begin{equation}
\begin{aligned}
\CVaR_\alpha(Y)
=\min_{\gamma\in\R}
\left\{\gamma+\frac{1}{1-\alpha}\E_P[(Y-\gamma)_+]\right\}
=\min_{\gamma\in\R}
\left\{\gamma+\frac{1}{1-\alpha}
\sum_{s=1}^S p_s(y_s-\gamma)_+\right\}.
\end{aligned}
\label{eq:cvar_ru_new}
\end{equation}
The first two lines are valid for continuous, discrete, and mixed distributions; the last line is the finite scenario specialization used in the nominal equilibrium models. To make the tail allocation used throughout the paper explicit, define the lower $\alpha$-quantile by
\begin{equation}
\VaR_\alpha(Y)
:=\inf\left\{y\in\R:\Prob(Y\le y)\ge\alpha\right\}.
\label{eq:var_definition}
\end{equation}
For a general probability law $P$, define the selector set
\begin{equation}
\begin{aligned}
\mathcal C_\alpha(P)
:=\left\{\chi:\Xi\to[0,1]\text{ measurable}:\E_P[\chi]=1-\alpha\right\}.
\end{aligned}
\label{eq:cvar_selector_set}
\end{equation}
The equivalent CVaR risk-envelope representation is
\begin{equation}
\begin{aligned}
\CVaR_\alpha(Y)
=\sup_{\chi\in\mathcal C_\alpha(P)}
\frac{1}{1-\alpha}\E_P[\chi Y]
=\sup_{\chi\in\mathcal C_\alpha(P)}
\frac{1}{1-\alpha}
\int_{\Xi}\chi(\zeta)Y(\zeta)\,P(\mathrm d\zeta).
\end{aligned}
\label{eq:cvar_risk_envelope}
\end{equation}
For the finite law $P=\sum_{s=1}^{S}p_s\delta_{\xi^s}$ with $Y(\xi^s)=y_s$, the selector set and the risk envelope become
\begin{equation}
\begin{aligned}
\mathcal C_\alpha^S
:=\left\{\chi\in[0,1]^S:\E_P[\chi]=1-\alpha\right\}=\left\{\chi\in[0,1]^S:\sum_{s=1}^{S}p_s\chi_s=1-\alpha\right\},
\end{aligned}
\label{eq:finite_selector_polytope}
\end{equation}
\begin{equation}
\begin{aligned}
\CVaR_\alpha(Y)
=\sup_{\chi\in\mathcal C_\alpha^S}
\frac{1}{1-\alpha}\E_P[\chi Y]=\sup_{\chi\in\mathcal C_\alpha^S}
\frac{1}{1-\alpha}\sum_{s=1}^{S}p_s\chi_sy_s.
\end{aligned}
\label{eq:cvar_risk_envelope_finite}
\end{equation}
Any maximizer of \eqref{eq:cvar_risk_envelope_finite} is called a valid finite scenario CVaR selector. To display one such selector explicitly, relabel the scenarios so that $y_1\le\cdots\le y_S$ and define
$
P_j:=\sum_{r=1}^{j}p_r, \text{ and  }m_\alpha:=\min\left\{j:P_j\ge\alpha\right\}.
$
A canonical selector that includes exactly the fraction of the VaR-boundary atom required to complete the upper $1-\alpha$ probability mass is
\begin{equation}
\chi_s^\alpha=
\begin{cases}
0,&s<m_\alpha,\\[2pt]
\dfrac{P_{m_\alpha}-\alpha}{p_{m_\alpha}},&s=m_\alpha,\\[8pt]
1,&s>m_\alpha.
\end{cases}
\label{eq:finite_selector}
\end{equation}
It satisfies the exact mass and CVaR identities
\begin{equation}
\begin{aligned}
\E_P[\chi^\alpha]
=\int_{\Xi}\chi^\alpha(\zeta)\,P(\mathrm d\zeta)
=\sum_{s=1}^{S}p_s\chi_s^\alpha
=1-\alpha,\quad 
\CVaR_\alpha(Y)
=\frac{1}{1-\alpha}\E_P[\chi^\alpha Y]
=\frac{(P_{m_\alpha}-\alpha)y_{m_\alpha}
+\sum_{s=m_\alpha+1}^{S}p_sy_s}{1-\alpha}.
\end{aligned}
\label{eq:finite_selector_mass_cvar}
\end{equation}
If several scenarios have the same value at $\VaR_\alpha(Y)$, the required boundary mass can be distributed among those tied scenarios in any way that preserves the first identity in \eqref{eq:finite_selector_mass_cvar}. The CVaR value is unchanged, although a decision derivative can depend on the selected allocation.

CVaR is useful because it accounts for both the probability and magnitude of upper-tail outcomes without imposing a parametric travel-time distribution. Pure CVaR nevertheless discards information below its selected tail. When several routes share the same catastrophic state, or when a large atom fills the entire upper tail, pure CVaR can assign similar perceived costs to routes that differ materially in ordinary conditions.

\begin{proposition}
\label{prop:cvar_upper_atom_new}
Suppose $Y\le M$ almost surely and $\Prob(Y=M)\ge1-\alpha$. Then
$
\CVaR_\alpha(Y)=M,
$
independently of the distribution of $Y$ below $M$.
\end{proposition}
\noindent\textit{Proof and provenance.} Appendix~\ref{app:proof_cvar_upper_atom} gives a self-contained proof from the CVaR risk envelope of \citet{rockafellar_uryasev_2002b}.

The proposition is relevant when all failed or censored trips are assigned the same finite penalty $M$. In that representation, pure CVaR cannot distinguish improvements below $M$ once the atom at $M$ fills the selected tail. The loss of discrimination is partly created by the common penalty or censoring convention rather than by CVaR alone. Conversely, expectation can underweight the catastrophic state. We therefore use the standard mean--CVaR certainty equivalent
\begin{equation}
\begin{aligned}
\CE_{\alpha,\eta}(Y)
:=(1-\eta)\E_P[Y]+\eta\CVaR_\alpha(Y)=(1-\eta)\sum_{s=1}^S p_sy_s
 +\eta\CVaR_\alpha(Y),
\qquad \eta\in[0,1].
\end{aligned}
\label{eq:ce_eta}
\end{equation}
The confidence level $\alpha$ determines which probability mass constitutes the upper tail, whereas the effective weight $\eta$ determines how strongly that tail enters perceived cost. The risk premium is
\begin{equation}
\begin{aligned}
\CE_{\alpha,\eta}(Y)-\E_P[Y] =\eta\left(\CVaR_\alpha(Y)-\E_P[Y]\right) =\eta\left(\CVaR_\alpha(Y)-\sum_{s=1}^S p_sy_s\right)\ge0.
\end{aligned}
\label{eq:ce_risk_premium_new}
\end{equation}
Keeping $\alpha$ and $\eta$ distinct is important: increasing $\alpha$ changes the composition of the tail, while increasing $\eta$ changes the weight assigned to the selected tail.

\begin{proposition}
\label{prop:tail_tilt_representation}
Let $Y$ take scenario values $y_s$ with probabilities $p_s$, $s=1,\ldots,S$, and let $\chi^Y\in\mathcal C_\alpha^S$ be any maximizer of the finite risk envelope in \eqref{eq:cvar_risk_envelope_finite}. Thus, $\chi^Y$ has total selected mass $1-\alpha$ and attains the CVaR value; \eqref{eq:finite_selector} is one canonical choice after ordering the scenario values. Define
\begin{equation}
\widetilde p_s^{Y}(\eta,\chi^Y)
:=p_s\left[(1-\eta)+\frac{\eta}{1-\alpha}\chi_s^Y\right],
\qquad s=1,\ldots,S.
\label{eq:tail_tilt_prob}
\end{equation}
Let $\widetilde P^Y$ denote the discrete law that assigns probability $\widetilde p_s^{Y}(\eta,\chi^Y)$ to scenario $s$. Then $\widetilde p^{Y}(\eta,\chi^Y)$ is a probability vector and
$
\CE_{\alpha,\eta}(Y)
=\E_{\widetilde P^{Y}}[Y]
=\sum_{s=1}^S\widetilde p_s^{Y}(\eta,\chi^Y)y_s.
$
\end{proposition}
\noindent\textit{Proof and provenance.} Appendix~\ref{app:proof_tail_tilt} gives the complete algebraic proof; the probability-tilt identity is derived in this paper.

Proposition~\ref{prop:tail_tilt_representation} provides the central behavioral bridge used in the paper. Mean--CVaR can be represented as ordinary expectation under a probability vector that preserves part of the nominal law and reallocates the remaining weight toward the upper tail. At $\eta=0$, the perceived probabilities equal $p_s$. At $\eta=1$, all perceived mass is concentrated on the exact CVaR tail, including any fractional VaR-boundary atom. This is a representation of risk-adjusted state weighting; it does not assert that the physical state probabilities have changed.

\section{Risk aggregation enters the equilibrium models}
The probability representation in Proposition~\ref{prop:tail_tilt_representation} exposes the precise difference between the two equilibrium constructions. In this section, we propose two model: Model A and B. Model A computes a separate selector for each path-time random variable, so each path can be evaluated under its own tail-tilted probability vector. Model B computes one selector from a network-wide systemic-severity index (the canonical congestion potential) and applies the resulting common tail-tilted probability vector to every path and OD pair through marginal costs. The first specification permits route-local state weighting; the second imposes a common state-weighting rule for the relative salience of network states. Section~\ref{sec:equivalence} shows that they become identical when the selectors coincide.

\subsection{Model A: path-based risk-sensitive TSUE as a variational inequality}
\label{sec:modelA}
Model A is the direct route-risk specification. At a candidate flow vector $f$, every path has its own random travel-time variable $T_k^{od}(f,\xi)$, defined through \eqref{eq:random_path_potential_values}. Specializing the certainty equivalent in \eqref{eq:ce_eta} to this random variable gives the path-level risk-adjusted generalized cost
\begin{equation}
\begin{aligned}
c_k^{A,od}(f)
&:=\CE_{\alpha,\eta}\bigl(T_k^{od}(f,\xi)\bigr)=(1-\eta)\E_P\!\left[T_k^{od}(f,\xi)\right]
 +\eta\CVaR_\alpha\!\left(T_k^{od}(f,\xi)\right) \\
&=(1-\eta)\sum_{s=1}^S p_sT_k^{od,s}(f)
 +\eta\CVaR_\alpha\!\left(T_k^{od}(f,\xi)\right)
\end{aligned}
\label{eq:modelA_cost}
\end{equation}
For each path, choose a valid selector $\chi_k^{A,od}(f)$ that is optimal in \eqref{eq:cvar_risk_envelope_finite}. Let $\widetilde p_{k,s}^{A,od}(f)$ denote the specialization of the generic tilted probability in \eqref{eq:tail_tilt_prob} to $Y=T_k^{od}(f,\xi)$, and let $\widetilde P_k^{A,od}(f)$ be the resulting discrete law. Proposition~\ref{prop:tail_tilt_representation} then gives
\begin{equation}
\begin{aligned}
c_k^{A,od}(f)
&=\E_{\widetilde P_k^{A,od}(f)}\!\left[T_k^{od}(f,\xi)\right]
=\sum_{s=1}^S\widetilde p_{k,s}^{A,od}(f)T_k^{od,s}(f) =\sum_{s=1}^S\widetilde p_{k,s}^{A,od}(f)
  \sum_{a\in A}\delta_{ak}^{od}t_a^s\!\left(x_a(f)\right)\\
&=\sum_{s=1}^S\widetilde p_{k,s}^{A,od}(f)
  \sum_{a\in A}\delta_{ak}^{od}
  t_a^s\!\left(
  \sum_{(o',d')\in\mathcal O}\sum_{\ell\in K_{o'd'}}
  \delta_{a\ell}^{o'd'}f_\ell^{o'd'}
  \right).
\end{aligned}
\label{eq:modelA_tilted_cost}
\end{equation}
Thus, Model A can be read as a route-choice model with route-specific tail-tilted scenario laws. A scenario that lies in the upper tail of one route need not lie in the upper tail of another, so the risk-adjusted state probabilities may vary across paths even when all travelers share the same parameters $(\alpha,\eta)$. This flexibility is behaviorally attractive when hazards affect corridors unevenly or when travelers evaluate reliability alternative by alternative.

The entropy potential and its derivative have already been defined in \eqref{eq:entropy_new} and \eqref{eq:entropy_derivatives}. The entropy-adjusted equilibrium mapping is therefore
\begin{equation}
F_k^{A,od}(f)
:=c_k^{A,od}(f)
+\frac{1}{\theta}\psi_{r_0}'(f_k^{od})
= c_k^{A,od}(f)
+ \frac{1}{\theta}\ln\left(1+\frac{f_k^{od}}{r_0}\right).
\label{eq:modelA_mapping}
\end{equation}
The scalar function $\psi_{r_0}$ enters the optimization potential, whereas its derivative enters the marginal path-cost mapping. Model A is the variational inequality (VI):
\begin{equation}
\text{find }(f^\star,x^\star)\in\mathcal F_{fx}\text{ such that}\quad
\sum_{(o,d)\in\mathcal O}\sum_{k\in K_{od}}
F_k^{A,od}(f^\star)(y_k^{od}-f_k^{od,\star})\ge0, \quad \eqref{eq:demand_conservation_new} - \eqref{eq:nonnegativity_new}, 
\quad
\forall (y,z)\in\mathcal F_{fx}.
\label{eq:modelA_vi_system}
\end{equation}
Membership in $\mathcal F_{fx}$ invokes all physical constraints in \eqref{eq:flow_feasibility}. The first term in $F^A$ is the route's risk-adjusted perceived travel cost; the second is the marginal entropy cost that generates stochastic dispersion and finite truncation. The VI therefore states that no feasible infinitesimal reallocation of OD demand can reduce the entropy-adjusted perceived cost.

\begin{proposition}
\label{prop:modelA_existence}
Suppose each $t_a^s$ is continuous. Then Model A has at least one solution. If the path-risk mapping $c^A(f)$ is monotone on $\mathcal F$, that is,
$
\bigl(c^A(f)-c^A(g)\bigr)^\top(f-g)\ge0,
\quad \forall f,g\in\mathcal F,
$
then $F^A$ is strongly monotone and Model A has a unique path-flow solution and therefore a unique associated link-flow vector.
\end{proposition}
\noindent\textit{Proof and provenance.} Appendix~\ref{app:proof_modelA_existence} gives a self-contained projection--Brouwer existence proof and the strong-monotonicity uniqueness argument.

\begin{proposition}
\label{prop:modelA_kernel}
A feasible flow pair $(f^\star,x^\star)$ solves \eqref{eq:modelA_vi_system} if and only if, for every OD pair, there is a scalar $\pi_{od}^{A}$ such that
\begin{equation}
0\le f_k^{od,\star}
\ \perp\ 
F_k^{A,od}(f^\star)-\pi_{od}^{A}\ge0,
\qquad k\in K_{od}.
\label{eq:modelA_complementarity}
\end{equation}

Equivalently, the generic truncated-flow and demand-balance relations \eqref{eq:generic_truncated_flow}--\eqref{eq:generic_pi_balance} hold with $c_k^{od}=c_k^{A,od}(f^\star)$ and $\pi_{od}=\pi_{od}^{A}$. Hence $\pi_{od}^{A}$ is unique, path $k$ is active if and only if $c_k^{A,od}(f^\star)<\pi_{od}^{A}$, and the associated link flow is recovered through \eqref{eq:link_flow_new}.
\end{proposition}
\noindent\textit{Proof and provenance.} Appendix~\ref{app:proof_modelA_kernel} proves the equivalence between the VI, simplex complementarity conditions, and the generic truncated kernel.

The reservation cost $\pi_{od}^{A}$ responds to all route-specific perceived costs. Consequently, increasing risk sensitivity does not imply a simple fixed-threshold nesting of active path sets: a path's direct risk premium and the indirect adjustment of $\pi_{od}^{A}$ operate simultaneously. The complete proofs are in Appendices~\ref{app:proof_modelA_existence} and~\ref{app:proof_modelA_kernel}. The monotonicity condition is sufficient rather than automatic. Applying pathwise CVaR to overlapping route sums can destroy a Beckmann structure because the tail selector can change by route and by flow. Under the additional regularity required by the chosen solver, monotone-VI methods such as projection or extragradient schemes can be used. 


\subsection{Model B: common-state-weighted potential TSUE}
\label{sec:modelB}

Model B is derived from a common state-weighting restriction. Model A lets each path identify its own CVaR tail, which is behaviorally flexible but can destroy a Beckmann potential. Model B instead assumes that all paths and OD pairs are evaluated under one common risk-adjusted scenario law. To generate this law endogenously while preserving integrability, states are ranked by the canonical congestion potential $L(f,\xi)$ in \eqref{eq:canonical_potential_x}--\eqref{eq:random_path_potential_values}. This potential is not total system travel time; by \eqref{eq:potential_gradient}, its path-flow derivative is the corresponding scenario path travel time.

Specializing \eqref{eq:ce_eta} to the systemic-severity random variable $L(f,\xi)$ defines
\begin{equation}
\begin{aligned}
\mathcal R_{\alpha,\eta}(f)
:=\CE_{\alpha,\eta}\bigl(L(f,\xi)\bigr)=(1-\eta)\E_P[L(f,\xi)]
 +\eta\CVaR_\alpha\bigl(L(f,\xi)\bigr)
=(1-\eta)\sum_{s=1}^S p_sL^s(f)
 +\eta\CVaR_\alpha\bigl(L(f,\xi)\bigr).
\end{aligned}
\label{eq:modelB_risk}
\end{equation}
Let $\chi\in\mathcal C_\alpha^S$ be any selector that is valid for $L(f,\xi)$. The selected marginal path cost generated by \eqref{eq:modelB_risk} is
\begin{equation}
\begin{aligned}
g_k^{B,od}(f;\chi)
:=(1-\eta)\E_P\!\left[T_k^{od}(f,\xi)\right]
+\frac{\eta}{1-\alpha}\E_P\!\left[\chi(\xi)T_k^{od}(f,\xi)\right]
=\sum_{s=1}^S p_s\left[(1-\eta)+\frac{\eta}{1-\alpha}\chi_s\right]
 \sum_{a\in A}\delta_{ak}^{od}t_a^s\!\left(x_a(f)\right).
\end{aligned}
\label{eq:modelB_marginal}
\end{equation}
Define the common tail-tilted scenario probabilities
\begin{equation}
\widetilde p_s^{B}(f;\chi)
:=p_s\left[(1-\eta)+\frac{\eta}{1-\alpha}\chi_s\right],
\qquad s=1,\ldots,S,
\label{eq:modelB_tilted_law}
\end{equation}
and let $\widetilde P^{B}(f;\chi)$ be the corresponding discrete law.

\begin{proposition}
\label{prop:modelB_common_law}
For every feasible $f$ and every admissible selector $\chi$, $\widetilde p^{B}(f;\chi)$ is a probability vector and
\begin{equation}
\begin{aligned}
g_k^{B,od}(f;\chi)
=\E_{\widetilde P^{B}(f;\chi)}\!\left[T_k^{od}(f,\xi)\right]
=\sum_{s=1}^S\widetilde p_s^{B}(f;\chi)
 \sum_{a\in A}\delta_{ak}^{od}t_a^s\!\left(x_a(f)\right),
\quad \forall (o,d),\ k\in K_{od}.
\end{aligned}
\label{eq:modelB_expected_under_tilt}
\end{equation}
Thus, one selected subgradient of the risk-adjusted potential is the path-cost field obtained by evaluating every route under the same endogenous tail-tilted scenario law.
\end{proposition}
\noindent\textit{Proof and provenance.} Appendix~\ref{app:proof_modelB_common_law} applies the probability transformation in \eqref{eq:tail_tilt_prob} and the selector-set condition in \eqref{eq:finite_selector_polytope}.

The scalar program now follows from integrability: the selected vector $g^B(f;\chi)$ belongs to $\partial\mathcal R_{\alpha,\eta}(f)$, and the entropy derivative supplies the truncated-choice term. Hence the common-state-weighted potential representation is
\begin{equation}
\min_{(f,x)\in\mathcal F_{fx}}
\left\{
\mathcal R_{\alpha,\eta}(f)
+\frac{1}{\theta}\Psi_{r_0}(f)
\right\}.
\label{eq:modelB_main_system}
\end{equation}
Using the RU representation in \eqref{eq:cvar_ru_new}, and writing $u(\xi^s)=u_s$, an equivalent finite-scenario formulation is
\begin{subequations}
\label{eq:modelB_ru}
\begin{align}
\min_{(f,x)\in\mathcal F_{fx},\,\gamma\in\R,\,u}\quad
\mathcal J_B(f,\gamma,u) &:= (1-\eta)\sum_{s=1}^S p_sL^s(f)
+\eta\gamma
+\frac{\eta}{1-\alpha}\sum_{s=1}^S p_su_s
+\frac{1}{\theta}\Psi_{r_0}(f)\nonumber\\
={}&(1-\eta)\sum_{s=1}^S p_s\sum_{a\in A}
\int_0^{x_a(f)}t_a^s(\omega)\,\mathrm d\omega
+\eta\gamma
+\frac{\eta}{1-\alpha}\sum_{s=1}^S p_su_s
+\frac{1}{\theta}\Psi_{r_0}(f)\nonumber\\
\text{s.t.}\quad
&u_s\ge L^s(f)-\gamma,
\quad u_s\ge0,  \quad \eqref{eq:demand_conservation_new} - \eqref{eq:nonnegativity_new},
\quad s=1,\ldots,S.
\end{align}
\end{subequations}
The RU variables $(\gamma,u)$ represent tail excess only; membership in $\mathcal F_{fx}$ continues to impose the physical constraints in \eqref{eq:flow_feasibility}.

\begin{proposition}
\label{prop:modelB_unique}
If every $t_a^s$ is continuous and nondecreasing, then \eqref{eq:modelB_main_system} is a convex optimization problem and has a unique optimal path-flow vector and a unique associated link-flow vector.
\end{proposition}
\noindent\textit{Proof and provenance.} Appendix~\ref{app:proof_modelB_unique} proves convexity through the Rockafellar--Uryasev representation and uniqueness through strict entropy convexity.

\begin{proposition}
\label{prop:modelB_kkt}
At the unique solution $(f^\star,x^\star)$ of Model B, there exist an admissible selector $\chi^\star$ and OD multipliers $\pi_{od}^{B}$ such that
$
0\le f_k^{od,\star}
\ \perp\ 
 g_k^{B,od}(f^\star;\chi^\star)
+\frac{1}{\theta}\psi_{r_0}'(f_k^{od,\star})
-\pi_{od}^{B}\ge0.
$
Equivalently, \eqref{eq:generic_truncated_flow}--\eqref{eq:generic_pi_balance} hold with $c_k^{od}=g_k^{B,od}(f^\star;\chi^\star)$ and $\pi_{od}=\pi_{od}^{B}$. For the selected marginal-cost vector, each OD demand-balance equation has a unique reservation-cost multiplier. The flow pair $(f^\star,x^\star)$ satisfies the physical
conditions in \eqref{eq:flow_feasibility}.
\end{proposition}

\noindent\textit{Proof and provenance.} Appendix~\ref{app:proof_modelB_kkt} derives the selected CVaR subgradient and the convex KKT system; the subgradient calculus follows \citet{rockafellar_uryasev_2002b}.

\begin{corollary}
\label{cor:modelB_sue_representation}
A feasible flow pair $(f^\star,x^\star)$ solves Model B if and only if there exists a selector $\chi^\star$ valid for $L(f^\star,\xi)$ such that
\begin{equation}
\sum_{(o,d)\in\mathcal O}\sum_{k\in K_{od}}
\left[
\sum_{s=1}^S\widetilde p_s^{B}(f^\star;\chi^\star)T_k^{od,s}(f^\star)
+\frac{1}{\theta}\psi_{r_0}'(f_k^{od,\star})
\right]
(y_k^{od}-f_k^{od,\star})\ge0,
\quad \forall (y,z)\in\mathcal F_{fx}.
\label{eq:modelB_common_law_vi}
\end{equation}
Hence, Model B is a truncated stochastic user equilibrium under a common endogenous tail-tilted scenario law.
\end{corollary}
\noindent\textit{Proof and provenance.} Appendix~\ref{app:proof_modelB_sue_corollary} proves both directions of the stated equivalence.

The corollary clarifies why the term user equilibrium remains appropriate. The equilibrium condition is expressed in perceived path costs, and the derivative of the potential generates those costs. A system-optimal assignment would instead internalize the effect of a user's flow on total travel time and would involve marginal social congestion costs. No such marginal-social-cost term appears in \eqref{eq:modelB_expected_under_tilt} or \eqref{eq:modelB_common_law_vi}. Model B is therefore not a generic algorithm for solving Model A; outside the common-tail cases in Section~\ref{sec:equivalence}, it is a tractable common-state-weighted equilibrium specification.

\begin{remark}[One potential model]
For fixed $f$, the entropy term $\Psi_{r_0}(f)/\theta$ is constant with respect to the scenario. Translation equivariance of expectation and CVaR, applied to \eqref{eq:ce_eta}, therefore implies
\begin{equation}
\CE_{\alpha,\eta}\left(L(f,\xi)+\frac{1}{\theta}\Psi_{r_0}(f)\right)
=\mathcal R_{\alpha,\eta}(f)+\frac{1}{\theta}\Psi_{r_0}(f).
\end{equation}
Thus, placing the scenario-invariant entropy inside or outside the risk functional yields the same reduced objective. These are equivalent representations of Model B, not distinct formulations.
\end{remark}

\begin{remark}[Canonical normalization]
A deterministic Beckmann equilibrium is unchanged if a scenario-specific constant is added to its potential, but a system-level CVaR can change because such constants alter scenario ordering. Model B is therefore defined using the canonical normalization in \eqref{eq:canonical_potential_x}. Under this normalization, $L^s(f)$ is the integrated congestion cost required to build the current flow from zero in scenario $s$ and serves as the common systemic-severity index that determines the tail distortion. The tail-tilted-law representation does not eliminate this modeling choice; it makes its role explicit.
\end{remark}

\begin{remark}[Behavioral status of Model B.]
Model B does not assume that travelers directly perceive or optimize the Beckmann potential. Rather, the potential is the canonically normalized exact primitive of the scenario-specific path-cost field, satisfying
$
    \frac{\partial L^{s}(f)}
         {\partial f_{k}^{od}}
    =
    T_{k}^{od,s}(f).
    \label{eq:beckmann-path-cost-primitive}
$
Applying mean-CVaR to this primitive generates, through a selected subgradient, perceived path costs equal to expected route travel times under one common tail-tilted state law. Under the common-tail condition of Proposition~\ref{prop:common_tail} which will be introduced in the next subsection, this state law is simultaneously valid for the CVaR evaluation of every route. Model B is then behaviorally identical to Model A, rather than merely solution-equivalent. Outside that
condition, Model B represents a distinct state-based risk-perception specification in which all alternatives are evaluated under a common endogenous state-salience law.
\end{remark}

\subsection{When the path- and potential-based models coincide}
\label{sec:equivalence}
The representations make the connection between the two models transparent. Model A associates path $k$ with a selector $\chi_k^{A,od}(f)$ and a path-specific tail-tilted law $\widetilde p_k^{A,od}(f)$. Model B associates the network potential with one selector $\chi^B(f)$ and one common tail-tilted scenario law $\widetilde p^B(f)$ used for every route. The following proposition gives a sufficient common-tail condition under which the two models have the same
equilibrium flow and admit matching generalized costs.

\begin{proposition}
\label{prop:common_tail}
Suppose that, for every feasible $f$, there exists a selector $\chi^\alpha(f)$ that is simultaneously valid for the CVaR of every path travel-time random variable $T_k^{od}(f,\xi)$ and for the CVaR of $L(f,\xi)$. Then
$
c_k^{A,od}(f)=g_k^{B,od}(f;\chi^\alpha(f)),
\quad \forall (o,d),\ k\in K_{od},\ f\in\mathcal F.
$
Every solution of Model A is therefore optimal for Model B. Since Model A has at least one solution and Model B has a unique path-flow optimizer, the two solution sets coincide and consist of that unique path-flow vector.
\end{proposition}
\noindent\textit{Proof and provenance.} Appendix~\ref{app:proof_common_tail} gives the complete proof, including the common-ordering and severe-block sufficient conditions.

The common-tail condition also implies
$
\widetilde p_{k,s}^{A,od}(f)
=
\widetilde p_s^{B}(f;\chi^\alpha(f)),
\quad \forall (o,d),\ k,\ s,
$
because both sides specialize \eqref{eq:tail_tilt_prob} to the same selector
$\chi^\alpha(f)$. Thus, the equivalence is behavioral, not merely algebraic:
all routes are evaluated under the same risk-adjusted scenario probabilities.
Since the generalized costs agree, the OD demand-balance equations, reservation
costs, active sets, and truncated probabilities also coincide.

One sufficient condition is a common network-wide scenario ordering,
$
t_a^1(x)\le t_a^2(x)\le\cdots\le t_a^S(x),
\quad \forall a\in A,\ x\ge0 .
$
Then every path time and the canonical potential inherit the same order, so a common index-based CVaR selector is valid, including fractional allocation at a VaR boundary. This condition is natural for ordered systemic network states in which
every link is weakly worse under a more severe state, although deterioration magnitudes may differ by link. A weaker sufficient condition is a severe block
$H_\alpha$ such that
$
\sum_{s\in H_\alpha}p_s=1-\alpha,
\quad
t_a^\ell(x)\le t_a^h(x),
\quad \forall a,\ x,\ \ell\notin H_\alpha,\ h\in H_\alpha .
$
Then only the partition at the $\alpha$ cutoff must be common; rankings within the non-tail block and within the severe block may differ. These are sufficient conditions, not a general behavioral identity. The proof is in
Appendix~\ref{app:proof_common_tail}.

For example, if two paths rank scenarios as $a,b,c\mid d,e,f$ and $b,c,a\mid e,f,d$, where the vertical bar denotes the $\alpha$ cutoff, the internal orderings within $\{a,b,c\}$ and $\{d,e,f\}$ do not matter: the expectation term is order-free, and CVaR uses the same tail block $\{d,e,f\}$. If $\alpha$ cuts inside a heterogeneous severe block, however,
different paths may select different VaR-boundary or tail scenarios, so Model A-Model B equivalence is no longer guaranteed. Outside common-tail cases, Model B is a tractable common-state-weighted surrogate rather than a numerical
method for solving Model A. When $\eta=0$, the CVaR selector drops out of the generalized costs, so the common-tail conditions are unnecessary and the two models have the same solution set.

\section{$1$-Wasserstein distributionally robust common-state-weighted TSUE}
\label{sec:dro}

Models A and B evaluate tail risk under a given joint distribution of network states. The common-tail analysis gives sufficient conditions for the two models to agree. But neither model accounts for errors in the given distribution. Increasing the CVaR weight puts more weight on adverse outcomes under that distribution. It does not allow the distribution itself to change. In a truncated equilibrium, errors in the distribution can also change whether a path's generalized cost is below its endogenous OD reservation cost.

We now extend Model B by allowing the joint distribution of network states to vary within a $1$-Wasserstein ambiguity set centered on the empirical distribution. The model keeps its common-state-weighted equilibrium structure. At equilibrium, a selected worst-case distribution and mean--CVaR tail weighting under that distribution give one common risk-adjusted distribution for all routes. This is a robust extension of Model B, not a path-specific worst-case version of Model A. The physical flow constraints and truncated-choice regularizer remain unchanged.

\subsection{Ambiguity set and standardized ground metric}
Let $P^\star$ be the unknown true law of the joint uncertainty vector $\xi\in\Xi\subset\R^m$, and let
$
\widehat P_N:=\frac{1}{N}\sum_{i=1}^N\delta_{\xi^i}
$
be the empirical law of independent training samples. For later use, an expectation under a general ambiguity-set distribution $Q$ and under the empirical distribution expand as
\begin{equation}
\begin{aligned}
\E_Q[h(\xi)] 
=\int_{\Xi}h(\zeta)\,Q(\mathrm d\zeta), \qquad
\E_{\widehat P_N}[h(\xi)]
=\int_{\Xi}h(\zeta)\,\widehat P_N(\mathrm d\zeta) =\int_{\Xi}h(\zeta)\left(\frac1N\sum_{i=1}^N\delta_{\xi^i}\right)(\mathrm d\zeta) = \frac1N\sum_{i=1}^Nh(\xi^i).
\end{aligned}
\label{eq:dro_expectation_expansions}
\end{equation}
If a candidate $Q$ is itself discrete, $Q=\sum_{j=1}^{M_Q}q_j\delta_{\zeta^j}$, then the first line further becomes $\E_Q[h(\xi)]
=\int_{\Xi}h(\zeta)\,Q(\mathrm d\zeta)
=\sum_{j=1}^{M_Q}q_jh(\zeta^j)$.

To avoid mixing incomparable physical units, let $D_\xi=\operatorname{diag}(d_1,\ldots,d_m)$ with $d_j>0$ and define
$
d_\xi(\xi,\zeta):=\left\|D_\xi^{-1}(\xi-\zeta)\right\|_2.
$
The scaling factors can be physical tolerances, observed standard deviations, or prespecified engineering ranges. The resulting radius is dimensionless and is invariant to a mere change of measurement units when $D_\xi$ is transformed consistently. For the finite-sample coverage interpretation below, $D_\xi$ is treated as fixed; if it is estimated from the same training sample, its estimation error must be incorporated into the calibration of $\rho$.

The corresponding $1$-Wasserstein distance is
$
W_1^{d_\xi}(Q,\widehat P_N)
:=\inf_{\Gamma\in\Pi(Q,\widehat P_N)}
\int_{\Xi\times\Xi}d_\xi(\xi,\zeta)\,\Gamma(\mathrm d\xi,\mathrm d\zeta),
$
and the ambiguity set is
$
\mathcal D_N(\rho)
:=\left\{Q\in\mathcal P(\Xi):W_1^{d_\xi}(Q,\widehat P_N)\le\rho\right\}.
$
We assume $P^\star(\Xi)=1$ whenever a statistical coverage statement is invoked.

\subsection{Robust risk functional and unique equilibrium}
The robust aggregate risk functional is
\begin{equation}
\begin{aligned}
\mathcal R_{\alpha,\eta}^{\mathrm{DRO}}(f)
&:=\sup_{Q\in\mathcal D_N(\rho)}
\left\{(1-\eta)\E_Q[L(f,\xi)]
+\eta\CVaR_{\alpha,Q}(L(f,\xi))\right\}.
\end{aligned}
\label{eq:dro_risk_new}
\end{equation}
For a discrete candidate law $Q=\sum_{j=1}^{M_Q}q_j\delta_{\zeta^j}$, its mean term is
$
\E_Q[L(f,\xi)]
=\int_{\Xi}L(f,\zeta)\,Q(\mathrm d\zeta)
=\sum_{j=1}^{M_Q}q_jL(f,\zeta^j).
$
The distributionally robust TSUE is
\begin{equation}
\min_{(f,x)\in\mathcal F_{fx}}
\left\{
\mathcal R_{\alpha,\eta}^{\mathrm{DRO}}(f)
+\frac{1}{\theta}\Psi_{r_0}(f)
\right\}.
\label{eq:dro_main_system}
\end{equation}
This is a robust extension of Model B. It is not a path-by-path worst-case CVaR model.  Distributional ambiguity changes the risk functional but leaves the physical assignment constraints unchanged: all OD demand is conserved, link flow is induced by path flow, and both flow vectors remain nonnegative.

\begin{assumption}[Compact-support regularity]
\label{ass:dro_compact_new}
The feasible set $\mathcal F$ in \eqref{eq:feasible_set} is nonempty and therefore compact and convex; $\Xi$ is nonempty and compact; every empirical sample satisfies $\xi^i\in\Xi$; $d_\xi$ is a continuous metric on $\Xi$; the Wasserstein radius satisfies $\rho>0$; and $L:\mathcal F\times\Xi\to\R$ is jointly continuous and convex in $f$ for every $\xi$.
\end{assumption}
The restriction $\rho>0$ guarantees attainment of the scalar Wasserstein dual used below. The case $\rho=0$ reduces to the empirical Model~B objective and can be handled directly.

\begin{proposition}
\label{prop:dro_unique_new}
Under Assumption~\ref{ass:dro_compact_new}, the robust risk functional is finite, continuous, and convex on $\mathcal F$. Problem~\eqref{eq:dro_main_system} has a unique optimal path-flow vector and a unique associated link-flow vector.
\end{proposition}
\noindent\textit{Proof and provenance.} Appendix~\ref{app:proof_dro_unique} proves finiteness, uniform continuity, convexity, existence, and uniqueness; the continuity argument is consistent with Berge's maximum theorem.

\subsection{Exact minimax and Wasserstein dual reformulation}
Define global bounds and the compact threshold interval
$
\underline \ell:=\min_{(f,x)\in\mathcal F_{fx},\,\xi\in\Xi}L(f,\xi),
$ and $
\overline \ell:=\max_{(f,x)\in\mathcal F_{fx},\,\xi\in\Xi}L(f,\xi),
\quad
\mathcal T:=[\underline \ell,\overline \ell].
$
For $t\in\mathcal T$, let
$
g_t(f,\xi)
:=(1-\eta)L(f,\xi)
+\frac{\eta}{1-\alpha}\bigl(L(f,\xi)-t\bigr)_+.
$

\begin{proposition}
\label{prop:dro_minimax_new}
Under Assumption~\ref{ass:dro_compact_new}, for every $f\in\mathcal F$,
\begin{equation}
\begin{aligned}
\mathcal R_{\alpha,\eta}^{\mathrm{DRO}}(f)
=\min_{t\in\mathcal T}
\left\{\eta t+\sup_{Q\in\mathcal D_N(\rho)}\E_Q[g_t(f,\xi)]\right\}
=\min_{t\in\mathcal T}
\left\{\eta t+\sup_{Q\in\mathcal D_N(\rho)}
\int_{\Xi}g_t(f,\zeta)\,Q(\mathrm d\zeta)\right\}.
\end{aligned}
\label{eq:dro_minimax_new}
\end{equation}
If the optimizing distribution is discrete, $Q=\sum_{j=1}^{M_Q}q_j\delta_{\zeta^j}$, the inner expectation further expands as
$
\E_Q[g_t(f,\xi)]
=\int_{\Xi}g_t(f,\zeta)\,Q(\mathrm d\zeta)
=\sum_{j=1}^{M_Q}q_jg_t(f,\zeta^j).
$
The function on the right is jointly convex in $(f,t)$ after taking the supremum over $Q$.
\end{proposition}
\noindent\textit{Proof and provenance.} Appendix~\ref{app:proof_dro_minimax} verifies every hypothesis of Sion's minimax theorem and proves the threshold restriction.

\begin{proposition}
\label{prop:dro_dual_new}
Under Assumption~\ref{ass:dro_compact_new}, problem~\eqref{eq:dro_main_system} is equivalent to the following semi-infinite program. Let $s(\xi^i):=s_i$; its empirical expectation is the uniform-weight specialization of \eqref{eq:finite_expectation}:
\begin{subequations}
\label{eq:dro_sip_new}
\begin{align}
\min_{(f,x)\in\mathcal F_{fx},\,t,\kappa,s}\quad
\mathcal J_{\mathrm{SIP}}(f,t,\kappa,s):={}&\eta t+\kappa\rho+\E_{\widehat P_N}[s(\xi)]
+\frac{1}{\theta}\Psi_{r_0}(f)\nonumber
={}\eta t+\kappa\rho+\int_{\Xi}s(\zeta)\,\widehat P_N(\mathrm d\zeta)
+\frac{1}{\theta}\Psi_{r_0}(f)\nonumber\\
\text{s.t.}\quad
&s_i\ge g_t(f,\xi)-\kappa d_\xi(\xi,\xi^i),
\quad \forall \xi\in\Xi,\ i=1,\ldots,N,
\label{eq:dro_sip_con_new}\\
&t\in\mathcal T,\quad \eqref{eq:demand_conservation_new} - \eqref{eq:nonnegativity_new}, 
\quad \kappa\ge0.
\label{eq:dro_sip_other_domains}
\end{align}
\end{subequations}
\end{proposition}
\noindent\textit{Proof and provenance.} Appendix~\ref{app:proof_dro_dual} applies Wasserstein strong duality from \citet{gao_kleywegt_2022} and proves attainment under $\rho>0$.

Compactness is used for the exact minimax interchange and attainment. More general unbounded-support versions require separate growth and integrability conditions and are not claimed here.

\subsection{Robust equilibrium interpretation}
The robust objective remains an aggregate potential model, but its KKT system retains the truncated kernel.

\begin{assumption}
\label{ass:dro_gradient_new}
For every $\xi\in\Xi$, $L(\cdot,\xi)$ is continuously differentiable on a relative open neighborhood of $\mathcal F$ in the nonnegative flow orthant, with derivatives that extend continuously to the boundary; $\nabla_fL(f,\xi)$ is jointly continuous on $\mathcal F\times\Xi$; and
$
\frac{\partial L(f,\xi)}{\partial f_k^{od}}=T_k^{od}(f,\xi)
\quad \forall (o,d),\ k\in K_{od}.
$
\end{assumption}
\begin{proposition}
\label{prop:dro_kkt_new}
Under Assumptions~\ref{ass:dro_compact_new} and~\ref{ass:dro_gradient_new}, at the unique flow pair $(f^\star,x^\star)$ of \eqref{eq:dro_main_system} there exist a robust-risk subgradient $g^{R}(f^\star)\in\partial\mathcal R_{\alpha,\eta}^{\mathrm{DRO}}(f^\star)$ and OD multipliers $\pi_{od}^{R}$ such that
$
0\le f_k^{od,\star}
\ \perp\ 
 g_k^{R,od}(f^\star)
+\frac{1}{\theta}\psi_{r_0}'(f_k^{od,\star})
-\pi_{od}^{R}\ge0.
$
Equivalently, \eqref{eq:generic_truncated_flow}--\eqref{eq:generic_pi_balance} hold with $c_k^{od}=g_k^{R,od}(f^\star)$ and $\pi_{od}=\pi_{od}^{R}$. Hence each multiplier is unique for the selected robust marginal costs, and the optimizer satisfies \eqref{eq:flow_feasibility}.
\end{proposition}
\noindent\textit{Proof and provenance.} Appendix~\ref{app:proof_dro_kkt} gives the convex subgradient/KKT proof; Appendix~\ref{app:proof_dro_tilted_law} constructs an explicit worst-case selected subgradient.

\begin{proposition}
\label{prop:dro_tilted_law_representation}
Under Assumptions~\ref{ass:dro_compact_new} and
\ref{ass:dro_gradient_new}, let $(f^\star,x^\star)$ be the unique
solution of \eqref{eq:dro_main_system}. There exist a worst-case distribution $Q^\star\in\mathcal D_N(\rho)$ for \eqref{eq:dro_risk_new} at $f^\star$, a robust-risk subgradient $g^R(f^\star)$, and OD multipliers $\pi_{od}^R$ that satisfy the
KKT conditions in Proposition~\ref{prop:dro_kkt_new} and admit the following representation.

If $\eta>0$, there exist an RU-optimal threshold
$t^\star\in\mathcal T$ for $L(f^\star,\xi)$ under $Q^\star$ and a measurable function $b^\star:\Xi\to[0,1]$ such that
\begin{equation}
Q^\star\!\left\{L(f^\star,\xi)<t^\star\right\}
\le\alpha
\le
Q^\star\!\left\{L(f^\star,\xi)\le t^\star\right\}.
\label{eq:robust_quantile_condition}
\end{equation}
The selected CVaR selector is
\begin{equation}
\zeta^\star(\xi)
=
\mathbf 1_{\{L(f^\star,\xi)>t^\star\}}
+
b^\star(\xi)\,
\mathbf 1_{\{L(f^\star,\xi)=t^\star\}},
\qquad Q^\star\text{-almost surely},
\label{eq:robust_selector_main}
\end{equation}
where the boundary allocation satisfies
\begin{equation}
\begin{aligned}
\E_{Q^\star}[\zeta^\star]
&=
Q^\star\!\left\{L(f^\star,\xi)>t^\star\right\}
+
\int_{\{L(f^\star,\xi)=t^\star\}}
b^\star(\xi)\,Q^\star(\mathrm d\xi) =1-\alpha.
\end{aligned}
\label{eq:robust_selector_mass_main}
\end{equation}
The distribution $Q^\star$ and selector $\zeta^\star$ are selected jointly so that the resulting marginal costs satisfy the equilibrium KKT conditions. In particular, $b^\star(\xi)$ is not required to be constant across states tied at the VaR threshold.

For $\eta>0$, define $\widetilde Q^R$ by
\begin{equation}
\frac{\mathrm d\widetilde Q^R}{\mathrm dQ^\star}(\xi)
=
(1-\eta)+\frac{\eta}{1-\alpha}\zeta^\star(\xi).
\label{eq:robust_tilted_measure}
\end{equation}
If $\eta=0$, the threshold and selector are immaterial, and we set $\widetilde Q^R=Q^\star$. In both cases, $\widetilde Q^R$ is a probability measure, and the selected robust marginal path costs satisfy
\begin{equation}
\begin{aligned}
g_k^{R,od}(f^\star)
=
\E_{\widetilde Q^R}
\!\left[T_k^{od}(f^\star,\xi)\right]
=
\int_{\Xi}T_k^{od}(f^\star,\xi)\,
\widetilde Q^R(\mathrm d\xi),
\qquad
(o,d)\in\mathcal O,\quad k\in K_{od}.
\end{aligned}
\label{eq:robust_tilted_law_cost}
\end{equation}
For $\eta>0$ and a discrete worst-case law
$Q^\star=\sum_{j=1}^{M_\star}q_j^\star\delta_{\zeta^j}$,
this becomes
$$
g_k^{R,od}(f^\star)
=
\sum_{j=1}^{M_\star}
q_j^\star
\left[
(1-\eta)+\frac{\eta}{1-\alpha}\zeta^\star(\zeta^j)
\right]
T_k^{od}(f^\star,\zeta^j).
$$
For $\eta=0$, the weights reduce to $q_j^\star$.

The same selected law is used for every path and OD pair. Thus, the KKT system in Proposition~\ref{prop:dro_kkt_new} is a truncated SUE under a common tail-tilted law generated
from a worst-case network-state distribution. The generic relations \eqref{eq:generic_truncated_flow}--
\eqref{eq:generic_pi_balance} apply with
$c_k^{od}=g_k^{R,od}(f^\star)$.
For these selected marginal costs, each OD reservation-cost
multiplier is uniquely determined by demand balance.
Neither the worst-case law nor its boundary allocation is
claimed to be unique.
\end{proposition}

\noindent\textit{Proof and provenance.}
Appendix~\ref{app:proof_dro_tilted_law} uses the CVaR risk envelope in \eqref{eq:cvar_risk_envelope} to select a worst-case distribution and tail allocation jointly with the equilibrium optimality conditions. It then establishes the boundary-selector representation and the common tilted law.

The two layers in \eqref{eq:robust_tilted_measure} have different meanings. The move from $\widehat P_N$ to $Q^\star$ represents epistemic caution about the physical law. The move from $Q^\star$ to $\widetilde Q^{R}$ represents risk aversion conditional on that adverse law. The first transformation is controlled by the Wasserstein radius $\rho$; the second by $(\alpha,\eta)$. Both transformations are common across paths, so travelers still compare alternative-specific travel times under one shared scenario law. This commonality is what preserves the potential formulation and distinguishes the model from a collection of independently robust shortest-path problems.

\subsection{Finite-sample coverage and decision certificate}
Under the light-tail conditions of \citet{mohajerin_esfahani_kuhn_2017}, applied to the standardized vector $D_\xi^{-1}\xi$, there are distribution-dependent constants for which a radius $\rho_{N,\delta}$ can be selected so that
$
\Prob^N\left(P^\star\in\mathcal D_N(\rho_{N,\delta})\right)\ge1-\delta.
$
For small radii, the generic exponent depends on the uncertainty dimension, with the familiar $N^{-1/\max\{m,2\}}$ behavior up to the special logarithmic correction in dimension two. Thus, high-dimensional link-by-link uncertainty can require prohibitively many observations; low-dimensional hazard factors or grouped link parameters are preferable when possible.

Let $(f_N^\star,x_N^\star)$ solve \eqref{eq:dro_main_system} with a radius satisfying $\Prob^N\left(P^\star\in\mathcal D_N(\rho_{N,\delta})\right)\ge1-\delta$. On the event $P^\star\in\mathcal D_N(\rho_{N,\delta})$,
$
(1-\eta)\E_{P^\star}[L(f_N^\star,\xi)]
+\eta\CVaR_{\alpha,P^\star}(L(f_N^\star,\xi))
+\frac{1}{\theta}\Psi_{r_0}(f_N^\star)\nonumber \le
\mathcal R_{\alpha,\eta}^{\mathrm{DRO}}(f_N^\star)
+\frac{1}{\theta}\Psi_{r_0}(f_N^\star).
$

Hence the robust optimal value is a high-confidence upper certificate for the true-distribution objective at the selected decision. The concentration constants are generally unknown, so validation or bootstrap calibration of $\rho$ remains necessary in applications.

\subsection{Finite and conic special cases}
If $\Xi=\{\zeta^1,\ldots,\zeta^M\}$ is a finite scenario library, the semi-infinite constraints reduce to
$
s_i\ge g_t(f,\zeta^j)-\kappa d_\xi(\zeta^j,\xi^i),
\quad i=1,\ldots,N,\ j=1,\ldots,M,
$
which produces a finite convex program together with the explicit demand-conservation, link-flow, and nonnegativity constraints in \eqref{eq:flow_feasibility}.

For a continuous conic special case, work in standardized coordinates $y=D_\xi^{-1}\xi$, with samples $y^i=D_\xi^{-1}\xi^i$, and suppose
$
\widetilde\Xi:=\{y:\|y-y^c\|_2\le R\},
$ and $
L(f,y)=L_0(f)+h(f)^\top y,$
where $h(f)$ is affine, $L_0$ has an SOCP-representable epigraph, and the full flow-pair set $\mathcal F_{fx}$ in \eqref{eq:feasible_set} is SOCP representable (in fact polyhedral for fixed candidate paths). Then $g_t$ is the maximum of two affine functions of $y$:
\begin{align}
a_0(f)&=(1-\eta)h(f),
&b_0(f,t)&=(1-\eta)L_0(f),\\
a_1(f)&=\left(1-\eta+\frac{\eta}{1-\alpha}\right)h(f),
&b_1(f,t)&=\left(1-\eta+\frac{\eta}{1-\alpha}\right)L_0(f)-\frac{\eta}{1-\alpha}t.
\end{align}
For $j\in\{0,1\}$,
\begin{equation}
\begin{aligned}
&\sup_{y\in\widetilde\Xi}
\left\{a_j(f)^\top y+b_j(f,t)-\kappa\|y-y^i\|_2\right\}
\nonumber =
\min_{\|u_{ij}\|_2\le\kappa}
\left\{b_j(f,t)+u_{ij}^\top y^i
+(a_j(f)-u_{ij})^\top y^c
+R\|a_j(f)-u_{ij}\|_2\right\}.
\label{eq:conic_support_identity_new}
\end{aligned}
\end{equation}

Introducing standard epigraph variables gives exact second-order-cone constraints for the Wasserstein support terms. Define the closed exponential cone by
\begin{equation}
\mathcal K_{\exp}:=\operatorname{cl}\left\{(x,y,z):y>0,\ y\exp(x/y)\le z\right\}.
\end{equation}
For each path, an epigraph variable $e_k^{od}$ satisfies
\begin{equation}
(-e_k^{od},\ f_k^{od}+r_0,\ r_0)\in\mathcal K_{\exp}
\quad\Longleftrightarrow\quad
 e_k^{od}\ge(f_k^{od}+r_0)\ln\left(\frac{f_k^{od}+r_0}{r_0}\right).
\label{eq:entropy_cone_new}
\end{equation}
Replacing $\Psi_{r_0}(f)$ in the objective by
$\sum_{(o,d)\in\mathcal O}\sum_{k\in K_{od}}(e_k^{od}-f_k^{od})$
therefore gives an exact mixed exponential-cone/SOCP representation. The conic formulation also retains the explicit demand-conservation, path-to-link flow, and nonnegativity constraints in \eqref{eq:flow_feasibility}. It is not a pure SOCP. General BPR uncertainty need not satisfy the affine assumption in $\widetilde\Xi:=\{y:\|y-y^c\|_2\le R\},
$ and $
L(f,y)=L_0(f)+h(f)^\top y$.

\subsection{Exact exchange algorithm}
\label{sec:exchange}
For general compact support, we solve \eqref{eq:dro_sip_new} by an exchange method. At iteration $r$, a restricted main problem contains finite support subsets $\mathcal X_i^{(r)}\subset\Xi$. After solving the restricted problem, exact separation for sample $i$ computes
\begin{equation}
\widehat s_i^{(r)}
:=\sup_{\xi\in\Xi}
\left\{g_{t^{(r)}}(f^{(r)},\xi)-\kappa^{(r)}d_\xi(\xi,\xi^i)\right\}.
\label{eq:exact_separation_new}
\end{equation}
Define the samplewise separation violation by
$
\Delta_i^{(r)}:=\widehat s_i^{(r)}-s_i^{(r)},
\quad i=1,\ldots,N.$
A maximizer is added whenever $\Delta_i^{(r)}>\varepsilon_{\mathrm{sep}}$. Exact or certified global separation is required for a rigorous upper bound. Heuristic search may generate useful cuts, but it cannot certify feasibility or optimality unless accompanied by a valid global upper bound on \eqref{eq:exact_separation_new}.

Let $\mathrm{LB}_r$ be the optimal objective value of the restricted main problem at iteration $r$. Replacing each restricted epigraph value $s_i^{(r)}$ by the exact support value $\widehat s_i^{(r)}$ gives the full-program feasible candidate value
\begin{equation}
\begin{aligned}
\mathrm{UB}_r
:=\eta t^{(r)}+\kappa^{(r)}\rho
+\E_{\widehat P_N}[\widehat s^{(r)}(\xi)]
+\frac1\theta\Psi_{r_0}(f^{(r)}) =\eta t^{(r)}+\kappa^{(r)}\rho
+\frac1N\sum_{i=1}^{N}\widehat s_i^{(r)}
+\frac1\theta\Psi_{r_0}(f^{(r)}).
\end{aligned}
\label{eq:exchange_ub_main}
\end{equation}

\begin{proposition}
\label{prop:exchange_certificate}
Assume the restricted main problem and every separation problem are solved globally. Then
$
\mathrm{LB}_r\le V^\star\le\mathrm{UB}_r,
$
where $V^\star$ is the optimal value of \eqref{eq:dro_sip_new}. If $\max_i\Delta_i^{(r)}\le\varepsilon_{\mathrm{sep}}$, then
$
0\le\mathrm{UB}_r-\mathrm{LB}_r
=\E_{\widehat P_N}[\Delta^{(r)}(\xi)]
=\frac1N\sum_{i=1}^{N}\Delta_i^{(r)}
\le\varepsilon_{\mathrm{sep}}.
$
For $\varepsilon_{\mathrm{sep}}=0$, the returned solution is globally optimal.
\end{proposition}

\begin{proposition}
\label{prop:exchange_convergence_new}
Assume $\rho>0$, $\xi^i\in\mathcal X_i^{(0)}$ for every $i$, $\Xi$ is compact, all functions in \eqref{eq:dro_sip_con_new} are continuous, all generated cuts are retained, and every restricted main and separation problem is solved globally. With zero separation tolerance, the restricted-main values form a nondecreasing sequence of lower bounds, the master-solution sequence is bounded, and every accumulation point is optimal for \eqref{eq:dro_sip_new}. If $\Xi$ is finite, zero-tolerance exact separation terminates after finitely many distinct cuts. At a positive tolerance, certified termination is governed by Proposition~\ref{prop:exchange_certificate}.
\end{proposition}

Algorithm~\ref{alg:exchange_new} and the proofs of Propositions~\ref{prop:exchange_certificate} and~\ref{prop:exchange_convergence_new} are provided in Appendix~\ref{app:exchange_algorithm}. Safe accelerations include parallel separation, shared hinge epigraphs, warm starts, and periodic exact full audits. Proximal stabilization must be implemented in a level or bundle framework that preserves the lower-bound interpretation of the original restricted problem.

\section{Braess network numerical experiments}
\label{sec:braess_numerics}

We then test all the proposed mechanisms on the Braess network. The Braess network has nodes $\{A,B,C,D\}$ and links $A\to B$, $A\to C$,
$B\to C$, $B\to D$, and $C\to D$ as shown in the Figure \ref{fig:braess_tikz}. The single OD pair is $A\to D$ with demand $q=80$, and the retained paths are
$$
P_1=A\to B\to D,\qquad
P_2=A\to C\to D,\qquad
P_3=A\to B\to C\to D.
$$

Link flows follow \eqref{eq:link_flow_reduced}; scenario-dependent link times instantiate \eqref{eq:bpr_new} with $r_0=1$ and $\theta=0.60$.

\begin{figure}[htbp]
\centering
\begin{tikzpicture}[
    >=Latex,
    node/.style={circle, draw, minimum size=8mm, inner sep=0pt, font=\small},
    link/.style={->, thick},
    lab/.style={font=\scriptsize, fill=white, inner sep=1pt},
    pathlab/.style={font=\scriptsize, align=left}
]
\node[node] (A) at (0,0) {$A$};
\node[node] (B) at (2.7,1.35) {$B$};
\node[node] (C) at (2.7,-1.35) {$C$};
\node[node] (D) at (5.4,0) {$D$};
\draw[link] (A) -- node[lab, above left] {$a_1:A\!\to\!B$} (B);
\draw[link] (A) -- node[lab, below left] {$a_2:A\!\to\!C$} (C);
\draw[link] (B) -- node[lab, right] {$a_3:B\!\to\!C$} (C);
\draw[link] (B) -- node[lab, above right] {$a_4:B\!\to\!D$} (D);
\draw[link] (C) -- node[lab, below right] {$a_5:C\!\to\!D$} (D);
\node[pathlab, anchor=west] at (5.95,0.75)
{$P_1=A\to B\to D$\\
 $P_2=A\to C\to D$\\
 $P_3=A\to B\to C\to D$};
\node[font=\scriptsize, anchor=west] at (0,-2.25) {OD demand: $q_{AD}=80$};
\end{tikzpicture}
\caption{Braess network used for mechanism validation.}
\label{fig:braess_tikz}
\end{figure}

\begin{table}[htbp]
\centering
\scriptsize
\setlength{\tabcolsep}{3pt}
\renewcommand{\arraystretch}{1.10}
\caption{Route-local Braess scenario probabilities and explicit link travel time functions.}
\label{tab:braess_scenario_time_functions}
\begin{tabular}{lllp{0.60\linewidth}}
\toprule
State & $p_s$ & Link & Explicit travel-time function\\
\midrule
\multirow{5}{*}{Normal} & \multirow{5}{*}{0.55} & $a_1:A\to B$ & $t_{AB}^s(x_{AB})=3.00\left[1+0.15\left(\frac{x_{AB}}{130.00}\right)^4\right]$\\
 &  & $a_2:A\to C$ & $t_{AC}^s(x_{AC})=7.50\left[1+0.15\left(\frac{x_{AC}}{85.00}\right)^4\right]$\\
 &  & $a_3:B\to C$ & $t_{BC}^s(x_{BC})=0.80\left[1+0.15\left(\frac{x_{BC}}{120.00}\right)^4\right]$\\
 &  & $a_4:B\to D$ & $t_{BD}^s(x_{BD})=14.00\left[1+0.15\left(\frac{x_{BD}}{65.00}\right)^4\right]$\\
 &  & $a_5:C\to D$ & $t_{CD}^s(x_{CD})=8.00\left[1+0.15\left(\frac{x_{CD}}{110.00}\right)^4\right]$\\
\addlinespace[0.25em]
\multirow{5}{*}{Upper incident} & \multirow{5}{*}{0.18} & $a_1:A\to B$ & $t_{AB}^s(x_{AB})=3.24\left[1+0.15\left(\frac{x_{AB}}{101.40}\right)^4\right]+4.00$\\
 &  & $a_2:A\to C$ & $t_{AC}^s(x_{AC})=7.65\left[1+0.15\left(\frac{x_{AC}}{81.60}\right)^4\right]+0.30$\\
 &  & $a_3:B\to C$ & $t_{BC}^s(x_{BC})=0.82\left[1+0.15\left(\frac{x_{BC}}{115.20}\right)^4\right]+0.30$\\
 &  & $a_4:B\to D$ & $t_{BD}^s(x_{BD})=16.52\left[1+0.15\left(\frac{x_{BD}}{35.75}\right)^4\right]+16.00$\\
 &  & $a_5:C\to D$ & $t_{CD}^s(x_{CD})=8.16\left[1+0.15\left(\frac{x_{CD}}{105.60}\right)^4\right]+0.30$\\
\addlinespace[0.25em]
\multirow{5}{*}{Lower incident} & \multirow{5}{*}{0.17} & $a_1:A\to B$ & $t_{AB}^s(x_{AB})=3.06\left[1+0.15\left(\frac{x_{AB}}{124.80}\right)^4\right]+0.30$\\
 &  & $a_2:A\to C$ & $t_{AC}^s(x_{AC})=8.55\left[1+0.15\left(\frac{x_{AC}}{52.70}\right)^4\right]+10.00$\\
 &  & $a_3:B\to C$ & $t_{BC}^s(x_{BC})=0.82\left[1+0.15\left(\frac{x_{BC}}{115.20}\right)^4\right]+0.30$\\
 &  & $a_4:B\to D$ & $t_{BD}^s(x_{BD})=14.28\left[1+0.15\left(\frac{x_{BD}}{62.40}\right)^4\right]+0.30$\\
 &  & $a_5:C\to D$ & $t_{CD}^s(x_{CD})=8.64\left[1+0.15\left(\frac{x_{CD}}{82.50}\right)^4\right]+4.00$\\
\addlinespace[0.25em]
\multirow{5}{*}{Shortcut incident} & \multirow{5}{*}{0.06} & $a_1:A\to B$ & $t_{AB}^s(x_{AB})=3.03\left[1+0.15\left(\frac{x_{AB}}{127.40}\right)^4\right]+0.20$\\
 &  & $a_2:A\to C$ & $t_{AC}^s(x_{AC})=7.58\left[1+0.15\left(\frac{x_{AC}}{83.30}\right)^4\right]+0.20$\\
 &  & $a_3:B\to C$ & $t_{BC}^s(x_{BC})=1.00\left[1+0.15\left(\frac{x_{BC}}{42.00}\right)^4\right]+24.00$\\
 &  & $a_4:B\to D$ & $t_{BD}^s(x_{BD})=14.14\left[1+0.15\left(\frac{x_{BD}}{63.70}\right)^4\right]+0.20$\\
 &  & $a_5:C\to D$ & $t_{CD}^s(x_{CD})=8.08\left[1+0.15\left(\frac{x_{CD}}{107.80}\right)^4\right]+0.20$\\
\addlinespace[0.25em]
\multirow{5}{*}{System-wide} & \multirow{5}{*}{0.04} & $a_1:A\to B$ & $t_{AB}^s(x_{AB})=3.66\left[1+0.15\left(\frac{x_{AB}}{75.40}\right)^4\right]+10.00$\\
 &  & $a_2:A\to C$ & $t_{AC}^s(x_{AC})=9.15\left[1+0.15\left(\frac{x_{AC}}{49.30}\right)^4\right]+10.00$\\
 &  & $a_3:B\to C$ & $t_{BC}^s(x_{BC})=0.98\left[1+0.15\left(\frac{x_{BC}}{69.60}\right)^4\right]+10.00$\\
 &  & $a_4:B\to D$ & $t_{BD}^s(x_{BD})=17.08\left[1+0.15\left(\frac{x_{BD}}{37.70}\right)^4\right]+10.00$\\
 &  & $a_5:C\to D$ & $t_{CD}^s(x_{CD})=9.76\left[1+0.15\left(\frac{x_{CD}}{63.80}\right)^4\right]+10.00$\\
\bottomrule
\end{tabular}
\end{table}

Table~\ref{tab:braess_scenario_time_functions} gives the link travel time functions for each scenario. The normal state is the baseline, with nominal free-flow times and capacities and no additive incident delay. The upper incident disrupts the upper corridor by increasing free flow times, reducing effective capacities, and adding delays on links $A\to B$ and $B\to D$. This mainly increases the cost of $P_1$ but also affects $P_3$ through $A\to B$. The lower incident causes similar changes on links $A\to C$ and $C\to D$, mainly increasing the cost of $P_2$ but also affecting $P_3$ through $C\to D$. The shortcut incident disrupts link $B\to C$ by sharply reducing its effective capacity and adding a large delay. Path $P_3$ is attractive under normal conditions but faces a risk of severe delays when the shortcut is disrupted. The system-wide state increases free flow times, reduces capacities, and adds delays on all links, creating a common severe state for all three paths.

\subsection{Risk sensitivity under the direct $(\alpha,\eta)$ grid}
\label{subsec:braess_direct_eta}

We first examine the tail-selection mechanism on the route-local Braess instance. Figure~\ref{fig:braess_tail_tilt} reports the pure CVaR tail law $p_s\chi_s/(1-\alpha)$ for $\alpha\in\{0.80,0.82,\ldots,0.98\}$,
evaluated at the Model B equilibrium flow with $\eta=1.0$. Model A computes the tail selector separately for each path-time random variable in \eqref{eq:modelA_cost}. Model B computes a common selector from the systemic congestion potential in \eqref{eq:modelB_risk} -\eqref{eq:modelB_tilted_law}. The figure thus shows which scenarios enter the exact upper tail, and with what fractional tail mass, under each model's law.

\begin{figure}[htbp]
\centering
\includegraphics[width=0.85\textwidth]{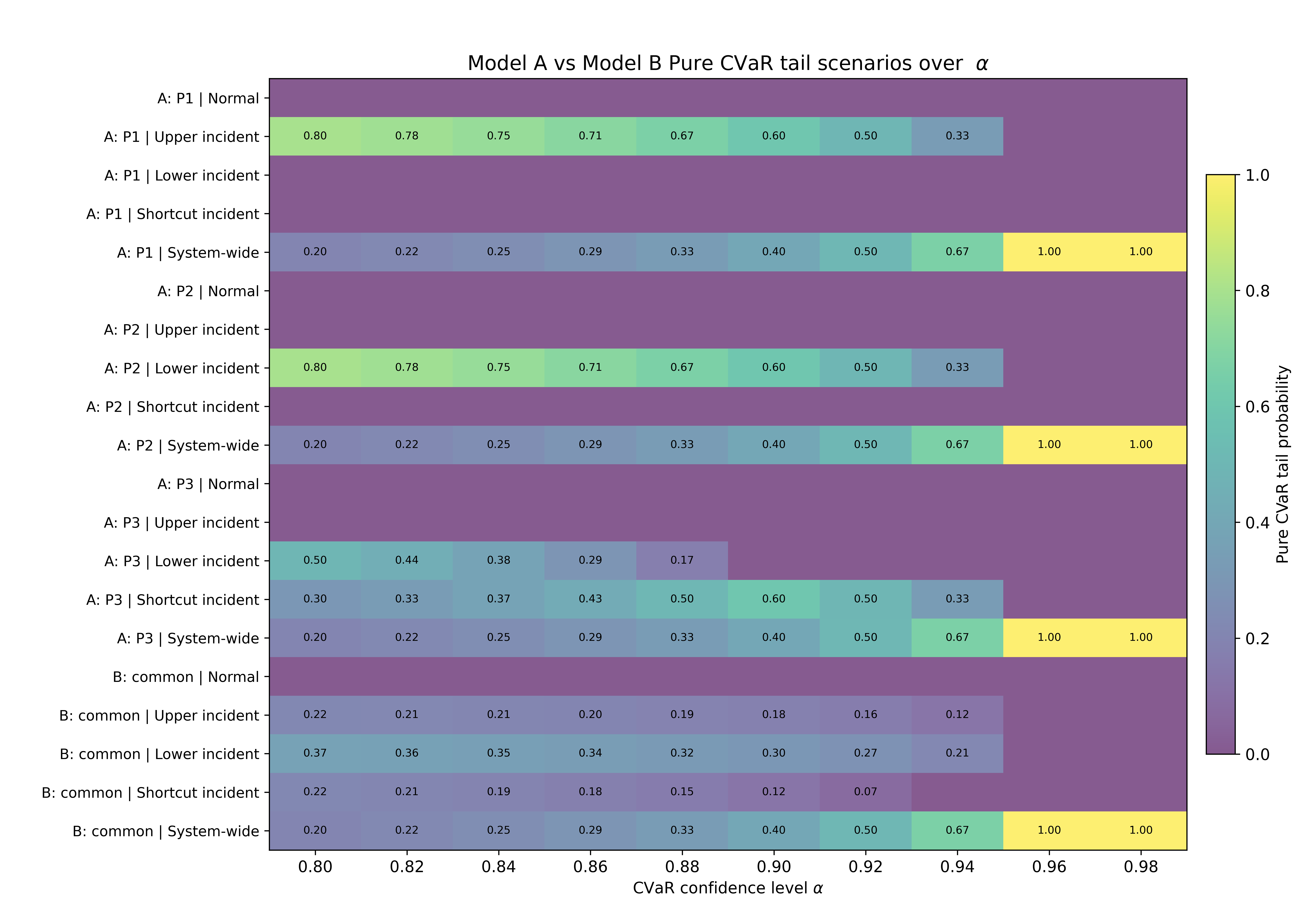}
\caption{Nominal and tail-tilted scenario probabilities on the route-local
Braess instance. Model A generates path-specific tilted laws, whereas Model B uses one
common systemic tail-tilted law.}
\label{fig:braess_tail_tilt}
\end{figure}

Figure~\ref{fig:braess_tail_tilt} confirms that route-local disruptions produce different upper-tail laws across paths. At $\alpha=0.80$, Model A selects route-specific adverse states. The tail of $P_1$ puts mass $0.80$ on the upper incident and $0.20$ on the system-wide state. The tail of $P_2$ puts mass $0.80$ on the lower incident and $0.20$ on the system-wide state. The tail of $P_3$ puts masses $0.50$, $0.30$, and $0.20$ on the lower incident, shortcut incident, and system-wide state, respectively. By contrast, Model B uses one common systemic tail law for all paths. At $\alpha=0.80$, this law assigns mass $0.37$ to the lower incident and $0.20$ to the system-wide state.

As $\alpha$ increases, the selected upper tail narrows. For $\alpha<0.96$, the pure CVaR tail laws still differ across Model A paths and between the two models. This is because the upper, lower, and shortcut incidents are route-local rather than common severe states. At $\alpha\ge0.96$, the upper tail is exhausted by the system-wide state, whose nominal probability is $0.04=1-0.96$. All Model A path tails and the Model B common tail then collapse to the same scenario, with unit mass on the system-wide. 

We next solve Model A and Model B on the route-local Braess instance over the
retained direct grid $\alpha\in\{0.80,0.84,0.88,0.92,0.96\}$, and $\eta\in\{0,0.2,0.4,0.6,0.8,1.0\}$.

 For each $(\alpha,\eta)$, Table~\ref{tab:braess_modelB_direct_eta_grid_full} reports the normalized path shares
$(f_1^A/q,f_2^A/q,f_3^A/q)$ and $(f_1^B/q,f_2^B/q,f_3^B/q)$, together with the reservation generalized costs $\pi^A$ and $\pi^B$. The remaining columns quantify the Model A - Model B discrepancy. In Model B flow $f^B$, let $\chi_k^A(f^B)$ be a valid CVaR selector for the path-time variable $T_k(f^B,\xi)$, and let $\chi^B(f^B)$ be the common selector from the
systemic potential $L(f^B,\xi)$. The pure tail-law distance is
\begin{equation}
D_{\rm tail}^{AB}
=
\max_{k\in\{1,2,3\}}
\left\|
\frac{p\circ\chi_k^A(f^B)}{1-\alpha}
-
\frac{p\circ\chi^B(f^B)}{1-\alpha}
\right\|_1,
\end{equation}
where $\circ$ is componentwise multiplication. The effective tilted-law distance is
\begin{equation}
D_p^{AB}
=
\max_{k\in\{1,2,3\}}
\left\|
\widetilde p_k^A(f^B;\alpha,\eta)
-
\widetilde p^B(f^B;\alpha,\eta)
\right\|_1,
\end{equation}
with
\begin{equation}
\widetilde p_{k,s}^A
=
p_s\left[(1-\eta)+\frac{\eta}{1-\alpha}\chi_{k,s}^A(f^B)\right],
\qquad
\widetilde p_s^B
=
p_s\left[(1-\eta)+\frac{\eta}{1-\alpha}\chi_s^B(f^B)\right].
\end{equation}
The flow distance is
\begin{equation}
D_f^{AB}
=
\frac{\|f^A-f^B\|_1}{q}
=
\frac{|f_1^A-f_1^B|+|f_2^A-f_2^B|+|f_3^A-f_3^B|}{80}.
\end{equation}
The Model A VI residual at the Model B solution, which is used to answer the question: Does the Model B solution $f^B$ satisfy Model A's variational inequality is 
\begin{equation}
\mathrm{Gap}_A(f^B)
=
\frac{
\max_{y\in\mathcal F}\left(F^A(f^B)\right)^\top(f^B-y)
}{
1+\left|\left(F^A(f^B)\right)^\top f^B\right|
},
\qquad
F_k^A(f)=c_k^A(f)+\frac{1}{\theta}\ln\left(1+\frac{f_k}{r_0}\right).
\end{equation}
The active-set column uses
$
\mathcal A^M=\{k:f_k^M>10^{-6}\}, M\in\{A,B\},
$
and reports $|\mathcal A^A|-|\mathcal A^B|$. The table thus records how the two models redistribute flow and whether their endogenous truncated supports differ.

\begin{table}[htbp]
\centering
\tiny
\setlength{\tabcolsep}{1.35pt}
\renewcommand{\arraystretch}{1.05}
\caption{Direct-$(\alpha,\eta)$ grid with Model A--Model B flow, reservation cost, tail-law, and VI diagnostics}
\label{tab:braess_modelB_direct_eta_grid_full}
\resizebox{\textwidth}{!}{%
\begin{tabular}{lrrrrrrrrrrrrrrr}
\toprule
 &  &  & \multicolumn{3}{c}{Model A flow shares} & \multicolumn{3}{c}{Model B flow shares} & \multicolumn{2}{c}{Reservation} & \multicolumn{5}{c}{Model A--Model B diagnostics}\\
\cmidrule(lr){4-6}\cmidrule(lr){7-9}\cmidrule(lr){10-11}\cmidrule(lr){12-16}
Risk case & $\alpha$ & $\eta$ & $f_1^A/q$ & $f_2^A/q$ & $f_3^A/q$ & $f_1^B/q$ & $f_2^B/q$ & $f_3^B/q$ & $\pi^A$ & $\pi^B$ & $D_{\rm tail}^{AB}$ & $D_p^{AB}$ & $D_f^{AB}$ & $\mathrm{Gap}_A(f^B)$ & $|\mathcal A^A|-|\mathcal A^B|$\\
\midrule
Expectation-only & 0.80 & 0.0 & 0.024 & 0.147 & 0.829 & 0.024 & 0.147 & 0.829 & 24.10 & 24.10 & 1.600 & 0.000 & $2.3\times10^{-11}$ & 0 & 0\\
Mean--CVaR & 0.80 & 0.2 & 0.010 & 0.130 & 0.861 & 0.111 & 0.183 & 0.706 & 26.83 & 26.15 & 1.600 & 0.320 & 0.310 & 0.0211 & 0\\
Mean--CVaR & 0.80 & 0.4 & 0.001 & 0.114 & 0.885 & 0.196 & 0.275 & 0.529 & 29.56 & 27.91 & 1.405 & 0.562 & 0.713 & 0.0630 & 0\\
Mean--CVaR & 0.80 & 0.6 & 0.000 & 0.099 & 0.901 & 0.266 & 0.384 & 0.350 & 32.24 & 29.51 & 1.313 & 0.788 & 1.102 & 0.1160 & -1\\
Mean--CVaR & 0.80 & 0.8 & 0.000 & 0.086 & 0.914 & 0.277 & 0.402 & 0.321 & 34.93 & 31.04 & 1.227 & 0.981 & 1.185 & 0.1278 & -1\\
CVaR-only & 0.80 & 1.0 & 0.000 & 0.075 & 0.925 & 0.277 & 0.402 & 0.321 & 37.62 & 32.57 & 1.167 & 1.167 & 1.208 & 0.1299 & -1\\
\addlinespace[0.25em]
Expectation-only & 0.84 & 0.0 & 0.024 & 0.147 & 0.829 & 0.024 & 0.147 & 0.829 & 24.10 & 24.10 & 1.500 & 0.000 & $2.3\times10^{-11}$ & 0 & 0\\
Mean--CVaR & 0.84 & 0.2 & 0.017 & 0.165 & 0.818 & 0.121 & 0.249 & 0.630 & 27.28 & 26.48 & 1.500 & 0.300 & 0.376 & 0.0254 & 0\\
Mean--CVaR & 0.84 & 0.4 & 0.011 & 0.184 & 0.806 & 0.257 & 0.370 & 0.373 & 30.46 & 28.35 & 1.412 & 0.565 & 0.865 & 0.0830 & 0\\
Mean--CVaR & 0.84 & 0.6 & 0.006 & 0.202 & 0.792 & 0.277 & 0.402 & 0.321 & 33.62 & 30.07 & 1.246 & 0.747 & 0.942 & 0.0925 & 0\\
Mean--CVaR & 0.84 & 0.8 & 0.002 & 0.221 & 0.777 & 0.277 & 0.402 & 0.321 & 36.77 & 31.80 & 1.146 & 0.917 & 0.913 & 0.0875 & 0\\
CVaR-only & 0.84 & 1.0 & 0.000 & 0.239 & 0.761 & 0.277 & 0.402 & 0.321 & 39.91 & 33.52 & 1.087 & 1.087 & 0.881 & 0.0832 & -1\\
\addlinespace[0.25em]
Expectation-only & 0.88 & 0.0 & 0.024 & 0.147 & 0.829 & 0.024 & 0.147 & 0.829 & 24.10 & 24.10 & 1.333 & 0.000 & $2.3\times10^{-11}$ & 0 & 0\\
Mean--CVaR & 0.88 & 0.2 & 0.032 & 0.236 & 0.732 & 0.130 & 0.381 & 0.488 & 27.97 & 26.95 & 1.333 & 0.267 & 0.488 & 0.0352 & 0\\
Mean--CVaR & 0.88 & 0.4 & 0.037 & 0.342 & 0.621 & 0.277 & 0.402 & 0.321 & 31.72 & 28.98 & 1.310 & 0.524 & 0.601 & 0.0598 & 0\\
Mean--CVaR & 0.88 & 0.6 & 0.041 & 0.446 & 0.512 & 0.277 & 0.402 & 0.321 & 35.41 & 31.02 & 1.112 & 0.667 & 0.472 & 0.0403 & 0\\
Mean--CVaR & 0.88 & 0.8 & 0.046 & 0.533 & 0.421 & 0.277 & 0.402 & 0.321 & 39.10 & 33.06 & 1.013 & 0.810 & 0.463 & 0.0385 & 0\\
CVaR-only & 0.88 & 1.0 & 0.052 & 0.598 & 0.350 & 0.277 & 0.402 & 0.321 & 42.82 & 35.10 & 0.953 & 0.953 & 0.452 & 0.0468 & 0\\
\addlinespace[0.25em]
Expectation-only & 0.92 & 0.0 & 0.024 & 0.147 & 0.829 & 0.024 & 0.147 & 0.829 & 24.10 & 24.10 & 1.000 & 0.000 & $2.3\times10^{-11}$ & 0 & 0\\
Mean--CVaR & 0.92 & 0.2 & 0.053 & 0.301 & 0.646 & 0.125 & 0.460 & 0.415 & 28.63 & 27.64 & 1.000 & 0.200 & 0.462 & 0.0351 & 0\\
Mean--CVaR & 0.92 & 0.4 & 0.084 & 0.474 & 0.442 & 0.255 & 0.414 & 0.331 & 32.87 & 30.27 & 1.000 & 0.400 & 0.343 & 0.0231 & 0\\
Mean--CVaR & 0.92 & 0.6 & 0.119 & 0.598 & 0.283 & 0.277 & 0.402 & 0.321 & 36.98 & 32.91 & 0.845 & 0.507 & 0.392 & 0.0391 & 0\\
Mean--CVaR & 0.92 & 0.8 & 0.156 & 0.671 & 0.173 & 0.277 & 0.402 & 0.321 & 41.02 & 35.58 & 0.782 & 0.625 & 0.539 & 0.0508 & 0\\
CVaR-only & 0.92 & 1.0 & 0.189 & 0.713 & 0.098 & 0.277 & 0.402 & 0.321 & 44.96 & 38.25 & 0.855 & 0.855 & 0.622 & 0.0604 & 0\\
\addlinespace[0.25em]
Expectation-only & 0.96 & 0.0 & 0.024 & 0.147 & 0.829 & 0.024 & 0.147 & 0.829 & 24.10 & 24.10 & 0.000 & 0.000 & $2.3\times10^{-11}$ & 0 & 0\\
Mean--CVaR & 0.96 & 0.2 & 0.087 & 0.298 & 0.616 & 0.087 & 0.298 & 0.616 & 29.45 & 29.45 & 0.000 & 0.000 & $2.1\times10^{-11}$ & 0 & 0\\
Mean--CVaR & 0.96 & 0.4 & 0.174 & 0.448 & 0.377 & 0.174 & 0.448 & 0.377 & 34.21 & 34.21 & 0.000 & 0.000 & $1.6\times10^{-11}$ & 0 & 0\\
Mean--CVaR & 0.96 & 0.6 & 0.260 & 0.540 & 0.201 & 0.260 & 0.540 & 0.201 & 38.67 & 38.67 & 0.000 & 0.000 & $1.7\times10^{-11}$ & $1.3\times10^{-12}$ & 0\\
Mean--CVaR & 0.96 & 0.8 & 0.322 & 0.585 & 0.094 & 0.322 & 0.585 & 0.094 & 43.02 & 43.02 & 0.000 & 0.000 & $1.7\times10^{-11}$ & 0 & 0\\
CVaR-only & 0.96 & 1.0 & 0.360 & 0.604 & 0.036 & 0.360 & 0.604 & 0.036 & 47.28 & 47.28 & 0.000 & 0.000 & $2.0\times10^{-11}$ & $2.4\times10^{-15}$ & 0\\
\bottomrule
\end{tabular}%
}
\end{table}

Based on Table~\ref{tab:braess_modelB_direct_eta_grid_full}, the model at $\eta=0$ is the expectation-only TSUE. The model at $\eta=1$ is the CVaR-only
TSUE, and intermediate values are mean--CVaR cases. In the expectation-only case, Model A and Model B produce exactly the same solution. This is because the CVaR selector no longer affects the generalized costs. Both models compute the same expected path costs under the nominal scenario law $p_s$ and, through the same truncated kernel, produce the same $\pi$, active set, and equilibrium
flow.

The entries with $|\mathcal A^A|-|\mathcal A^B|=-1$ indicate that Model A truncates a route that Model B keeps active. These entries occur at low confidence levels and high tail weights, notably $\alpha=0.80$ with $\eta\ge0.6$ and $\alpha=0.84$ with $\eta=1.0$. In this scenario the CVaR tail is broad enough to include route-local adverse states, and the mean-CVaR weight is large enough for Model A's path-specific tail cost to truncate one path. Model B, using a common systemic tail law, still keeps that path active.

Based on Table~\ref{tab:braess_modelB_direct_eta_grid_full}, increasing $\eta$ shifts the effective perceived law toward the selected CVaR tail law, so the
tail-exposed shortcut path $P_3$ loses flow. The confidence level $\alpha$ has two effects. It sets the tail mass $1-\alpha$ and amplifies the tail through the coefficient $\eta/(1-\alpha)$. At moderate levels such as $\alpha=0.84$ and $\alpha=0.88$, a small $\eta$ already moves the Model B equilibrium to the localized-incident split
$(f_1^B/q,f_2^B/q,f_3^B/q)\approx(0.277,0.402,0.321)$. At $\alpha=0.80$, the tail is broader and more nominal mass stays in the certainty equivalent, so a larger $\eta$ is needed to reach this split. 

When $\alpha=0.96$, Model A and Model B coincide because $1-\alpha=0.04$ equals the probability of the system-wide scenario, so all path-specific Model A tails and the common systemic Model B tail collapse to the same single scenario. Therefore, the two models have the same tail-tilted law, perceived costs, reservation costs, active sets, and equilibrium flows up to solver residuals.

\subsection{Model B as a warm start}
\label{subsec:braess_modelb_warm_start}
Route-local tails generally violate the common-tail condition, so Models A and B may have different equilibrium flows. Model B is not generally an exact solver for the pathwise Model A VI outside this case. But Model B has a convex potential formulation with common state weights, and its equilibrium flow can still be used to initialize Model A. To test this use, we solve Model A on the route-local Braess instance from two starting points: uniform path flows and the corresponding Model B equilibrium flow. Figure~\ref{fig:braess_warm_start} shows the number of Model A fixed-point iterations for each initialization on the direct $(\alpha,\eta)$ grid.

\begin{figure}[htbp]
\centering
\includegraphics[width=1.0\textwidth]{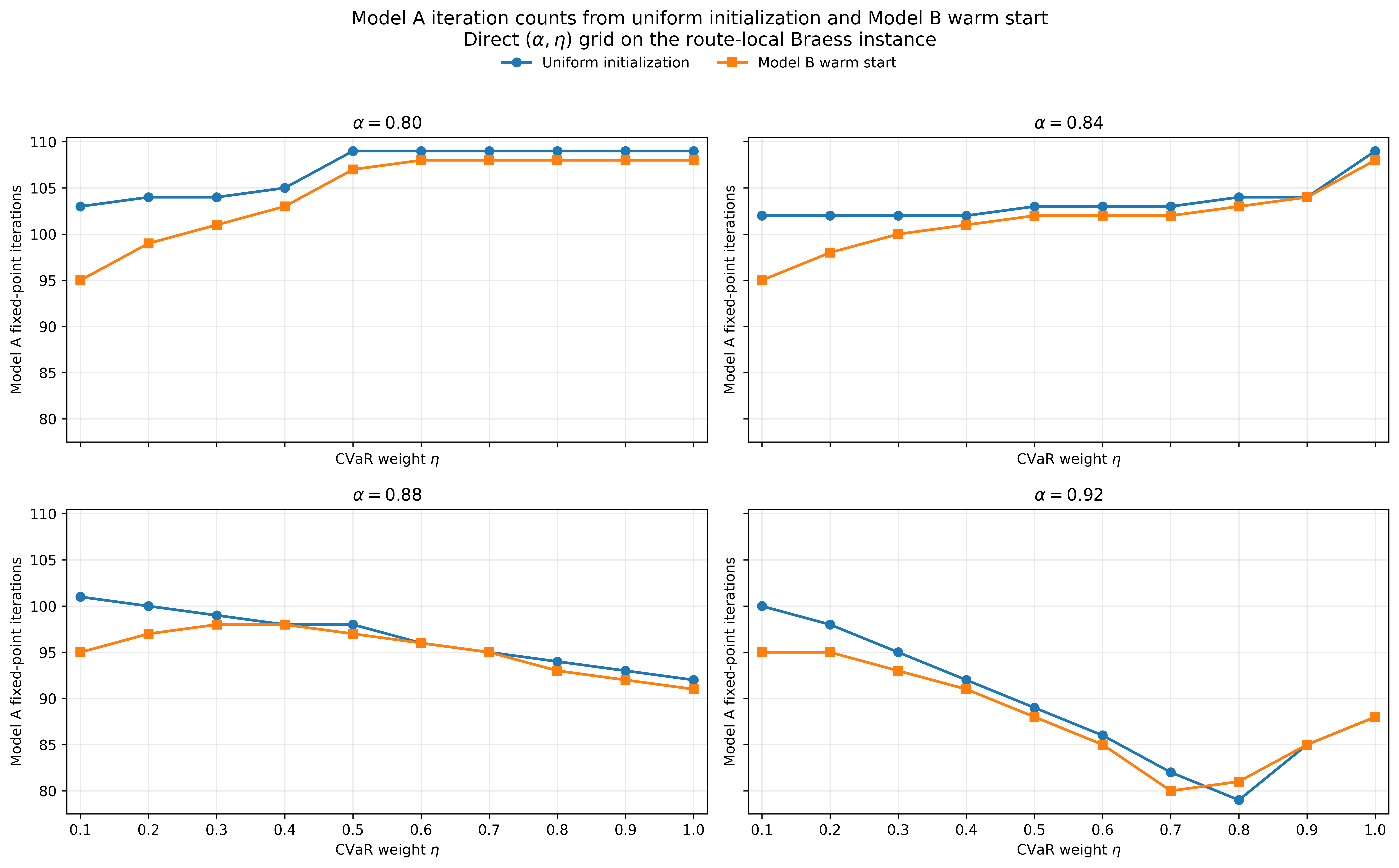}
\caption{Model A fixed-point iterations from uniform initialization and Model B warm start on the route-local Braess instance. }
\label{fig:braess_warm_start}
\end{figure}

Figure~\ref{fig:braess_warm_start} shows that the Model B initialization usually reduces the number of Model A iterations. The gain is most visible at small $\eta$, where the Model B solution is close to the Model A equilibrium, yet far enough from the uniform initialization to accelerate convergence. As $\eta$ increases, the two initializations often become comparable, and some cases are nearly tied. The noticeable reversal occurs around $(\alpha,\eta)=(0.88,0.5)$ and $(0.92,0.8)$, where the warm start needs slightly more iterations. 

\subsection{Route overlap and Independence of Irrelevant Alternatives (IIA) diagnostic}
\label{subsec:braess_iia}

We next test each model's sensitivity to route duplication, a concern related to the Independence of Irrelevant Alternatives (IIA) property. We check whether adding an exact copy of a route changes the total share assigned to that physical route.

We add an exact clone of the upper route,
$
P_{1'}=A\to B\to D,
$
with the same link incidence vector as $P_1=A\to B\to D$. This adds an alternative to the candidate set, but does not create a new physical route. If the assignment rule is invariant to duplication, the combined share of $\{P_1,P_{1'}\}$ should be equal to the original share of $P_1$.

We compare the original physical route shares with those obtained after adding the clone and combining duplicate paths. Let
$
S_g(f)=\frac{1}{q}\sum_{k\in g}f_k
$
denote the share of OD demand $q$ assigned to the physical route group $g$. We measure clone distortion as
\begin{equation}
\mathrm{CloneDistortion}
=
\frac12
\sum_{g\in\mathcal G}
\left|
S_g(f^{\mathrm{clone}})-S_g(f^{\mathrm{base}})
\right|,
\end{equation}
where $\mathcal G$ is the set of physical route groups formed by merging exact duplicates. A value of zero means that adding the clone leaves all physical route shares unchanged. A positive value indicates sensitivity to duplication of routes.

The truncation threshold $\pi_{od}$ determines which retained paths are active. To address overlap and duplication within the retained path set, we use the deterministic path-size correction
$
c_k^{\mathrm{corr}}
=
c_k-\frac{\beta_{PS}}{\theta}\log(PS_k),
PS_k=\sum_{a\in k}\frac{\ell_a/L_k}{n_a}.
$
Here, $\ell_a$ is the length of link $a$, $L_k=\sum_{a\in k}\ell_a$ is the path length, and $n_a$ is the number of retained paths using link $a$. Since $0<PS_k\le1$, the correction adds a nonnegative penalty when $\beta_{PS}>0$, with a larger penalty for smaller $PS_k$.

\begin{figure}[htbp]
\centering
\includegraphics[width=0.46\textwidth]{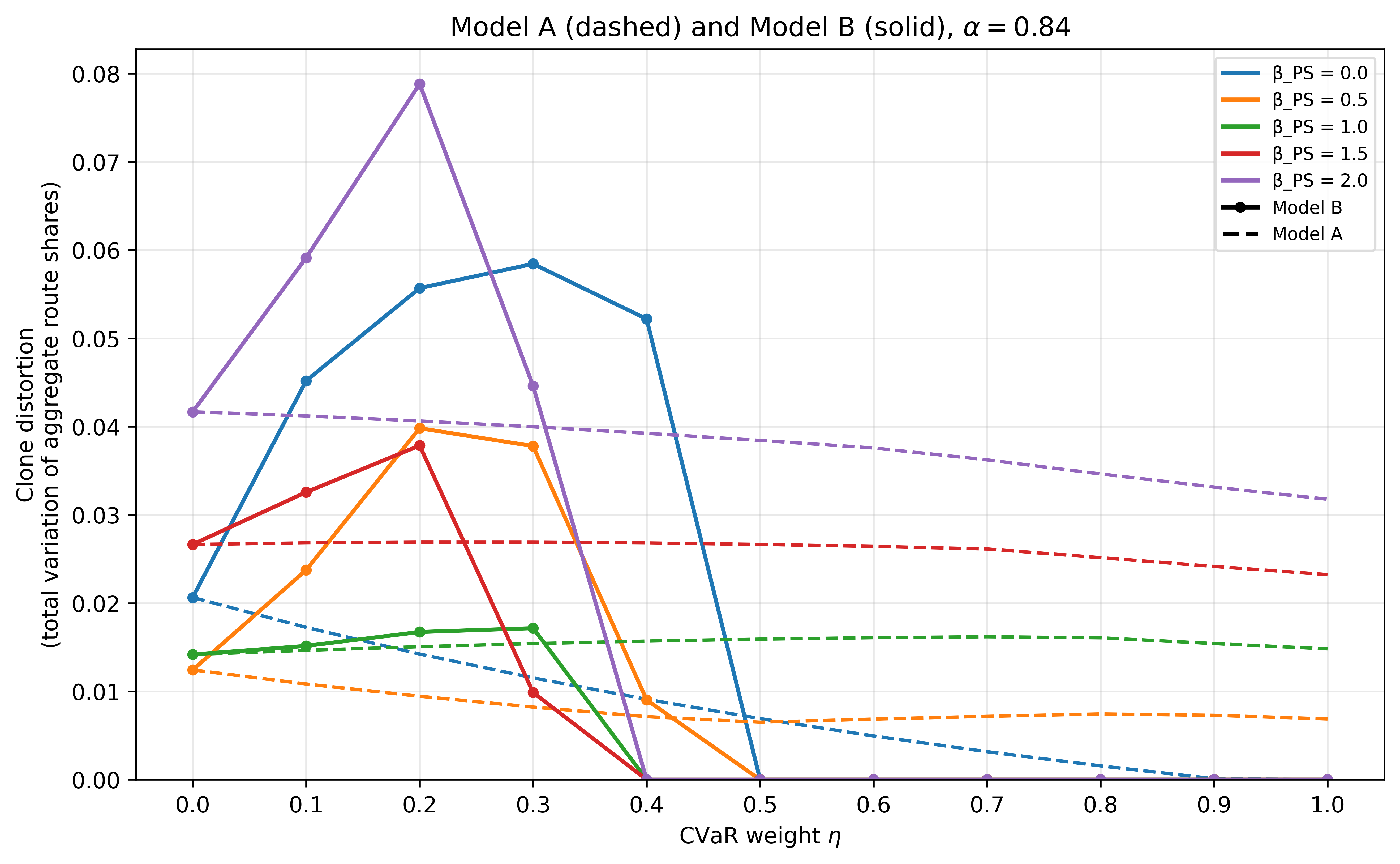}
\includegraphics[width=0.46\textwidth]{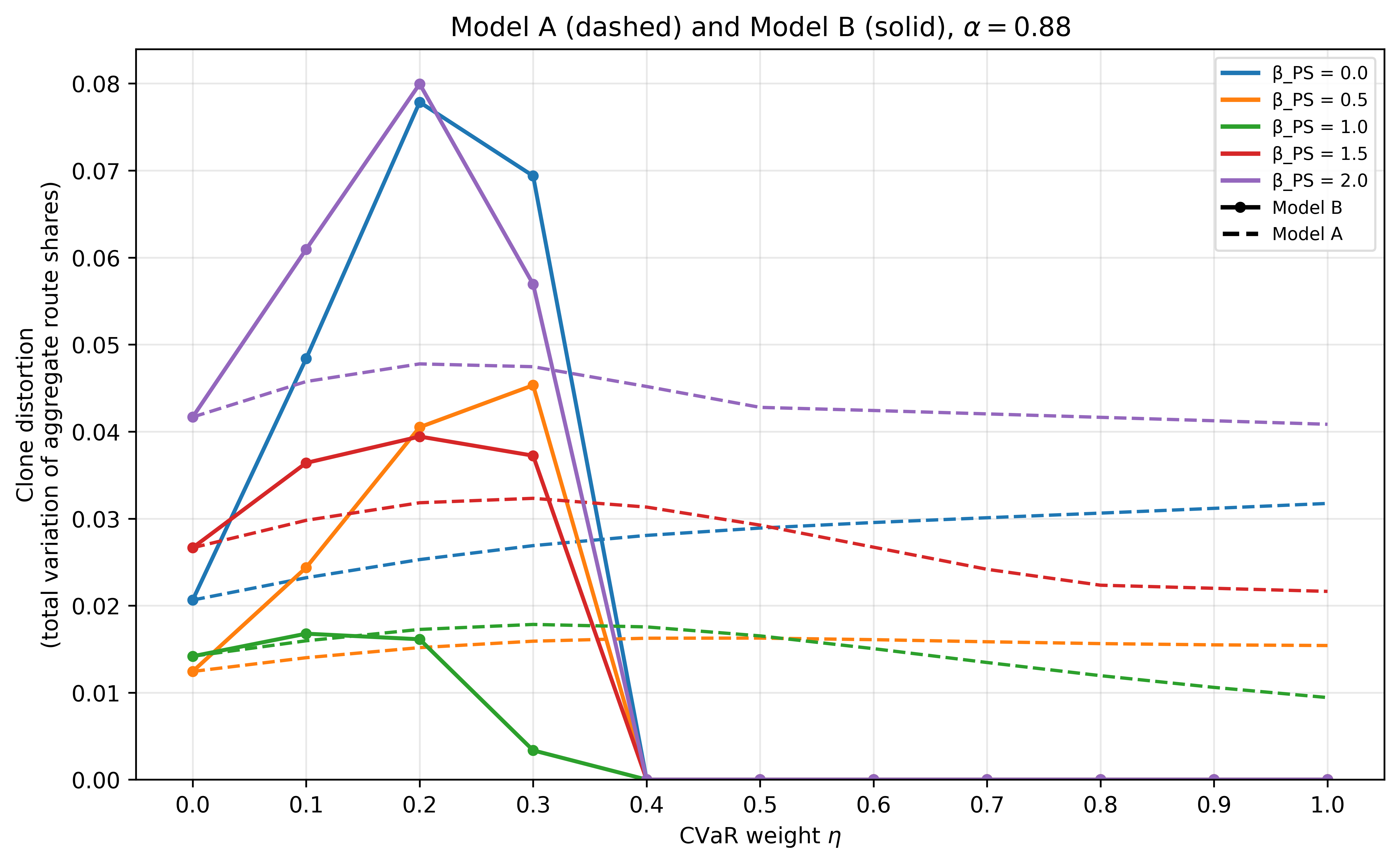}
\includegraphics[width=0.46\textwidth]{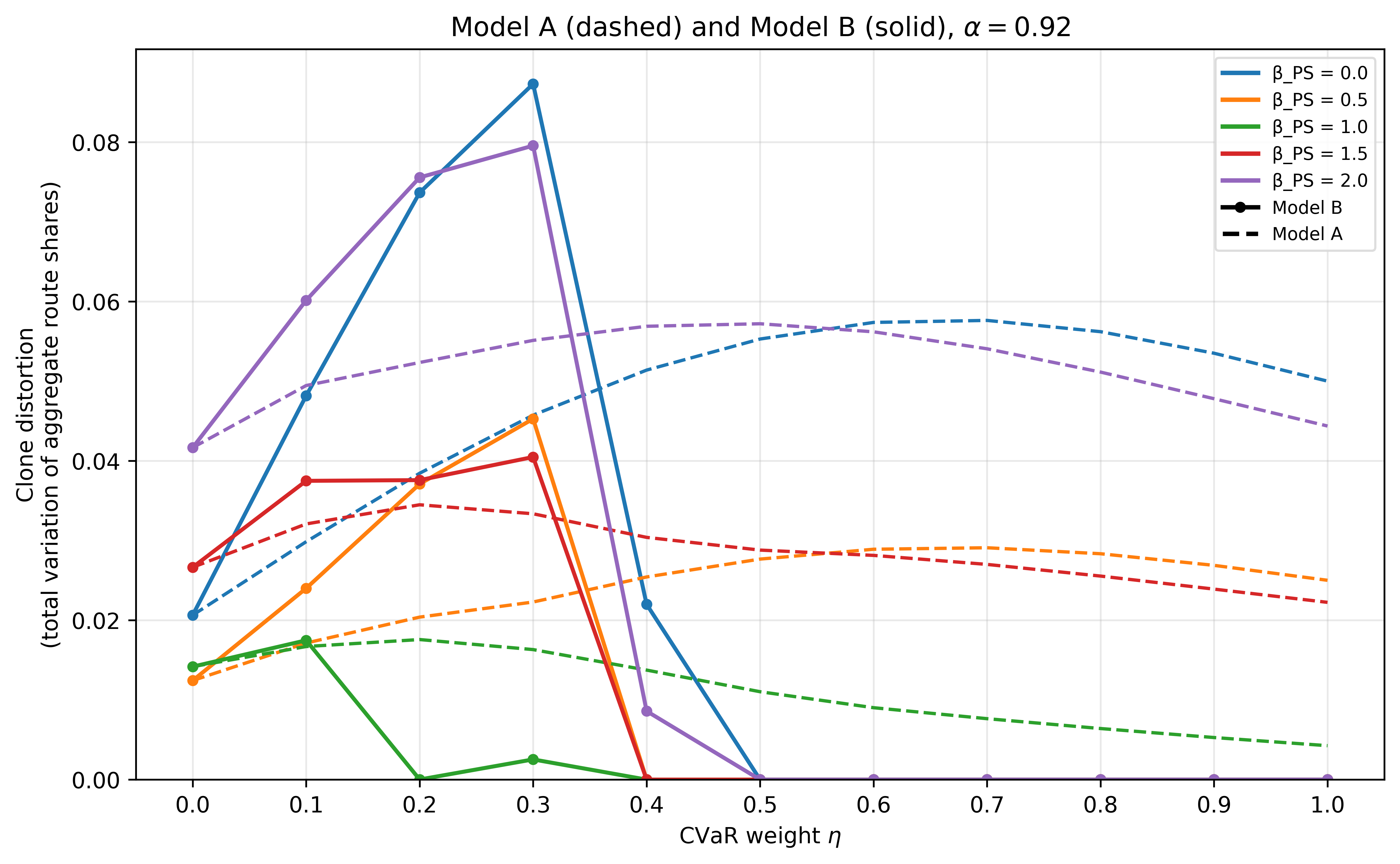}
\includegraphics[width=0.46\textwidth]{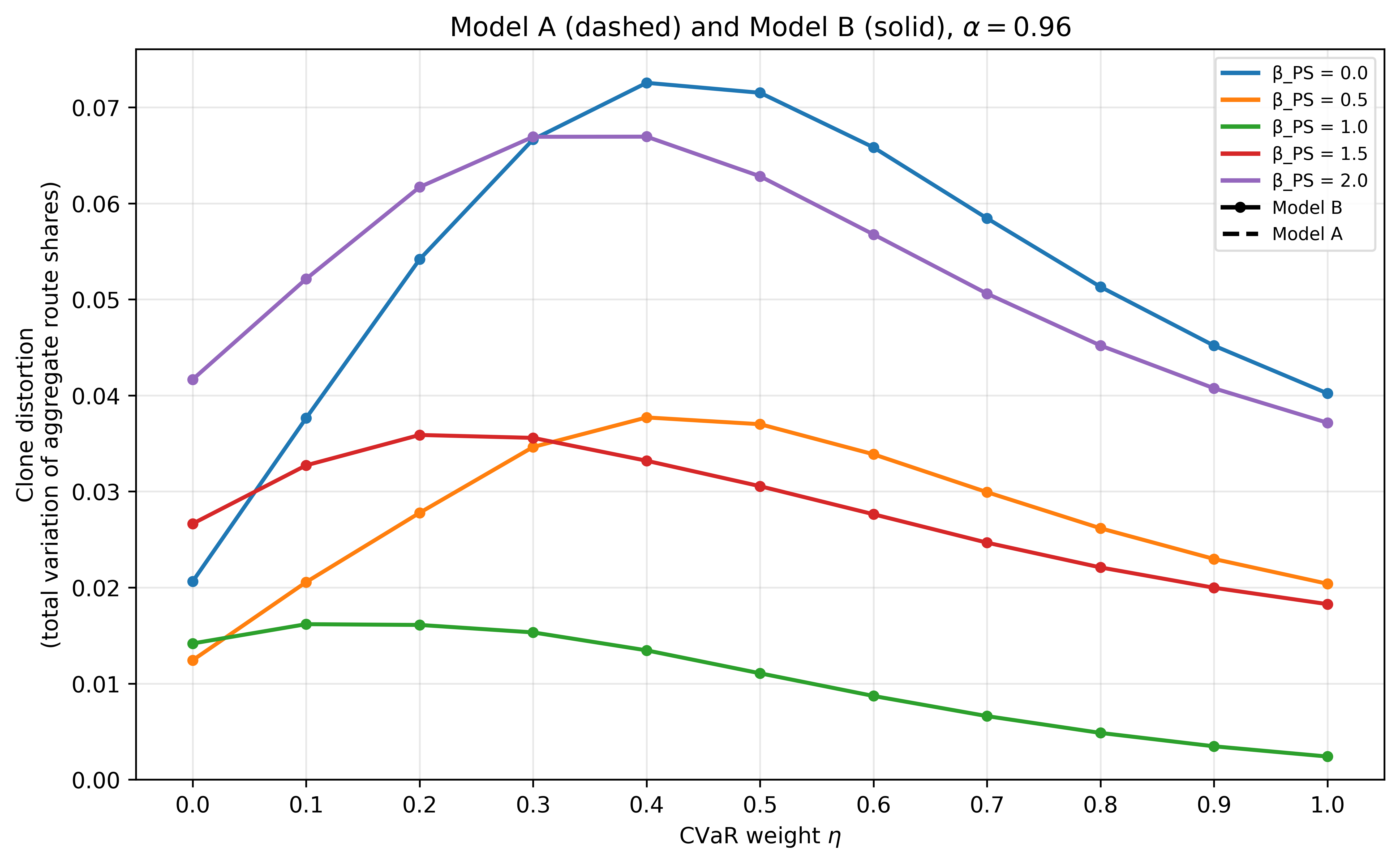}
\caption{Clone-path IIA diagnostic on the route-local Braess instance }
\label{fig:braess_iia}
\end{figure}

Figure~\ref{fig:braess_iia} shows positive clone distortion for Model A across the parameter grid, so duplicating a candidate path changes its assignment. The endogenous truncation multiplier $\pi$ determines which retained paths are active, but it does not guarantee invariance to duplication. The path-size parameter $\beta_{PS}$ controls the penalty for route overlap before the truncated kernel is applied. Its effect on clone distortion is not monotone. Moderate values, especially around $\beta_{PS}=1.0$, often reduce distortion. A weak correction may penalize duplicates too little, and a strong correction may penalize the cloned corridor too much. The path-size parameter therefore needs calibration and does not guarantee lower distortion.

For Model B with $\alpha\in\{0.84,0.88,0.92\}$, clone distortion can fall to zero when the CVaR weight $\eta$ is large enough. In these tests, strong tail weighting can leave the aggregate physical route shares unchanged after a clone is added. This does not mean that Model B resolves IIA in general, because the result depends on the combination of $\alpha$, $\eta$, and $\beta_{PS}$.

\subsection{Finite-support Wasserstein DRO and reliable $B\to C$ upgrade}
\label{subsec:braess_dro}

We next test the Wasserstein DRO model and to see whether Wasserstein ambiguity changes a reliability-oriented link-improvement decision. The experiment uses the route-local Braess instance, but replaces the nominal scenario law by an empirical training law
$
\widehat p=(0.545,0.182,0.182,0.091,0.000)
$
over the five states \{Normal, Upper incident, Lower incident, Shortcut incident, System-wide\}. The zero probability on the system-wide state is intentional: the training sample misses the most severe network-wide disruption. All computed flows are then evaluated under the held-out law
$
p_{\rm test}=(0.48,0.17,0.16,0.12,0.07),
$
which assigns positive mass to the system-wide state and slightly more mass to the shortcut incident. 


For the ground metric defined in Section~\ref{sec:dro}, we represent the five network states by fixed dimensionless coordinates $z^s$. In the order Normal, Upper incident, Lower incident, Shortcut incident, and System-wide, these coordinates are $z^1=(0,0,0)$, $z^2=(1,0,0)$, $z^3=(0,1,0)$, $z^4=(0,0,1)$, and $z^5=(0.85,0.85,0.85)$. The three components correspond to upper-corridor, lower-corridor, and shortcut disruptions. These coordinates are prescribed benchmark inputs, not estimates obtained from the link travel-time parameters. We apply the ground metric with identity scaling in this three-dimensional representation, with no additional normalization.

Let $d_{ij}$ denote the distance between states $i$ and $j$. Substituting the coordinates into the metric gives four groups of off-diagonal entries. The distances from Normal to the three localized incident states are $d_{12}=d_{13}=d_{14}=\sqrt{(1-0)^2+0^2+0^2}=1$. The distances between the localized incident states are $d_{23}=d_{24}=d_{34}=\sqrt{(1-0)^2+(0-1)^2+0^2}$, which gives $\sqrt{2}\approx1.414214$. The distance from Normal to System-wide is $d_{15}=\sqrt{(0-0.85)^2+(0-0.85)^2+(0-0.85)^2}$, which gives $\sqrt{2.1675}\approx1.472243$. The distances from the three localized incident states to System-wide are $d_{25}=d_{35}=d_{45}=\sqrt{(1-0.85)^2+(0-0.85)^2+(0-0.85)^2}$, which gives $\sqrt{1.4675}\approx1.211404$. The matrix is symmetric with zero diagonal. The resulting scenario-to-scenario distance matrix is
\begin{equation}
\mathbf D_{\mathrm{dist}}
\approx
\begin{pmatrix}
0        & 1        & 1        & 1        & 1.472243\\
1        & 0        & 1.414214 & 1.414214 & 1.211404\\
1        & 1.414214 & 0        & 1.414214 & 1.211404\\
1        & 1.414214 & 1.414214 & 0        & 1.211404\\
1.472243 & 1.211404 & 1.211404 & 1.211404 & 0
\end{pmatrix}.
\label{eq:braess_distance_matrix}
\end{equation}
Rows and columns follow the state order above. The radius is a dimensionless transport budget, not a disruption probability. For example, at $\rho=0.10$, moving mass only from Normal to System-wide permits a transfer of $0.10/\sqrt{2.1675}\approx0.067924$. Moving mass only from Shortcut incident to System-wide permits a transfer of $0.10/\sqrt{1.4675}\approx0.082549$. Both transfers are feasible under the training law when all remaining mass stays in its original states.

The DRO model solves \eqref{eq:dro_main_system} with the ambiguity set in Section~\ref{sec:dro} and the distances in \eqref{eq:braess_distance_matrix}. The finite scenario library gives a finite convex nonlinear program. The model can assign mass to the system-wide state within the transport budget because that state remains in the support. The selected worst-case law $Q^\star$ accounts for uncertainty in scenario probabilities. Mean--CVaR then reweights the upper tail of the Beckmann potential under this law to give the common perceived law $\widetilde Q^R$ in \eqref{eq:robust_tilted_measure}. The worst-case law and tail allocation are selected together so that the resulting marginal costs satisfy the equilibrium KKT conditions. These costs enter the truncated kernel and the demand-balance equation in \eqref{eq:generic_truncated_flow}--\eqref{eq:generic_pi_balance}.

We then define 
$
Q^*_{\rm sys}:=Q^\star(\{\xi_{\rm sys}\}),
\widetilde Q^R_{\rm sh}:=\widetilde Q^R(\{\xi_{\rm sh}\}),
\widetilde Q^R_{\rm sys}:=\widetilde Q^R(\{\xi_{\rm sys}\}),
$
where $\xi_{\rm sh}$ and $\xi_{\rm sys}$ are the shortcut-incident and system-wide states. The quantity $Q^*_{\rm sys}$ is the system-wide mass before tail weighting. The other two quantities are perceived weights after mean--CVaR weighting. They distinguish the selected worst-case law from the law used to evaluate path costs. When $\eta=0$, no CVaR weighting is applied and $\widetilde Q^R=Q^\star$.

For the score comparison, let $d\in\{\mathrm{before},\mathrm{after}\}$ denote the network configuration and let $L^{s,d}(f)$ be its scenario-$s$ Beckmann potential. Both configurations use the canonical normalization $L^{s,d}(0)=0$. For a given flow $f$, define
\begin{equation}
\begin{aligned}
\mathrm{MeanTest}^{d}(f)
:=\sum_s p_{\mathrm{test},s}L^{s,d}(f), \quad
\mathrm{CVaRTest}^{d}_{\alpha}(f):=\min_{\gamma\in\mathbb R}
\left\{
\gamma+\frac{1}{1-\alpha}
\sum_s p_{\mathrm{test},s}\bigl(L^{s,d}(f)-\gamma\bigr)_+
\right\}.
\end{aligned}
\label{eq:braess_test_metrics_dro}
\end{equation}
The held-out mean--CVaR potential score is
\begin{equation}
\mathrm{CETest}^{d}_{\alpha,\eta}(f)
:=(1-\eta)\mathrm{MeanTest}^{d}(f)
+\eta\mathrm{CVaRTest}^{d}_{\alpha}(f).
\label{eq:braess_ce_test_dro}
\end{equation}

Let $f^{m,d}_{\alpha,\eta,\rho}$ be the equilibrium for model $m\in\{\mathrm{SP},\mathrm{DRO}\}$ and network $d$, with $\rho=0$ for SP. The before--after score changes are
\begin{equation}
\begin{aligned}
\Delta\mathrm{CE}^{\mathrm{SP}}_{\alpha,\eta}
:={}\mathrm{CETest}^{\mathrm{after}}_{\alpha,\eta}
\!\left(f^{\mathrm{SP},\mathrm{after}}_{\alpha,\eta,0}\right)-\mathrm{CETest}^{\mathrm{before}}_{\alpha,\eta}
\!\left(f^{\mathrm{SP},\mathrm{before}}_{\alpha,\eta,0}\right),
\end{aligned}
\label{eq:braess_delta_ce_sp}
\end{equation}
\begin{equation}
\begin{aligned}
\Delta\mathrm{CE}^{\mathrm{DRO}}_{\alpha,\eta}(\rho)
:={}\mathrm{CETest}^{\mathrm{after}}_{\alpha,\eta}
\!\left(f^{\mathrm{DRO},\mathrm{after}}_{\alpha,\eta,\rho}\right)-\mathrm{CETest}^{\mathrm{before}}_{\alpha,\eta}
\!\left(f^{\mathrm{DRO},\mathrm{before}}_{\alpha,\eta,\rho}\right).
\end{aligned}
\label{eq:braess_delta_ce_dro}
\end{equation}
Each equilibrium is evaluated using the link parameters of its own network. A negative $\Delta\mathrm{CE}$ means that the after-upgrade equilibrium has a lower held-out potential score. It does not imply a reduction in expected total system travel time, its CVaR, or the tail travel time of each route. The comparison also excludes construction costs and does not provide a benefit--cost assessment. We omit the superscript $d$ in tables where the network configuration is clear.

The link change modifies only the scenario-dependent parameters of $B\to C$. OD demand and the candidate paths stay the same. The change makes the link less attractive in normal conditions but raises its capacity and lowers its added delay in the shortcut-incident and system-wide states. Table~\ref{tab:braess_dro_upgrade_bc_functions} gives the before--after parameters.

\begin{table}[htbp]
\centering
\scriptsize
\setlength{\tabcolsep}{3.5pt}
\renewcommand{\arraystretch}{1.08}
\caption{Before--after parameters for the reliable $B\to C$ upgrade.}
\label{tab:braess_dro_upgrade_bc_functions}
\begin{tabular}{lrrrrrrr}
\toprule
& & \multicolumn{3}{c}{Before upgrade} & \multicolumn{3}{c}{After upgrade}\\
\cmidrule(lr){3-5}\cmidrule(lr){6-8}
State & $p_s$ & $t^0_{BC}$ & $c_{BC}$ & $\Delta_{BC}$ & $t^0_{BC}$ & $c_{BC}$ & $\Delta_{BC}$\\
\midrule
Normal & 0.55 & 0.80 & 120.0 & 0.0 & 1.76 & 78.0 & 4.0\\
Upper incident & 0.18 & 0.82 & 115.2 & 0.3 & 1.20 & 95.0 & 1.5\\
Lower incident & 0.17 & 0.82 & 115.2 & 0.3 & 1.20 & 95.0 & 1.5\\
Shortcut incident & 0.06 & 1.00 & 42.0 & 24.0 & 1.20 & 110.0 & 3.0\\
System-wide & 0.04 & 0.98 & 69.6 & 10.0 & 1.40 & 100.0 & 5.0\\
\bottomrule
\end{tabular}
\end{table}

Table~\ref{tab:braess_dro_direct_eta_grid_full} first reports the before-upgrade DRO solutions for the retained grid $\alpha\in\{0.92,0.96\}$ and $\eta\in\{0,0.2,0.4,0.6,0.8,1.0\}$, with $\rho=0.10$. Even when $\eta=0$, the assignment is distributionally robust rather than empirical risk-neutral, so the shortcut share is reduced to $0.720q$. As $\eta$ increases, CVaR tail tilting is applied on top of the adverse law, the effective perceived mass on the system-wide state increases, and the shortcut share falls sharply. For example, at $\alpha=0.96$, $f_3^R/q$ decreases from $0.720$ at $\eta=0$ to $0.036$ at $\eta=1$.

\begin{table}[htbp]
\centering
\scriptsize
\setlength{\tabcolsep}{1.75pt}
\renewcommand{\arraystretch}{1.05}
\caption{Finite-support Wasserstein DRO direct-$(\alpha,\eta)$ grid on the before-upgrade Braess network. The radius is fixed at $\rho=0.10$.}
\label{tab:braess_dro_direct_eta_grid_full}
\resizebox{\textwidth}{!}{%
\begin{tabular}{lrrrrrrrrrrr}
\toprule
 &  &  & \multicolumn{3}{c}{DRO flow shares} & \multicolumn{2}{c}{Equilibrium} & \multicolumn{3}{c}{Robust-law diagnostics} & \multicolumn{1}{c}{Held-out test}\\
\cmidrule(lr){4-6}\cmidrule(lr){7-8}\cmidrule(lr){9-11}\cmidrule(lr){12-12}
Risk case & $\alpha$ & $\eta$ & $f_1^R/q$ & $f_2^R/q$ & $f_3^R/q$ & $\pi^R$ & $|\mathcal A^R|$ & $Q^*_{\rm sys}$ & $\widetilde Q^R_{\rm sh}$ & $\widetilde Q^R_{\rm sys}$ & $\mathrm{CETest}_{\alpha,\eta}$\\
\midrule
Expectation-only & 0.92 & 0.0 & 0.070 & 0.210 & 0.720 & 25.61 & 3 & 0.073 & 0.091 & 0.073 & 1550.6 \\
Mean--CVaR & 0.92 & 0.2 & 0.188 & 0.327 & 0.485 & 30.21 & 3 & 0.080 & 0.073 & 0.264 & 1918.6 \\
Mean--CVaR & 0.92 & 0.4 & 0.255 & 0.449 & 0.296 & 34.61 & 3 & 0.080 & 0.053 & 0.448 & 2224.4 \\
Mean--CVaR & 0.92 & 0.6 & 0.256 & 0.523 & 0.221 & 39.04 & 3 & 0.080 & 0.009 & 0.632 & 2517.6 \\
Mean--CVaR & 0.92 & 0.8 & 0.320 & 0.578 & 0.103 & 43.21 & 3 & 0.083 & 0.002 & 0.817 & 2795.0 \\
CVaR-only & 0.92 & 1.0 & 0.360 & 0.604 & 0.036 & 47.28 & 3 & 0.083 & 0.000 & 1.000 & 3064.9 \\
Expectation-only & 0.96 & 0.0 & 0.070 & 0.210 & 0.720 & 25.61 & 3 & 0.073 & 0.091 & 0.073 & 1550.6 \\
Mean--CVaR & 0.96 & 0.2 & 0.117 & 0.369 & 0.514 & 30.47 & 3 & 0.068 & 0.073 & 0.254 & 1939.6 \\
Mean--CVaR & 0.96 & 0.4 & 0.200 & 0.487 & 0.314 & 34.88 & 3 & 0.068 & 0.055 & 0.441 & 2288.9 \\
Mean--CVaR & 0.96 & 0.6 & 0.274 & 0.554 & 0.172 & 39.09 & 3 & 0.068 & 0.036 & 0.627 & 2606.6 \\
Mean--CVaR & 0.96 & 0.8 & 0.320 & 0.578 & 0.103 & 43.21 & 3 & 0.083 & 0.002 & 0.817 & 2912.5 \\
CVaR-only & 0.96 & 1.0 & 0.360 & 0.604 & 0.036 & 47.28 & 3 & 0.083 & 0.000 & 1.000 & 3204.9 \\
\bottomrule
\end{tabular}%
}
\end{table}

Table~\ref{tab:braess_dro_alpha096_rho_eta_grid} fixes $\alpha=0.96$ and compares the three radii at each $\eta$. Larger radii allow more mass to be assigned to the system-wide state. At $\eta=0.4$, the selected system-wide mass rises from $0.068$ to $0.136$ and then to $0.204$ as $\rho$ increases from $0.10$ to $0.30$. The shortcut share falls from $0.314$ to $0.275$ and then to $0.240$. At $\eta=0.8$, the shortcut shares are $0.103$, $0.099$, and $0.093$, so the flow response is smaller.

\begin{table}[htbp]
\centering
\scriptsize
\setlength{\tabcolsep}{3pt}
\caption{Finite-support Wasserstein radius sensitivity on the before-upgrade network, $\alpha=0.96$ and $\rho\in\{0.10, 0.20, 0.30\}$.}
\label{tab:braess_dro_alpha096_rho_eta_grid}
\resizebox{\textwidth}{!}{%
\begin{tabular}{lrrrrrrrrrrr}
\toprule
Risk case & $\rho$ & $\eta$ & $f_1^R/q$ & $f_2^R/q$ & $f_3^R/q$ & $\pi^R$ & $|\mathcal A^R|$ & $Q^*_{\rm sys}$ & $\widetilde Q^R_{\rm sh}$ & $\widetilde Q^R_{\rm sys}$ & $\mathrm{CE}^{\rm Test}_{\alpha,\eta}$ \\
\midrule
Expectation-only & 0.10 & 0.0 & 0.070 & 0.210 & 0.720 & 25.61 & 3 & 0.073 & 0.091 & 0.073 & 1550.6 \\
Mean--CVaR & 0.10 & 0.2 & 0.117 & 0.369 & 0.514 & 30.47 & 3 & 0.068 & 0.073 & 0.254 & 1939.6 \\
Mean--CVaR & 0.10 & 0.4 & 0.200 & 0.487 & 0.314 & 34.88 & 3 & 0.068 & 0.055 & 0.441 & 2288.9 \\
Mean--CVaR & 0.10 & 0.6 & 0.274 & 0.554 & 0.172 & 39.09 & 3 & 0.068 & 0.036 & 0.627 & 2606.6 \\
Mean--CVaR & 0.10 & 0.8 & 0.320 & 0.578 & 0.103 & 43.21 & 3 & 0.083 & 0.002 & 0.817 & 2912.5 \\
CVaR-only & 0.10 & 1.0 & 0.360 & 0.604 & 0.036 & 47.28 & 3 & 0.083 & 0.000 & 1.000 & 3204.9 \\
Expectation-only & 0.20 & 0.0 & 0.079 & 0.278 & 0.643 & 27.73 & 3 & 0.138 & 0.091 & 0.138 & 1558.6 \\
Mean--CVaR & 0.20 & 0.2 & 0.139 & 0.416 & 0.445 & 32.02 & 3 & 0.136 & 0.073 & 0.309 & 1942.0 \\
Mean--CVaR & 0.20 & 0.4 & 0.215 & 0.510 & 0.275 & 35.98 & 3 & 0.136 & 0.055 & 0.482 & 2287.0 \\
Mean--CVaR & 0.20 & 0.6 & 0.279 & 0.559 & 0.162 & 39.80 & 3 & 0.138 & 0.032 & 0.655 & 2605.2 \\
Mean--CVaR & 0.20 & 0.8 & 0.322 & 0.580 & 0.099 & 43.56 & 3 & 0.152 & 0.000 & 0.830 & 2911.7 \\
CVaR-only & 0.20 & 1.0 & 0.360 & 0.604 & 0.036 & 47.28 & 3 & 0.165 & 0.000 & 1.000 & 3204.9 \\
Expectation-only & 0.30 & 0.0 & 0.093 & 0.349 & 0.559 & 29.79 & 3 & 0.204 & 0.091 & 0.204 & 1568.6 \\
Mean--CVaR & 0.30 & 0.2 & 0.161 & 0.458 & 0.381 & 33.54 & 3 & 0.204 & 0.073 & 0.363 & 1944.9 \\
Mean--CVaR & 0.30 & 0.4 & 0.230 & 0.530 & 0.240 & 37.08 & 3 & 0.204 & 0.055 & 0.522 & 2285.5 \\
Mean--CVaR & 0.30 & 0.6 & 0.281 & 0.558 & 0.162 & 40.51 & 3 & 0.210 & 0.022 & 0.684 & 2605.4 \\
Mean--CVaR & 0.30 & 0.8 & 0.324 & 0.583 & 0.093 & 43.92 & 3 & 0.220 & 0.000 & 0.844 & 2910.4 \\
CVaR-only & 0.30 & 1.0 & 0.360 & 0.604 & 0.036 & 47.28 & 3 & 0.248 & 0.000 & 1.000 & 3204.9 \\
\bottomrule
\end{tabular}%
}
\end{table}

At $\eta=1$, all three radii give the same before-upgrade flow shares $(0.360,0.604,0.036)$, reservation cost $47.28$, and test score $3204.9$ to the reported precision. The selected raw system-wide mass still rises from $0.083$ to $0.165$ and then to $0.248$. But $\widetilde Q^R_{\rm sys}=1$ at every radius because the system-wide state fills the selected CVaR tail. In these cases, the change in raw mass does not change the perceived law or the equilibrium flow.

A larger radius does not always lower the held-out potential score. At $\eta=0$, the score rises from $1550.6$ to $1558.6$ and then to $1568.6$. At $\eta=0.4$, it falls from $2288.9$ to $2287.0$ and then to $2285.5$. At $\eta=0.6$, it first falls from $2606.6$ to $2605.2$ and then rises slightly to $2605.4$.

Table~\ref{tab:braess_dro_reliable_bc_decision} compares the before--after potential scores. At $\eta=0$, the after-upgrade score is higher by $23.3$ under empirical SP and by $12.9$, $19.0$, and $8.9$ under DRO with $\rho=0.10$, $0.20$, and $0.30$. For every tested positive weight $\eta\in\{0.2,0.4,0.6,0.8,1.0\}$, the after-upgrade score is lower under both models at both confidence levels and all three radii. SP and DRO agree on which network has the lower score in each reported case, but the size of the score change differs.

\begin{table}[htbp]
\centering
\scriptsize
\setlength{\tabcolsep}{3pt}
\renewcommand{\arraystretch}{1.06}
\caption{Reliable $B\to C$ upgrade decision on the retained direct-$(\alpha,\eta)$ grid.}
\label{tab:braess_dro_reliable_bc_decision}
\resizebox{\textwidth}{!}{%
\begin{tabular}{lrrrrrrrrl}
\toprule
Risk case & $\alpha$ & $\eta$ & $P_3^{\rm SP,before}$ & $P_3^{\rm SP,after}$ & $\Delta\mathrm{CE}^{\rm SP}$ & $\Delta\mathrm{CE}^{\rm DRO}_{\rho=.05}$ & $\Delta\mathrm{CE}^{\rm DRO}_{\rho=.10}$ & $\Delta\mathrm{CE}^{\rm DRO}_{\rho=.20}$ & CE decision\\
\midrule
Expectation-only & 0.92 & 0.0 & 0.802 & 0.638 & 23.3 & 12.9 & 19.0 & 8.9 & Do not improve \\
Mean--CVaR & 0.92 & 0.2 & 0.367 & 0.780 & -61.0 & -50.3 & -37.1 & -38.2 & Improve \\
Mean--CVaR & 0.92 & 0.4 & 0.321 & 0.865 & -89.9 & -57.0 & -53.2 & -53.5 & Improve \\
Mean--CVaR & 0.92 & 0.6 & 0.321 & 0.924 & -94.2 & -57.1 & -53.3 & -53.4 & Improve \\
Mean--CVaR & 0.92 & 0.8 & 0.321 & 0.961 & -93.9 & -46.9 & -46.4 & -45.8 & Improve \\
CVaR-only & 0.92 & 1.0 & 0.321 & 0.981 & -91.6 & -33.4 & -33.4 & -33.4 & Improve \\
Expectation-only & 0.96 & 0.0 & 0.802 & 0.638 & 23.3 & 12.9 & 19.0 & 8.9 & Do not improve \\
Mean--CVaR & 0.96 & 0.2 & 0.367 & 0.780 & -56.4 & -37.3 & -36.0 & -38.9 & Improve \\
Mean--CVaR & 0.96 & 0.4 & 0.321 & 0.865 & -80.2 & -55.1 & -52.8 & -51.1 & Improve \\
Mean--CVaR & 0.96 & 0.6 & 0.321 & 0.924 & -76.7 & -48.4 & -46.9 & -46.9 & Improve \\
Mean--CVaR & 0.96 & 0.8 & 0.321 & 0.961 & -68.0 & -36.0 & -35.2 & -33.9 & Improve \\
CVaR-only & 0.96 & 1.0 & 0.321 & 0.981 & -57.4 & -15.0 & -15.0 & -15.0 & Improve \\
\bottomrule
\end{tabular}%
}
\end{table}

The empirical SP route shares also respond differently across risk settings. At $\eta=0$, the shortcut share falls from $0.802$ before the upgrade to $0.638$ after it. At $\eta=0.8$, it rises from $0.321$ to $0.961$, and at $\eta=1$ it rises from $0.321$ to $0.981$. These shares are the same at both confidence levels reported.

For every tested positive $\eta$, the potential-score reduction is smaller under DRO than under empirical SP at all three radii. At $(\alpha,\eta)=(0.92,0.8)$, empirical SP gives $\Delta\mathrm{CE}=-93.9$. The DRO changes are $-46.9$, $-46.4$, and $-45.8$ for $\rho=0.10$, $0.20$, and $0.30$. Each change is measured from that model's own before-upgrade equilibrium to its after-upgrade equilibrium. It is not a comparison of the absolute test scores of SP and DRO.

The size of the reduction does not need to vary in one direction with the radius. At $(\alpha,\eta)=(0.96,0.2)$, the DRO changes are $-37.3$, $-36.0$, and $-38.9$ as $\rho$ increases. For CVaR-only cases, the changes are the same across all three radii to the reported precision: $-33.4$ at $\alpha=0.92$ and $-15.0$ at $\alpha=0.96$.

Table~\ref{tab:braess_dro_worst_case_probs} reports the selected worst-case laws $Q^\star$ before CVaR weighting for $(\alpha,\eta)=(0.96,0.6)$. The empirical SP law is shown once because it is unchanged by the link modification. Before the upgrade, the system-wide mass is $0.068$, $0.138$, and $0.210$ at $\rho=0.10$, $0.20$, and $0.30$. After the upgrade, it is $0.083$, $0.152$, and $0.220$. The shortcut-incident mass falls from $0.091$ to $0.008$ at $\rho=0.10$, from $0.081$ to $0$ at $\rho=0.20$, and from $0.055$ to $0$ at $\rho=0.30$. The upper- and lower-incident probabilities each remain $0.182$ in all six DRO laws.

\begin{table}[htbp]
\centering
\scriptsize
\setlength{\tabcolsep}{4pt}
\renewcommand{\arraystretch}{1.06}
\caption{Selected worst-case scenario laws recovered from the finite-support transport problem.}
\label{tab:braess_dro_worst_case_probs}
\begin{tabular}{llrrrrrr}
\toprule
Model & Network & $\rho$ & Normal & Upper & Lower & Shortcut & System-wide\\
\midrule
Empirical Model B-SP & Before/After & 0.00 & 0.545 & 0.182 & 0.182 & 0.091 & 0.000 \\
Wasserstein DRO Model B & Before & 0.10 & 0.478 & 0.182 & 0.182 & 0.091 & 0.068 \\
Wasserstein DRO Model B & After & 0.10 & 0.545 & 0.182 & 0.182 & 0.008 & 0.083 \\
Wasserstein DRO Model B & Before & 0.20 & 0.418 & 0.182 & 0.182 & 0.081 & 0.138 \\
Wasserstein DRO Model B & After & 0.20 & 0.484 & 0.182 & 0.182 & 0.000 & 0.152 \\
Wasserstein DRO Model B & Before & 0.30 & 0.372 & 0.182 & 0.182 & 0.055 & 0.210 \\
Wasserstein DRO Model B & After & 0.30 & 0.416 & 0.182 & 0.182 & 0.000 & 0.220 \\
\bottomrule
\end{tabular}
\end{table}

The $B\to C$ experiment shows that the preferred network depends on the risk setting used in the potential-based comparison. Under the held-out expected Beckmann potential, the original network is preferred when $\eta=0$. The upgraded network has a lower held-out mean--CVaR potential score for every tested positive weight $\eta\in\{0.2,0.4,0.6,0.8,1.0\}$ at both $\alpha=0.92$ and $\alpha=0.96$. This result holds under empirical SP and Wasserstein DRO with $\rho\in\{0.10,0.20,0.30\}$. Based on this score, the comparison favors keeping the original network in the expectation-only setting and choosing the upgraded network in all tested tail-sensitive settings.

\section{Sioux Falls network computational experiment}
\label{sec:sioux_falls_validation}

We evaluate the proposed models on the Sioux Falls benchmark network, which contains 24 nodes, and 76 directed links. Total OD demand is \(360{,}600\) vehicles per hour, and all positive OD demand pairs are retained. For numerical conditioning, flows and capacities are expressed in thousands
of vehicles per hour. Thus, one numerical flow unit represents \(1000\) vehicles per hour, and the entropy reference scale is \(r_0=1\).

Scenario-dependent travel time on each link \(a\in A\) follows
\eqref{eq:bpr_new}. Baseline free-flow times, capacities, and BPR parameters
are read from the Sioux Falls data, with coefficient \(0.15\) and exponent
\(\gamma_a=4\) on every link. Scenarios modify only effective capacity,
free flow time, the BPR coefficient, and additive delay; network topology,
OD demand, and the feasible flow set remain unchanged.

For each positive demand OD pair, we generate up to five shortest simple paths
under baseline free flow time. A path is retained only if its free-flow time
is at most \(2.20\) times the corresponding OD shortest-path time. The
candidate path set is fixed before solving, while the behaviorally active
subset is determined endogenously through
\eqref{eq:generic_truncated_flow}--\eqref{eq:generic_pi_balance}.

The scenario library contains six joint full-network states, summarized in
Table~\ref{tab:sioux_scenario_summary_main}. Probabilities refer to complete
network states rather than independent link-level events. Scenario-generation
rules, spatial exposure groups, scenario-conditioned TSUE benchmarks, and
Model A--Model B diagnostics are reported in
Appendix~\ref{app:sioux_additional_results}. The complete
\(76\times6=456\)-row scenario--link parameterization is available in the
\href{https://github.com/bwc021600/Risk-Averse-Stochastic-User-Equilibrium-on-Uncertain-Transportation-Networks}
{Github} .

\begin{table}[htbp]
\centering
\scriptsize
\setlength{\tabcolsep}{4pt}
\renewcommand{\arraystretch}{1.08}
\caption{Sioux Falls finite scenario library used in the network's computations experiments.}
\label{tab:sioux_scenario_summary_main}
\begin{tabularx}{\textwidth}{llX}
\toprule
Scenario & Probability & Interpretation \\
\midrule
Normal
& 0.550
& Baseline Sioux Falls link performance. \\

Light rain
& 0.200
& Mild network-wide deterioration with a north--south gradient. \\

Minor incident
& 0.150
& Localized incident deterioration centered in the middle-east part of the
network. \\

Flood upper
& 0.040
& Flooding with the strongest exposure in the upper network band. \\

Flood middle
& 0.035
& Flooding with the strongest exposure in the middle network band. \\

Flood lower
& 0.025
& Flooding with the strongest exposure in the lower network band. \\
\bottomrule
\end{tabularx}
\end{table}

Let
$
H_{0.90}
=
\{\text{Flood upper},\text{Flood middle},\text{Flood lower}\}.
$
For scenario \(s\), define the capacity, free flow time, and BPR coefficient
multipliers by
$
m_{a,s}^{c}:=\frac{c_a^s}{c_a},
$$
m_{a,s}^{t}:=\frac{t_a^{0,s}}{t_a^0},
$ and $
m_{a,s}^{\beta}:=\frac{\beta_a^s}{\beta_a}.
$

The flooding fields have a network-wide severity floor of $0.55$. The
scenario-generation rules therefore imply, uniformly over all directed links,
\begin{equation}
\begin{aligned}
&m_{a,h}^{c}
\le 0.549 < 0.72 \le m_{a,\ell}^{c}, \quad
m_{a,\ell}^{t}
\le 1.12 < 1.396 \le m_{a,h}^{t}, \quad 
m_{a,\ell}^{\beta}
\le 1.40 < 1.7975 \le m_{a,h}^{\beta},\\ 
&\Delta_a^\ell
\le 4.70 < 19.20 \le \Delta_a^h, \quad 
\gamma_a^\ell=\gamma_a^h=4,\quad 
\ell\notin H_{0.90},\quad h\in H_{0.90}.
\end{aligned}
\label{eq:sioux_parameter_dominance}
\end{equation}
Because the BPR travel time function is non decreasing in free flow time, the BPR coefficient, and additive delay, and non increasing in capacity for fixed
\(x\ge0\). Then \eqref{eq:sioux_parameter_dominance} gives
$
t_a^\ell(x)\le t_a^h(x),
\forall a\in A,\quad
\forall x\ge0. 
$
Every path travel time and the canonical congestion potential inherit this
ordering. Moreover,
$
\sum_{h\in H_{0.90}}p_h
=
0.040+0.035+0.025
=
0.10
=
1-0.90.
$
Hence, the indicator of \(H_{0.90}\) is a common CVaR selector for every path time and for the congestion potential. Proposition~\ref{prop:common_tail}
therefore establishes Model A--Model B equivalence at \(\alpha=0.90\) for
every \(\eta\in[0,1]\). The numerical diagnostics in
Appendix~\ref{app:sioux_additional_results} confirm this analytical result up
to solver tolerance.

For \(\alpha\in\{0.96,0.965,0.975\}\), the CVaR tail contains only part of the heterogeneous flooding block. Its internal ordering can differ across paths and between path travel times and the network potential, so the common-severe-block sufficient condition no longer guarantees equivalence. The resulting probability law, flow, and VI residual diagnostics are reported in Appendix~\ref{app:sioux_additional_results}.

Figure~\ref{fig:sioux_modelB_flow} compares representative nominal Model B and Wasserstein DRO link-flow solutions at \((\alpha,\eta)=(0.90,0.60)\). The DRO panel uses a training sample of size \(N=50\), Wasserstein radius \(\rho=0.10\), and the exchange algorithm in Section~\ref{sec:exchange}. In both panels, the largest flows concentrate on the central trunk and its principal connectors. 
\begin{figure}[htbp]
\centering

\begin{subfigure}{0.48\textwidth}
\centering
\includegraphics[width=\textwidth]
{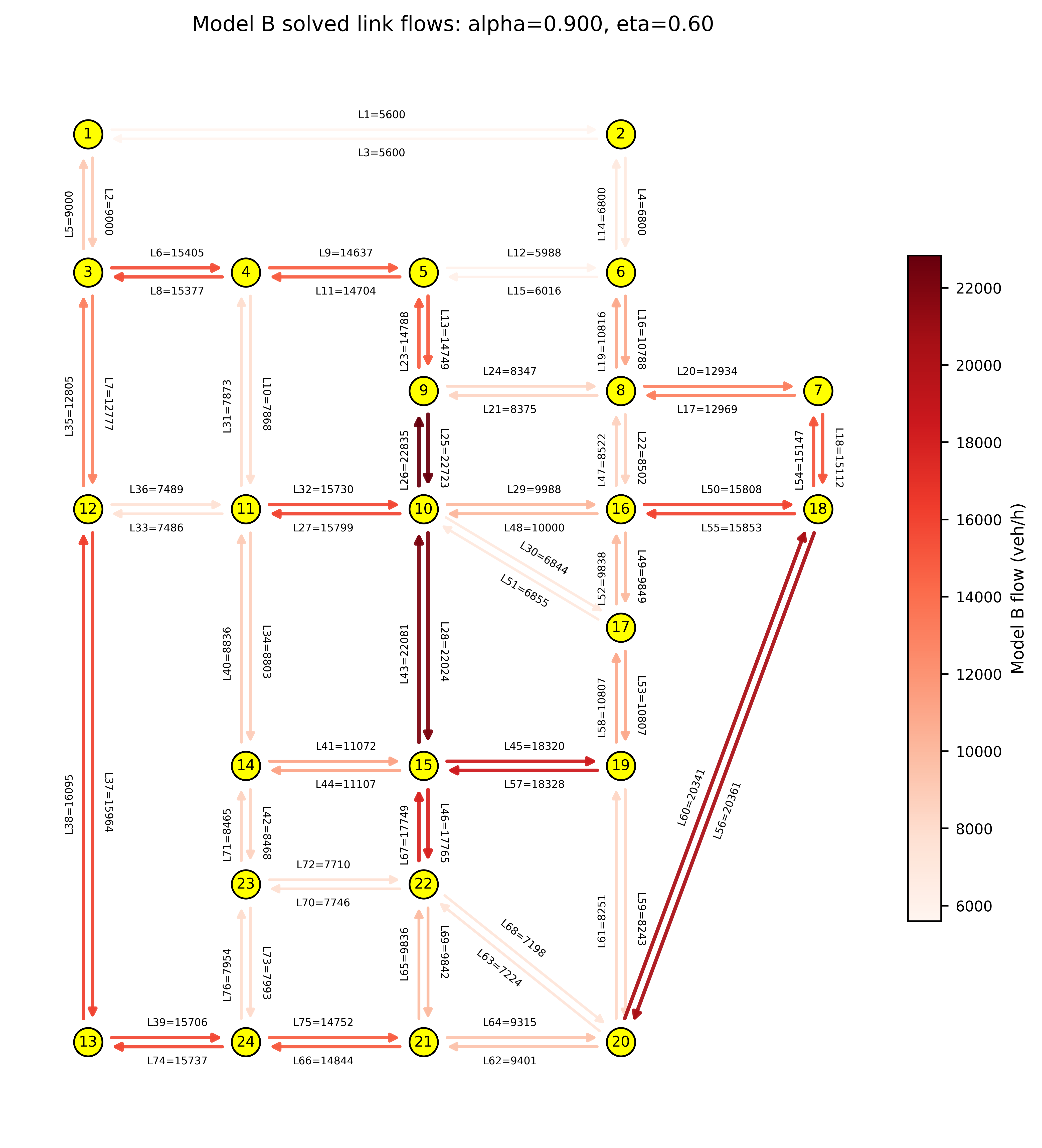}
\caption{Nominal Model B.}
\label{fig:sioux_modelB_flow_nominal}
\end{subfigure}
\hfill
\begin{subfigure}{0.48\textwidth}
\centering
\includegraphics[width=\textwidth]
{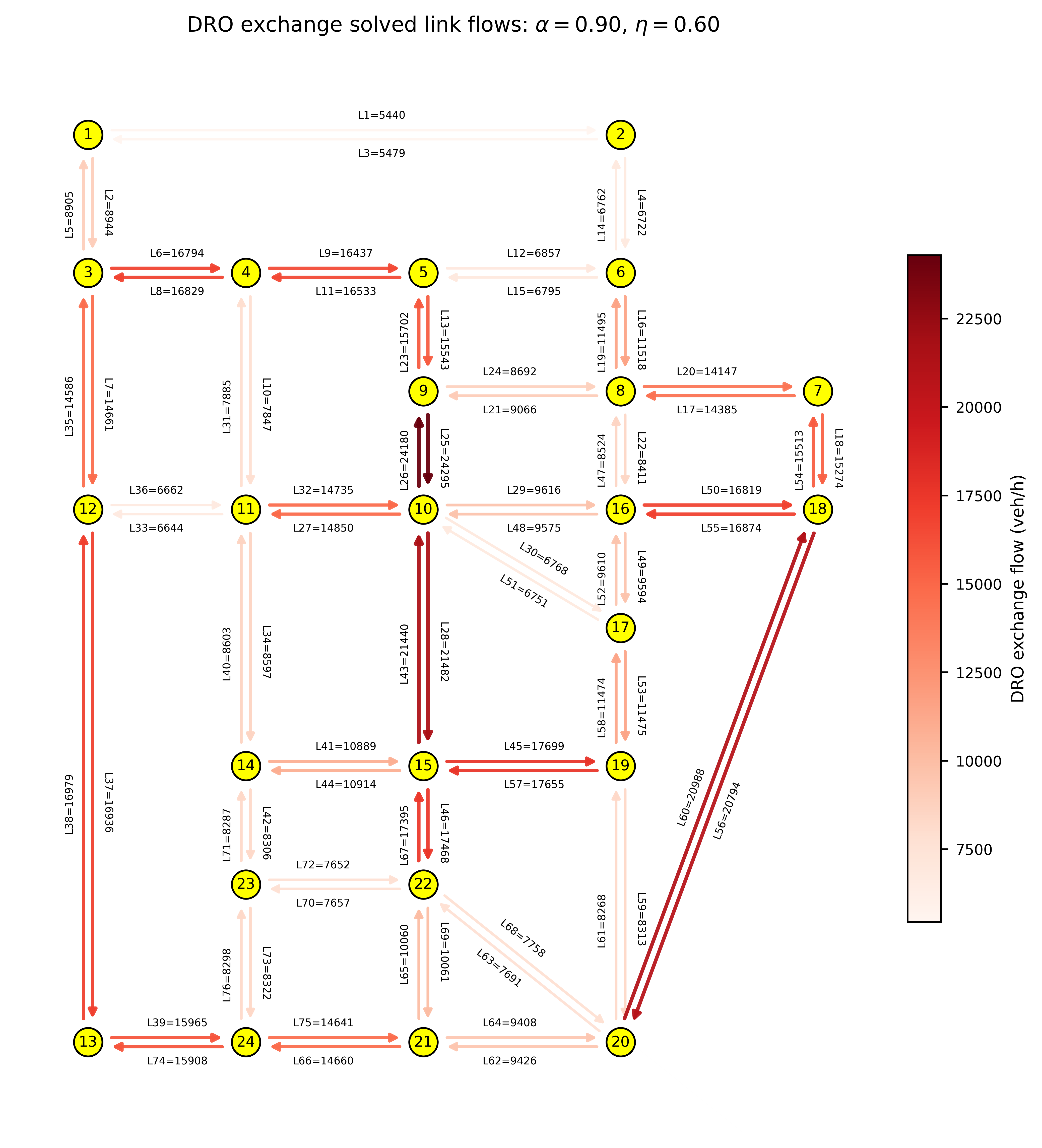}
\caption{Wasserstein DRO with \(N=50\) and \(\rho=0.10\).}
\label{fig:sioux_dro_flow_example}
\end{subfigure}

\caption{Representative Model B and Wasserstein DRO link-flow solutions for
\((\alpha,\eta)=(0.90,0.60)\). Within each panel, greater red intensity
indicates greater solved link flow.}
\label{fig:sioux_modelB_flow}
\end{figure}

\section{Conclusion}
\label{sec:conclusion}

This paper develops a risk- and ambiguity-aware truncated stochastic user equilibrium framework for transportation networks under uncertain supply-side disruptions. The model combines endogenous active-path selection, upper-tail risk sensitivity, and distributional robustness. In all formulations, the OD reservation generalized cost is recovered as the demand conservation multiplier, so route activation is determined jointly with congestion, risk weighting, and ambiguity.

Two mean-CVaR equilibrium models are proposed. Model A applies mean-CVaR directly to each path travel time and yields route-specific tail-tilted laws, so it is formulated as a variational inequality. Model B applies mean-CVaR to the canonically normalized Beckmann congestion potential and produces a convex
common-state-weighted TSUE. A common-tail condition identifies when Model A and Model B coincide in generalized costs, reservation costs, active sets, and equilibrium flows. The Wasserstein DRO extension protects Model B against misspecification of the scenario law. It has a twofold interpretation: the ambiguity set first selects an adverse law, and CVaR then tail-tilts that law before travelers evaluate route costs. 


The Braess and Sioux Falls experiments illustrate these mechanisms. They show when route-specific and common tail laws differ, when the common-tail equivalence holds, how Wasserstein ambiguity affects reliable-route use and upgrade decisions.

However, several limitations remain. The active set is endogenous only within a prespecified candidate path set; thus, path generation, duplicate screening, overlap correction, and IIA-related sensitivities require additional treatment. Wasserstein radii should not be mechanically calibrated from very small
samples of severe network states using generic concentration bounds from very small samples using generic concentration bounds; stress test, validation-based, engineering, or rare-event calibration is needed.

\ACKNOWLEDGMENT{}
This research was partially supported by the Clean Energy and Equitable Transportation Solutions (CLEETS) NSF-UKRI Global Center under NSF award no.~2330565 and by NSF CAREER award no.~2237881.

\section*{Data and code availability}

The code, Sioux Falls input files, computational outputs, and the complete edge-level scenario parameterization used in this study are available in the project repository at \href{https://github.com/bwc021600/Risk-Averse-Stochastic-User-Equilibrium-on-Uncertain-Transportation-Networks}{Github}.

\bibliography{reference}

\newpage

\begin{appendices}

\section{Notation}
\label{app:notation}
\begin{longtable}{@{}p{0.20\textwidth}p{0.74\textwidth}@{}}
\caption{Notation used.}\label{tab:notation_new}\\
\toprule
\small
\textbf{Symbol} & \textbf{Meaning}\\
\midrule
\endfirsthead
\toprule
\textbf{Symbol} & \textbf{Meaning}\\
\midrule
\endhead
\bottomrule
\endfoot
$G=(V,A), \mathcal O$ & Directed network with node set $V$ and link set $A$ and Set of OD pairs.\\
$K_{od}$, $q_{od}$ & Prespecified candidate paths and fixed demand for OD pair $(o,d)$.\\
$f_k^{od}$, $x_a$ & Nonnegative path flow and link flow.\\
$\delta_{ak}^{od}$ & Link--path incidence indicator.\\
$\mathcal F_{fx}$, $\mathcal F$ & Feasible path--link flow-pair set and its path-flow projection; see \eqref{eq:flow_feasibility}--\eqref{eq:feasible_set}.\\
$\xi$, $\Xi$, $P$ & Joint network uncertainty vector, its support, and its law.\\
$\xi^s$, $p_s$ & Joint finite scenario and its probability.\\
$t_a^s(x_a)$ & Scenario-$s$ link travel-time function.\\
$T_k^{od,s}(f)$ & Scenario-$s$ path travel time.\\
$L^s(x)$, $L^s(f)$ & Canonically normalized scenario congestion potential before and after substitution of the link-flow definition.\\
$Z^{\mathrm{TSUE}}(f,\xi)$ & Full scenario TSUE potential.\\
$\theta$ & Utility sensitivity parameter; inverse generalized-cost units.\\
$r_0$ & Fixed reference-flow scale used to make $f/r_0$ dimensionless.\\
$\Psi_{r_0}(f)$ & Truncated-logit entropy regularizer.\\
$\pi_{od}$ & OD demand multiplier and endogenous reservation generalized cost.\\
$\alpha$ & CVaR confidence level.\\
$\eta$ & Effective CVaR weight in the standard mean--CVaR certainty equivalent.\\
$c_k^{A,od}(f)$ & Model A path-level mean--CVaR cost.\\
$F^{A}(f)$ & Model A VI mapping.\\
$\widetilde p_{k,s}^{A,od}(f)$ & Route-specific tail-tilted scenario probability in Model A.\\
$g_k^{B,od}(f;\chi)$ & Selected Model B common-state-weighted marginal path cost.\\
$\widetilde p_s^{B}(f;\chi)$ & Common tail-tilted scenario probability in Model B.\\
$\chi$ & CVaR upper-tail selector, including fractional VaR-boundary mass.\\
$\mathcal C_\alpha(P)$, $\mathcal C_\alpha^S$ & General and finite-scenario CVaR selector sets in \eqref{eq:cvar_selector_set} and \eqref{eq:finite_selector_polytope}.\\
$P_j$, $m_\alpha$, $\chi^\alpha$ & Ordered cumulative probability, VaR-boundary index, and canonical fractional selector.\\
$P^\star$, $\widehat P_N$ & Unknown true law and empirical law.\\
$D_\xi$, $d_\xi$ & Scaling matrix and standardized ground metric.\\
$\rho$, $\mathcal D_N(\rho)$ & Wasserstein radius and ambiguity set.\\
$Q^\star$, $\widetilde Q^R$ & Selected worst-case law and its CVaR tail-tilted law.\\
$b^\star(\xi)$, $\zeta^\star(\xi)$ &
State-dependent allocation on the VaR boundary and the associated KKT-compatible CVaR selector in
\eqref{eq:robust_selector_main}.\\
$t$, $\mathcal T$ & RU threshold and its compact admissible interval.\\
$\kappa$, $s_i$ & Wasserstein dual variable and sample epigraph variables.\\
$\widehat s_i^{(r)}$, $\Delta_i^{(r)}$ & Exact separation value and samplewise violation at exchange iteration $r$.\\
$\mathrm{LB}_r$, $\mathrm{UB}_r$ & Restricted-master lower bound and certified feasible upper bound in $\mathrm{LB}_r\le V^\star\le\mathrm{UB}_r$.\\
$\mathcal K_{\exp}$ & Closed exponential cone used to represent the entropy epigraph.\\
\end{longtable}

\section{Reservation-cost benchmark derivation and endogenous threshold}
\label{app:truncated_derivation}
This appendix connects the behavioral construction to the equilibrium KKT system. The reservation-cost benchmark transformation follows \citet{tan_xu_chen_2024}; the dimensional reference-flow extension and the multiplier proof are stated explicitly here.

\subsection{Reservation-cost benchmark, truncation, and normalized odds}
The reservation-cost benchmark comparison, linear transformation, zero truncation, and odds normalization are given in \eqref{eq:binary_auxiliary}--\eqref{eq:truncated_probability_new}. Those equations establish the behavioral motivation for the positive-part kernel and are not repeated here.

\subsection{KKT endogenization with a fixed reference-flow scale}
Every equilibrium formulation first imposes demand conservation, the path-to-link flow identities, and path/link nonnegativity in \eqref{eq:flow_feasibility}. Because \eqref{eq:link_flow_new} uniquely determines $x$ from $f$, substitute $x=x(f)$ before taking first-order conditions. The selected marginal generalized-cost vector $c(f)$ below therefore already includes the chain rule through the link-flow definition. This reduced KKT system is equivalent to the full system with explicit link-flow equality multipliers.

Consider any equilibrium problem whose nonentropy part has a selected marginal generalized-cost vector $c(f)$. This covers the differentiable risk-neutral problem, a selected subgradient of Model B, and the robust model. The Lagrangian terms relevant to one OD pair are
\begin{equation}
\mathcal L(f,\pi,\nu)
=R(f)+\frac{1}{\theta}\Psi_{r_0}(f)
+\pi_{od}\left(q_{od}-\sum_{k\in K_{od}}f_k^{od}\right)
-\sum_{k\in K_{od}}\nu_k^{od}f_k^{od},
\end{equation}
where $c_k^{od}(f)$ is the selected component of $\partial R(f)$ and $\nu_k^{od}\ge0$. Stationarity is
\begin{equation}
c_k^{od}(f)+\frac{1}{\theta}\psi_{r_0}'(f_k^{od})-\pi_{od}-\nu_k^{od}=0.
\label{eq:generic_stationarity_app}
\end{equation}
Together with $f_k^{od}\ge0$, $\nu_k^{od}\ge0$, and $\nu_k^{od}f_k^{od}=0$, stationarity yields the generic flow relation \eqref{eq:generic_truncated_flow}; demand conservation yields \eqref{eq:generic_pi_balance}. For fixed finite costs, the left-hand side of \eqref{eq:generic_pi_balance} is continuous, equals zero at $\min_k c_k^{od}$, is strictly increasing above that value, and diverges to infinity. Hence the multiplier is unique. Dividing \eqref{eq:generic_truncated_flow} by \eqref{eq:generic_pi_balance} gives the probability kernel in \eqref{eq:truncated_probability_new}, whose denominator equals $q_{od}/r_0$. This proves Proposition~\ref{prop:base_tsue} and the kernel portions of Propositions~\ref{prop:modelA_kernel}, \ref{prop:modelB_kkt}, and \ref{prop:dro_kkt_new}.

\section{CVaR for finite distributions and atoms at VaR}
\label{app:cvar_discrete}
The definitions required by the nominal models are stated in the main text. Specifically, \eqref{eq:var_definition} defines $\VaR_\alpha$, \eqref{eq:cvar_selector_set}--\eqref{eq:cvar_risk_envelope_finite} give the general and finite CVaR risk envelopes, and $
P_j:=\sum_{r=1}^{j}p_r, \text{ and  }m_\alpha:=\min\left\{j:P_j\ge\alpha\right\}
$ give the canonical fractional selector and its exact finite-scenario CVaR expansion. This appendix therefore uses those equations rather than introducing a second set of selector formulas.

The constant selector $\chi\equiv1-\alpha$ belongs to the selector set in \eqref{eq:cvar_selector_set}. Its objective value in \eqref{eq:cvar_risk_envelope} is
\begin{equation}
\begin{aligned}
\frac{1}{1-\alpha}\E_P[(1-\alpha)Y]
=\frac{1}{1-\alpha}\int_{\Xi}(1-\alpha)Y(\zeta)P(\mathrm d\zeta)
=\int_{\Xi}Y(\zeta)P(\mathrm d\zeta)
=\E_P[Y],
\end{aligned}
\end{equation}
which confirms $\CVaR_\alpha(Y)\ge\E_P[Y]$. If several realizations are tied at the VaR value, any allocation of the required boundary mass among those tied realizations that remains feasible in \eqref{eq:finite_selector_polytope} is optimal. The CVaR value remains \eqref{eq:finite_selector_mass_cvar}, while a selected derivative can depend on that allocation.

\subsection{Proof of Proposition~\ref{prop:cvar_upper_atom_new}}
\label{app:proof_cvar_upper_atom}
\begin{proof}{Proof.}
Let $p_M:=\Prob(Y=M)$. Since $Y\le M$ almost surely, every feasible selector in \eqref{eq:cvar_risk_envelope} satisfies
\begin{equation}
\begin{aligned}
\frac{1}{1-\alpha}\E_P[\chi Y]
=\frac{1}{1-\alpha}\int_{\Xi}\chi(\zeta)Y(\zeta)P(\mathrm d\zeta)
\le\frac{M}{1-\alpha}\int_{\Xi}\chi(\zeta)P(\mathrm d\zeta)
=\frac{M}{1-\alpha}\E_P[\chi]
=M.
\end{aligned}
\end{equation}
Taking the supremum gives $\CVaR_\alpha(Y)\le M$. Because $p_M\ge1-\alpha$, the selector
\begin{equation}
\chi^\star:=\frac{1-\alpha}{p_M}\mathbf 1_{\{Y=M\}}
\end{equation}
is feasible and attains
\begin{equation}
\begin{aligned}
\frac{1}{1-\alpha}\E_P[\chi^\star Y]
=\frac{1}{1-\alpha}\int_{\Xi}\chi^\star(\zeta)Y(\zeta)P(\mathrm d\zeta)
=\frac{1}{1-\alpha}\frac{1-\alpha}{p_M}
\int_{\{Y=M\}}M\,P(\mathrm d\zeta)
=\frac{1}{1-\alpha}\frac{1-\alpha}{p_M}Mp_M
=M.
\end{aligned}
\end{equation}
Thus $\CVaR_\alpha(Y)\ge M$, and the two inequalities prove $\CVaR_\alpha(Y)=M$. The selector is supported entirely on $\{Y=M\}$, so values below $M$ cannot affect the result.
\end{proof}

\subsection{Proof of Proposition~\ref{prop:tail_tilt_representation}}
\label{app:proof_tail_tilt}
\begin{proof}{Proof.}
Every component in \eqref{eq:tail_tilt_prob} is nonnegative. Using the selector-set and optimality conditions in \eqref{eq:finite_selector_polytope}--\eqref{eq:cvar_risk_envelope_finite},
\begin{equation}
\begin{aligned}
\E_{\widetilde P^Y}[1]
=\int_{\Xi}1\,\widetilde P^Y(\mathrm d\zeta)
=\sum_{s=1}^S\widetilde p_s^Y
=(1-\eta)\E_P[1]
+\frac{\eta}{1-\alpha}\E_P[\chi^Y]
=(1-\eta)\sum_{s=1}^Sp_s
+\frac{\eta}{1-\alpha}\sum_{s=1}^Sp_s\chi_s^Y
=(1-\eta)+\eta=1.
\end{aligned}
\end{equation}
Thus $\widetilde p^{Y}(\eta,\chi^Y)$ is a probability vector. The same substitution gives
\begin{equation}
\begin{aligned}
\E_{\widetilde P^Y}[Y]
=\int_{\Xi}Y(\zeta)\,\widetilde P^Y(\mathrm d\zeta)
=\sum_{s=1}^S\widetilde p_s^Yy_s
=(1-\eta)\sum_{s=1}^Sp_sy_s
+\frac{\eta}{1-\alpha}\sum_{s=1}^Sp_s\chi_s^Yy_s
=\CE_{\alpha,\eta}(Y),
\end{aligned}
\end{equation}
which proves $\CE_{\alpha,\eta}(Y)
=\E_{\widetilde P^{Y}}[Y]
=\sum_{s=1}^S\widetilde p_s^{Y}(\eta,\chi^Y)y_s.$.
\end{proof}

\section{Proofs for the nominal models}
\label{app:nominal_proofs}

\subsection{Proof of Proposition~\ref{prop:base_tsue}}
\label{app:proof_base_tsue}
\begin{proof}{Proof.}
For each link and scenario, define $\ell_a^s(x):=\int_0^x t_a^s(\omega)\,\mathrm d\omega$. Continuity of $t_a^s$ makes $\ell_a^s$ continuously differentiable, and nondecreasingness of $t_a^s=(\ell_a^s)'$ makes $\ell_a^s$ convex. Since $x(f)$ is affine, $L^s(f)=\sum_a\ell_a^s(x_a(f))$ is convex. The expected potential expands as
\begin{equation}
\begin{aligned}
\E_P[L(f,\xi)]
=\int_{\Xi}L(f,\zeta)P(\mathrm d\zeta)
=\sum_{s=1}^Sp_sL^s(f)
=\sum_{s=1}^Sp_s\sum_{a\in A}\ell_a^s(x_a(f)),
\end{aligned}
\end{equation}
which is convex as a nonnegative weighted sum of convex functions.

By \eqref{eq:entropy_derivatives}, $\Psi_{r_0}$ is strictly convex in the complete path-flow vector. Therefore, the objective in \eqref{eq:base_tsue_system} is strictly convex. The feasible set $\mathcal F$ is a nonempty product of closed bounded simplices and is consequently compact and convex. Continuity gives existence by the Weierstrass theorem, and strict convexity gives uniqueness of the path-flow minimizer.

At the unique minimizer, convex KKT conditions are necessary and sufficient. Equation~\eqref{eq:potential_gradient} gives the marginal cost in \eqref{eq:base_expected_path_cost}. The generic KKT derivation in Appendix~\ref{app:truncated_derivation} then yields \eqref{eq:generic_truncated_flow}--\eqref{eq:generic_pi_balance}; its monotonicity argument proves uniqueness of $\pi_{od}^\star$, the probability identity, and the active-path characterization.
\end{proof}

\subsection{Proof of Proposition~\ref{prop:modelA_existence}}
\label{app:proof_modelA_existence}
\begin{proof}{Proof.}
Every scenario path time $T_k^{od,s}(f)$ is continuous on $\mathcal F$. In the finite-scenario risk envelope \eqref{eq:cvar_risk_envelope_finite}, the selector ranges over the compact polytope $\mathcal C_\alpha^S$ defined in \eqref{eq:finite_selector_polytope}. The maximized selector objective expands as
\begin{equation}
\begin{aligned}
\frac{1}{1-\alpha}\E_P[\chi(\xi)T_k^{od}(f,\xi)]
=\frac{1}{1-\alpha}\int_{\Xi}\chi(\zeta)T_k^{od}(f,\zeta)P(\mathrm d\zeta)
=\frac{1}{1-\alpha}\sum_{s=1}^Sp_s\chi_sT_k^{od,s}(f),
\end{aligned}
\end{equation}
which is jointly continuous in $(f,\chi)$. Hence its value is continuous in $f$. It follows that $c^A$ and $F^A$ are continuous.

For completeness, we give the standard finite-dimensional VI existence argument. Fix $\tau>0$ and let $P_{\mathcal F}$ be Euclidean projection onto the nonempty closed convex set $\mathcal F$. The map
\begin{equation}
G(f):=P_{\mathcal F}\bigl(f-\tau F^A(f)\bigr)
\end{equation}
is continuous and maps the compact convex set $\mathcal F$ into itself. Brouwer's fixed-point theorem supplies $f^\star=G(f^\star)$. The projection characterization
\begin{equation}
\bigl(f^\star-\tau F^A(f^\star)-f^\star\bigr)^\top(y-f^\star)\le0,
\qquad \forall y\in\mathcal F,
\end{equation}
is exactly the path-flow projection of \eqref{eq:modelA_vi_system} after division by $-\tau$. Every $y\in\mathcal F$ has the unique associated link-flow vector $z$ defined by \eqref{eq:link_flow_reduced}, so quantification over $y\in\mathcal F$ is equivalent to quantification over $(y,z)\in\mathcal F_{fx}$. Thus a feasible pair $(f^\star,x^\star)$ exists.

For uniqueness, let $q_{\max}:=\max_{(o,d)}q_{od}$. Every feasible component lies in $[0,q_{\max}]$. Let $H(f):=\nabla\Psi_{r_0}(f)/\theta$, as defined componentwise by \eqref{eq:entropy_derivatives}. Its Jacobian is diagonal with entries $1/[\theta(r_0+f_k^{od})]\ge m_H$, where
\begin{equation}
m_H:=\frac{1}{\theta(r_0+q_{\max})}>0.
\end{equation}
The integral mean-value formula gives
\begin{equation}
(H(f)-H(g))^\top(f-g)\ge m_H\|f-g\|_2^2.
\end{equation}
If $c^A$ is monotone, then $F^A=c^A+H$ is strongly monotone with modulus at least $m_H$. If $f$ and $g$ both solve the VI, testing the VI at $f$ with $y=g$ and at $g$ with $y=f$ and adding gives
\begin{equation}
(F^A(f)-F^A(g))^\top(f-g)\le0.
\end{equation}
Strong monotonicity forces $f=g$. Hence the solution is unique under $\bigl(c^A(f)-c^A(g)\bigr)^\top(f-g)\ge0,
\quad \forall f,g\in\mathcal F$.
\end{proof}

\subsection{Proof of Proposition~\ref{prop:modelA_kernel}}
\label{app:proof_modelA_kernel}
\begin{proof}{Proof.}
The VI condition says precisely that $f^\star$ minimizes the linear function $y\mapsto F^A(f^\star)^\top y$ over $\mathcal F$. This linear program separates by OD pair. For one OD simplex, necessary and sufficient KKT conditions provide a scalar $\pi_{od}^A$ and multipliers $\nu_k^{od}\ge0$ such that
\begin{equation}
F_k^{A,od}(f^\star)-\pi_{od}^A-\nu_k^{od}=0,
\qquad
f_k^{od,\star}\nu_k^{od}=0.
\end{equation}
Eliminating $\nu_k^{od}$ gives \eqref{eq:modelA_complementarity}. Conversely, that complementarity system is the KKT system of the linear minimization problem and therefore implies the VI.

Substituting the mapping in \eqref{eq:modelA_mapping} into \eqref{eq:modelA_complementarity} gives the generic stationarity structure in \eqref{eq:generic_stationarity_app}. Therefore, \eqref{eq:generic_truncated_flow}--\eqref{eq:generic_pi_balance} apply with the Model A cost specialization, and the uniqueness argument in Appendix~\ref{app:truncated_derivation} proves that $\pi_{od}^A$ is unique.
\end{proof}

\subsection{Proof of Proposition~\ref{prop:modelB_unique}}
\label{app:proof_modelB_unique}
\begin{proof}{Proof.}
Each $L^s$ is convex by the link-integral argument used in Proposition~\ref{prop:base_tsue}. Applying the RU representation \eqref{eq:cvar_ru_new} to $L(f,\xi)$ gives a jointly convex function of $(f,\gamma)$ because $(f,\gamma)\mapsto L^s(f)-\gamma$ is convex and the positive-part operator preserves convexity. Partial minimization over $\gamma$ therefore makes $f\mapsto\CVaR_\alpha(L(f,\xi))$ convex. Thus $\mathcal R_{\alpha,\eta}$ is convex.

The entropy $\Psi_{r_0}$ is strictly convex by \eqref{eq:entropy_derivatives}. Consequently, the complete reduced objective in \eqref{eq:modelB_main_system} is strictly convex on $\mathcal F$. It is continuous, and $\mathcal F$ is compact, so a minimizer exists. Strict convexity implies that its path-flow vector is unique.
\end{proof}

\subsection{Proof of Proposition~\ref{prop:modelB_common_law}}
\label{app:proof_modelB_common_law}
\begin{proof}{Proof.}
By construction through \eqref{eq:tail_tilt_prob}, every component of $\widetilde p^{B}(f;\chi)$ is nonnegative. The selector-set condition in \eqref{eq:finite_selector_polytope} gives
\begin{equation}
\begin{aligned}
\E_{\widetilde P^B(f;\chi)}[1]
=\int_{\Xi}1\,\widetilde P^B(f;\chi)(\mathrm d\zeta)
=\sum_{s=1}^S\widetilde p_s^B(f;\chi)
=(1-\eta)\E_P[1]
+\frac{\eta}{1-\alpha}\E_P[\chi]
=(1-\eta)\sum_{s=1}^Sp_s
+\frac{\eta}{1-\alpha}\sum_{s=1}^Sp_s\chi_s
=1.
\end{aligned}
\end{equation}
Thus $\widetilde p^{B}(f;\chi)$ is a probability vector. Moreover,
\begin{equation}
\begin{aligned}
\E_{\widetilde P^B(f;\chi)}[T_k^{od}(f,\xi)]
=\sum_{s=1}^S\widetilde p_s^B(f;\chi)T_k^{od,s}(f)
=(1-\eta)\sum_{s=1}^Sp_sT_k^{od,s}(f)
+\frac{\eta}{1-\alpha}\sum_{s=1}^Sp_s\chi_sT_k^{od,s}(f),
\end{aligned}
\end{equation}
which is exactly \eqref{eq:modelB_marginal} and proves \eqref{eq:modelB_expected_under_tilt}.
\end{proof}

\subsection{Proof of Proposition~\ref{prop:modelB_kkt}}
\label{app:proof_modelB_kkt}
\begin{proof}{Proof.}
Let $\mathcal C_\alpha^S$ be the selector polytope in \eqref{eq:finite_selector_polytope}. Apply the finite risk-envelope representation \eqref{eq:cvar_risk_envelope_finite} to $Y=L(f,\xi)$.
Any admissible selector described in Section~\ref{sec:modelB} is a maximizer of the linear program in \eqref{eq:cvar_risk_envelope_finite}. Because each $L^s$ is differentiable, the gradient of any active selector objective is a CVaR subgradient; equivalently, by the maximum-function subgradient rule,
\begin{equation}
\begin{aligned}
\frac{1}{1-\alpha}\E_P[\chi(\xi)\nabla_fL(f,\xi)]
=\frac{1}{1-\alpha}\int_{\Xi}\chi(\zeta)\nabla_fL(f,\zeta)P(\mathrm d\zeta)
=\frac{1}{1-\alpha}\sum_{s=1}^Sp_s\chi_s\nabla_fL^s(f)
\in\partial\CVaR_\alpha(L(f,\xi)).
\end{aligned}
\end{equation}
Adding the mean term and using \eqref{eq:potential_gradient} gives exactly the vector with components \eqref{eq:modelB_marginal}.

The Model B objective is convex and continuous and $\mathcal F$ is polyhedral. Convex first-order optimality therefore requires a risk subgradient whose negative, after adding the entropy gradient, lies in the normal cone of $\mathcal F$. The subdifferential of the maximum over $\mathcal C_\alpha^S$ is the convex hull of gradients generated by active selectors. If several active selectors are needed in such a convex combination, their convex combination is itself feasible in $\mathcal C_\alpha^S$, remains active because the selector objective is linear in $\chi$, and generates exactly the convexly combined gradient. Hence the required risk subgradient can be represented by one admissible selector $\chi^\star$ that is optimal in \eqref{eq:cvar_risk_envelope_finite}.

The normal cone of each OD simplex admits a scalar equality multiplier and nonnegative lower-bound multipliers. Hence there are $\pi_{od}^B$ for which $0\le f_k^{od,\star}
\ \perp\ 
 g_k^{B,od}(f^\star;\chi^\star)
+\frac{1}{\theta}\psi_{r_0}'(f_k^{od,\star})
-\pi_{od}^{B}\ge0$ holds. Appendix~\ref{app:truncated_derivation} then gives \eqref{eq:generic_truncated_flow}--\eqref{eq:generic_pi_balance} under the Model B cost specialization and proves uniqueness of each demand-balance root.
\end{proof}

\subsection{Proof of Corollary~\ref{cor:modelB_sue_representation}}
\label{app:proof_modelB_sue_corollary}
\begin{proof}{Proof.}
If $f^\star$ solves Model B, Proposition~\ref{prop:modelB_kkt} supplies a valid selector $\chi^\star$ and a selected marginal-cost vector. Proposition~\ref{prop:modelB_common_law} replaces that vector by the tail-tilted-law expectation in \eqref{eq:modelB_common_law_vi}. The resulting inequality is the first-order optimality condition over $\mathcal F$. Because each path-flow vector has one associated link-flow vector, this is equivalent to the displayed quantification over $(y,z)\in\mathcal F_{fx}$.

Conversely, suppose a feasible $f^\star$ and a selector valid for $L(f^\star,\xi)$ satisfy \eqref{eq:modelB_common_law_vi}. By Proposition~\ref{prop:modelB_common_law} and the selector subgradient formula used in Proposition~\ref{prop:modelB_kkt}, the bracketed vector is a subgradient of the complete Model B objective at $f^\star$. The displayed VI is therefore the necessary and sufficient first-order condition for minimizing a convex function over $\mathcal F$. Hence $f^\star$ solves Model B.
\end{proof}

\section{Proof of common-tail equivalence}
\label{app:proof_common_tail}
\begin{proof}{Proof of Proposition~\ref{prop:common_tail}.}
Suppose $\chi^\alpha(f)$ is valid for every path random variable and for $L(f,\xi)$. Applying the finite risk envelope in \eqref{eq:cvar_risk_envelope_finite} to every path time with the common selector $\chi^\alpha(f)$ and substituting the result into \eqref{eq:modelA_cost} yields \eqref{eq:modelB_marginal}, proving $c_k^{A,od}(f)=g_k^{B,od}(f;\chi^\alpha(f)),
\quad \forall (o,d),\ k\in K_{od},\ f\in\mathcal F.$ for all feasible $f$.

Let $f^A$ be any Model A solution, which exists by Proposition~\ref{prop:modelA_existence}. At $f^A$, the common selector produces a valid subgradient of the Model B risk term, and the Model A VI is therefore the first-order optimality condition for the convex Model B objective. Thus $f^A$ is optimal for Model B. Model B has one path-flow optimizer by Proposition~\ref{prop:modelB_unique}; consequently every Model A solution equals that optimizer. Since Model A has at least one solution, the unique Model B optimizer is itself a Model A solution, and the solution sets coincide.

For the common-ordering sufficient condition $t_a^1(x)\le t_a^2(x)\le\cdots\le t_a^S(x),
\quad \forall a\in A,\ x\ge0$, summing link inequalities along a path gives the same scenario order for every path time. Integrating the link inequalities from zero to the common link-flow vector gives the same order for $L^s(f)$. The index-based selector in \eqref{eq:finite_selector} is therefore simultaneously valid for all these random variables, including fractional allocation at a VaR atom. Under the severe-block condition, every scenario in $H_\alpha$ dominates every scenario outside it and the block has mass $1-\alpha$; the selector that assigns one to $H_\alpha$ and zero to its complement is therefore common. These conditions establish the stated sufficient cases.
\end{proof}

\section{Proofs for the distributionally robust model}
\label{app:dro_proofs}

\subsection{Proof of Proposition~\ref{prop:dro_unique_new}}
\label{app:proof_dro_unique}
\begin{proof}{Proof.}
Compactness of $\mathcal F\times\Xi$ and continuity of $L$ imply a finite bound $|L(f,\xi)|\le M_L$. Hence every expectation and CVaR in \eqref{eq:dro_risk_new} is finite uniformly over $Q\in\mathcal D_N(\rho)$.

For fixed $Q$, convexity of
\begin{equation}
\begin{aligned}
f\mapsto(1-\eta)\E_Q[L(f,\xi)]
+\eta\CVaR_{\alpha,Q}(L(f,\xi))
=(1-\eta)\int_{\Xi}L(f,\zeta)Q(\mathrm d\zeta)
+\eta\CVaR_{\alpha,Q}(L(f,\xi))
\end{aligned}
\end{equation}
follows from the RU argument used for Model B. A pointwise supremum of convex functions is convex, so $\mathcal R_{\alpha,\eta}^{\mathrm{DRO}}$ is convex.

To prove continuity directly, set $C_{\alpha,\eta}:=(1-\eta)+\eta/(1-\alpha)$. The positive-part map is one-Lipschitz, so for every $Q$ and $t\in\mathcal T$,
\begin{equation}
\begin{aligned}
\left|\eta t+\E_Q[g_t(f,\xi)]-\eta t-\E_Q[g_t(h,\xi)]\right|
=\left|\int_{\Xi}\left[g_t(f,\zeta)-g_t(h,\zeta)\right]Q(\mathrm d\zeta)\right|
\le C_{\alpha,\eta}\sup_{\zeta\in\Xi}|L(f,\zeta)-L(h,\zeta)|.
\end{aligned}
\end{equation}
Uniform continuity of $L$ on the compact product makes the right side tend to zero with $\|f-h\|$. Taking first the minimum over $t$ and then the supremum over $Q$ preserves this uniform bound. Thus the robust risk functional is continuous. This is also an instance of the maximum-value continuity principle.

The complete objective is the sum of this convex continuous risk term and the strictly convex entropy term. It is continuous on compact $\mathcal F$, so a minimizer exists, and strict convexity makes the path-flow minimizer unique.
\end{proof}

\subsection{Proof of Proposition~\ref{prop:dro_minimax_new}}
\label{app:proof_dro_minimax}
\begin{proof}{Proof.}
Fix $f$ and define
\begin{equation}
\begin{aligned}
\Lambda_f(Q,t)
:=\eta t+\E_Q[g_t(f,\xi)]
=\eta t+\int_{\Xi}g_t(f,\zeta)Q(\mathrm d\zeta).
\end{aligned}
\end{equation}
The RU representation gives
\begin{equation}
\begin{aligned}
(1-\eta)\E_Q[L(f,\xi)]+\eta\CVaR_{\alpha,Q}(L(f,\xi))
=(1-\eta)\int_{\Xi}L(f,\zeta)Q(\mathrm d\zeta)
 +\eta\CVaR_{\alpha,Q}(L(f,\xi))
=\min_{t\in\R}\Lambda_f(Q,t).
\end{aligned}
\end{equation}
The threshold may be restricted to $\mathcal T$. If $t>\overline\ell$, the hinge is zero and
\begin{equation}
\Lambda_f(Q,t)-\Lambda_f(Q,\overline\ell)=\eta(t-\overline\ell)\ge0.
\end{equation}
If $t<\underline\ell$, the hinge equals $L-t$ and direct simplification gives
\begin{equation}
\Lambda_f(Q,t)-\Lambda_f(Q,\underline\ell)
=\frac{\eta\alpha}{1-\alpha}(\underline\ell-t)\ge0.
\end{equation}
When $\eta=0$, the value is independent of $t$.

Because $\Xi$ is compact, $\mathcal P(\Xi)$ is compact under weak convergence and $W_1^{d_\xi}$ metrizes weak convergence. The Wasserstein ball is closed and convex, hence compact and convex. For fixed $t$, $Q\mapsto\Lambda_f(Q,t)$ is continuous and affine. For fixed $Q$, $t\mapsto\Lambda_f(Q,t)$ is continuous and convex. Therefore all hypotheses of Sion's minimax theorem hold, and
\begin{equation}
\sup_{Q\in\mathcal D_N(\rho)}\min_{t\in\mathcal T}\Lambda_f(Q,t)
=\min_{t\in\mathcal T}\sup_{Q\in\mathcal D_N(\rho)}\Lambda_f(Q,t),
\end{equation}
which is \eqref{eq:dro_minimax_new}. Finally, for fixed $Q$, $\Lambda_f(Q,t)$ is jointly convex in $(f,t)$ because $L(\cdot,\xi)$ is convex and $(L(f,\xi)-t)_+$ is convex. Pointwise suprema preserve convexity.
\end{proof}

\subsection{Proof of Proposition~\ref{prop:dro_dual_new}}
\label{app:proof_dro_dual}
\begin{proof}{Proof.}
For fixed $(f,t)$, $g_t(f,\cdot)$ is continuous and bounded on compact $\Xi$. The strong duality theorem for worst-case expectations over a Wasserstein ball centered at an empirical measure \citep{gao_kleywegt_2022} gives
\begin{equation}
\begin{aligned}
&\sup_{Q\in\mathcal D_N(\rho)}\E_Q[g_t(f,\xi)]
=\sup_{Q\in\mathcal D_N(\rho)}\int_{\Xi}g_t(f,\zeta)Q(\mathrm d\zeta)
=\inf_{\kappa\ge0}
\left\{\kappa\rho
+\E_{\widehat P_N}\!\left[
\sup_{\xi\in\Xi}\{g_t(f,\xi)-\kappa d_\xi(\xi,\widehat\xi)\}
\right]\right\}\\
&=\inf_{\kappa\ge0}
\left\{\kappa\rho
+\int_{\Xi}
\sup_{\xi\in\Xi}\{g_t(f,\xi)-\kappa d_\xi(\xi,\widehat\xi)\}
\widehat P_N(\mathrm d\widehat\xi)\right\}
=\inf_{\kappa\ge0}
\left\{\kappa\rho+\frac1N\sum_{i=1}^N
\sup_{\xi\in\Xi}\left[g_t(f,\xi)-\kappa d_\xi(\xi,\xi^i)\right]\right\}.
\end{aligned}
\label{eq:wass_dual_proof}
\end{equation}
The infimum is attained here. Indeed, the support supremum for sample $i$ is at least its value at $\xi=\xi^i$, so the dual objective is bounded below by
\begin{equation}
\begin{aligned}
\kappa\rho+\E_{\widehat P_N}[g_t(f,\xi)]
=\kappa\rho+\int_{\Xi}g_t(f,\zeta)\widehat P_N(\mathrm d\zeta)
=\kappa\rho+\frac1N\sum_{i=1}^Ng_t(f,\xi^i),
\end{aligned}
\end{equation}
which diverges as $\kappa\to\infty$ because $\rho>0$. Continuity in $\kappa$ therefore yields a minimizer.

Substituting \eqref{eq:wass_dual_proof} into Proposition~\ref{prop:dro_minimax_new}, combining the two minimizations, adding the entropy term, and introducing $s_i$ as epigraph variables for the support suprema gives \eqref{eq:dro_sip_new}. Every transformation preserves equality, so the reformulation is exact.
\end{proof}

\subsection{Proof of Proposition~\ref{prop:dro_kkt_new}}
\label{app:proof_dro_kkt}
\begin{proof}{Proof.}
The saddle construction in Appendix~\ref{app:proof_dro_tilted_law} yields a selected robust-risk subgradient $g^R(f^\star)$ on $\mathcal F$. The same construction shows that the entropy-adjusted vector satisfies the first-order variational inequality on $\mathcal F$. This argument uses the compactness, convexity, and gradient assumptions directly and does not rely on the conclusion of the present proposition.

Since $\mathcal F$ is a product of OD simplices, this variational inequality is equivalent to the stated complementarity conditions with one equality multiplier $\pi_{od}^R$ for each OD pair and nonnegative multipliers for the path-flow lower bounds.

Applying the generic KKT derivation in Appendix~\ref{app:truncated_derivation} gives \eqref{eq:generic_truncated_flow}-- \eqref{eq:generic_pi_balance}. For the selected robust marginal costs, each OD demand-balance equation has a unique reservation-cost multiplier.
\end{proof}

\subsection{Proof of
Proposition~\ref{prop:dro_tilted_law_representation}}
\label{app:proof_dro_tilted_law}

\begin{proof}{Proof.}
Let $\mathcal M_+(\Xi)$ denote the finite nonnegative Borel measures on $\Xi$. Introduce the joint distribution--tail-mass set
$
\mathcal E_\alpha
:=
\left\{
(Q,\nu):
Q\in\mathcal D_N(\rho),\quad
\nu\in\mathcal M_+(\Xi),\quad
0\le\nu\le Q,\quad
\nu(\Xi)=1-\alpha
\right\},
$
where $\nu\le Q$ means that $Q-\nu$ is a nonnegative measure. For fixed $Q$, this is the measure form of the selector set in \eqref{eq:cvar_selector_set}: every feasible $\nu$ has a Radon--Nikodym derivative $\zeta=\mathrm d\nu/\mathrm dQ$ with $0\le\zeta\le1$ and $\E_Q[\zeta]=1-\alpha$, and every such selector defines a feasible $\nu$.

The set $\mathcal E_\alpha$ is nonempty because
$\nu=(1-\alpha)Q$ is feasible for every $Q\in\mathcal D_N(\rho)$. It is convex because all its measure constraints are preserved by joint convex combinations of $(Q,\nu)$. It is also compact under weak convergence. Indeed, $\mathcal D_N(\rho)$ is weakly compact, the measures $\nu$ have fixed finite mass on compact $\Xi$, and the
conditions $\nu(\Xi)=1-\alpha$ and $Q-\nu\ge0$ are weakly closed.

Define
$$
\Phi(f;Q,\nu)
:=
(1-\eta)\int_{\Xi}L(f,\xi)\,Q(\mathrm d\xi)
+
\frac{\eta}{1-\alpha}
\int_{\Xi}L(f,\xi)\,\nu(\mathrm d\xi).
$$
The CVaR risk envelope in \eqref{eq:cvar_risk_envelope} gives
$
\mathcal R_{\alpha,\eta}^{\mathrm{DRO}}(f)
=
\max_{(Q,\nu)\in\mathcal E_\alpha}\Phi(f;Q,\nu).
$
This representation holds also for $\eta=0$, when $\Phi$ does not depend on $\nu$.

Set $J(f;Q,\nu):=\Phi(f;Q,\nu)+\Psi_{r_0}(f)/\theta$. Under Assumption~\ref{ass:dro_compact_new}, $J$ is continuous on $\mathcal F\times\mathcal E_\alpha$, convex in $f$, and affine in $(Q,\nu)$. Both sets are compact and convex. Sion's minimax theorem gives
$$
\min_{f\in\mathcal F}
\max_{(Q,\nu)\in\mathcal E_\alpha}J(f;Q,\nu)
=
\max_{(Q,\nu)\in\mathcal E_\alpha}
\min_{f\in\mathcal F}J(f;Q,\nu).
$$
Compactness and continuity give attainment on both sides.
Choose a pair $(Q^\star,\nu^\star)$ attaining the right-hand
side. Together with the given minimizer $f^\star$, minimax
equality yields the saddle inequalities
$$
J(f^\star;Q,\nu)
\le
J(f^\star;Q^\star,\nu^\star)
\le
J(f;Q^\star,\nu^\star),
\qquad
f\in\mathcal F,\quad
(Q,\nu)\in\mathcal E_\alpha.
$$
The first inequality shows that $\Phi(f^\star;Q^\star,\nu^\star) =\mathcal R_{\alpha,\eta}^{\mathrm{DRO}}(f^\star)$. The second shows that $f^\star$ minimizes the differentiable convex function $J(\cdot;Q^\star,\nu^\star)$ on $\mathcal F$.

Under Assumption~\ref{ass:dro_gradient_new}, differentiation under the integrals along feasible directions gives
$$
g^R(f^\star)
:=
(1-\eta)\int_{\Xi}
\nabla_fL(f^\star,\xi)\,Q^\star(\mathrm d\xi)
+
\frac{\eta}{1-\alpha}\int_{\Xi}
\nabla_fL(f^\star,\xi)\,\nu^\star(\mathrm d\xi).
$$
For every $f\in\mathcal F$, convexity and the equality at $f^\star$ imply
$$
\begin{aligned}
\mathcal R_{\alpha,\eta}^{\mathrm{DRO}}(f)
\ge \Phi(f;Q^\star,\nu^\star) \ge
\mathcal R_{\alpha,\eta}^{\mathrm{DRO}}(f^\star)
+
\left\langle g^R(f^\star),f-f^\star\right\rangle.
\end{aligned}
$$
Thus, $g^R(f^\star)$ is a subgradient of the robust risk
on $\mathcal F$. The second saddle inequality also gives
\[
\left\langle
g^R(f^\star)+\frac1\theta\nabla\Psi_{r_0}(f^\star),
f-f^\star
\right\rangle
\ge0,
\qquad f\in\mathcal F.
\]
Since $\mathcal F$ is a product of OD simplices, this
first-order condition supplies OD multipliers $\pi_{od}^R$
and is equivalent to the complementarity conditions in
Proposition~\ref{prop:dro_kkt_new}. The selected pair therefore
generates a KKT-compatible subgradient, not merely a subgradient
that attains the risk value.

Now suppose $\eta>0$ and write $Y(\xi)=L(f^\star,\xi)$.
For a fixed $Q$, maximizing $\Phi(f^\star;Q,\nu)$ over the
feasible tail measures gives the mean--CVaR value under $Q$.
Because $(Q^\star,\nu^\star)$ attains the joint maximum,
$Q^\star$ is a worst-case law in \eqref{eq:dro_risk_new},
and $\nu^\star$ attains the CVaR risk envelope under $Q^\star$.

Choose an RU-optimal threshold $t^\star\in\mathcal T$
for $Y$ under $Q^\star$. Its quantile condition is
\eqref{eq:robust_quantile_condition}. Let
$\zeta^\star=\mathrm d\nu^\star/\mathrm dQ^\star$.
Then $0\le\zeta^\star\le1$ and
$\E_{Q^\star}[\zeta^\star]=1-\alpha$.
Equality of the RU and risk-envelope values gives
\[
0
=
\int_{\Xi}
\left[
(Y(\xi)-t^\star)_+
-
\zeta^\star(\xi)(Y(\xi)-t^\star)
\right]
Q^\star(\mathrm d\xi).
\]
The integrand is nonnegative. It follows that
$\zeta^\star=1$ on $\{Y>t^\star\}$ and
$\zeta^\star=0$ on $\{Y<t^\star\}$,
$Q^\star$-almost surely. On $\{Y=t^\star\}$, define
$b^\star(\xi)=\zeta^\star(\xi)$ and extend $b^\star$
measurably to $\Xi$ with values in $[0,1]$.
This gives \eqref{eq:robust_selector_main}.
The identity $\nu^\star(\Xi)=1-\alpha$ then gives
\eqref{eq:robust_selector_mass_main}.

The measure $\widetilde Q^R=(1-\eta)Q^\star+
\eta\nu^\star/(1-\alpha)$ is nonnegative and has total mass $(1-\eta)+\eta=1$. Its Radon--Nikodym derivative is \eqref{eq:robust_tilted_measure}. Using $\partial L(f^\star,\xi)/\partial f_k^{od} =T_k^{od}(f^\star,\xi)$ in the selected gradient above gives \eqref{eq:robust_tilted_law_cost}. When $\eta=0$, the same saddle construction selects a worst-case expectation law $Q^\star$ and a KKT-compatible gradient. Setting $\widetilde Q^R=Q^\star$ gives \eqref{eq:robust_tilted_law_cost} directly, without any threshold or boundary allocation. The demand-balance equation has a unique root for each OD pair once the selected marginal costs are fixed.
\end{proof}

\section{Support-function identity for the conic special case}
\label{app:conic_identity}
For any vector $a$, sample $y^i$, radius multiplier $\kappa\ge0$, and compact convex support $\widetilde\Xi$, use
\begin{equation}
-\kappa\|y-y^i\|_2
=\min_{\|u\|_2\le\kappa}\{-u^\top y+u^\top y^i\}.
\end{equation}
The objective is affine in both $y$ and $u$, and both feasible sets are compact and convex. Sion's theorem therefore permits interchange of the supremum over $y$ and the minimum over $u$:
\begin{align}
\sup_{y\in\widetilde\Xi}\{a^\top y+b-\kappa\|y-y^i\|_2\}
\nonumber
=\min_{\|u\|_2\le\kappa}
\left\{b+u^\top y^i+\sup_{y\in\widetilde\Xi}(a-u)^\top y\right\}.
\end{align}
For the ball $\widetilde\Xi=\{y:\|y-y^c\|_2\le R\}$, the support function is
\begin{equation}
\sup_{y\in\widetilde\Xi}(a-u)^\top y
=(a-u)^\top y^c+R\|a-u\|_2,
\end{equation}
which proves \eqref{eq:conic_support_identity_new}. Applying the identity separately to both affine pieces of $g_t$ is exact because the supremum of a pointwise maximum of finitely many functions equals the maximum of their individual suprema.

\section{Exact exchange algorithm and stopping certificate}
\label{app:exchange_algorithm}
\begin{algorithm}[h!]
\caption{Exact exchange algorithm for compact-support TSUE-DRO}
\label{alg:exchange_new}
\begin{algorithmic}[1]
\Require Samples $\{\xi^i\}_{i=1}^N$, compact support $\Xi$, radius $\rho$, parameters $(\alpha,\eta,\theta,r_0)$, tolerance $\varepsilon_{\mathrm{sep}}\ge0$.
\State Initialize finite nonempty cut sets $\mathcal X_i^{(0)}\subset\Xi$, for example $\mathcal X_i^{(0)}=\{\xi^i\}$.
\For{$r=0,1,2,\ldots$}
\State Solve the restricted convex main problem
\begin{align*}
\min_{(f,x)\in\mathcal F_{fx},\,t,\kappa,s}\quad
&\mathcal J_r(f,t,\kappa,s)\\
\mathcal J_r(f,t,\kappa,s)
:={}&\eta t+\kappa\rho+\E_{\widehat P_N}[s(\xi)]+\frac1\theta\Psi_{r_0}(f)\\
={}&\eta t+\kappa\rho+\int_{\Xi}s(\zeta)\widehat P_N(\mathrm d\zeta)+\frac1\theta\Psi_{r_0}(f)
={}\eta t+\kappa\rho+\frac1N\sum_{i=1}^Ns_i+\frac1\theta\Psi_{r_0}(f)\\
\text{s.t.}\quad
&s_i\ge g_t(f,\zeta)-\kappa d_\xi(\zeta,\xi^i),
\quad \forall \zeta\in\mathcal X_i^{(r)},\ i=1,\ldots,N,\\
&t\in\mathcal T,\quad \kappa\ge0.
\end{align*}
\State Denote its optimal solution by $(f^{(r)},x^{(r)},t^{(r)},\kappa^{(r)},s^{(r)})$ and its objective by $\mathrm{LB}_r$.
\For{$i=1,\ldots,N$}
\State Solve the separation problem \eqref{eq:exact_separation_new} globally, obtaining $\widehat s_i^{(r)}$ and a maximizer $\xi_i^{(r)}$.
\State Compute $\Delta_i^{(r)}$ from $\Delta_i^{(r)}:=\widehat s_i^{(r)}-s_i^{(r)},
\quad i=1,\ldots,N$.
\If{$\Delta_i^{(r)}>\varepsilon_{\mathrm{sep}}$}
\State $\mathcal X_i^{(r+1)}\gets\mathcal X_i^{(r)}\cup\{\xi_i^{(r)}\}$.
\Else
\State $\mathcal X_i^{(r+1)}\gets\mathcal X_i^{(r)}$.
\EndIf
\EndFor
\State Compute the feasible candidate value $\mathrm{UB}_r$ from \eqref{eq:exchange_ub_main}.
\If{$\max_i\Delta_i^{(r)}\le\varepsilon_{\mathrm{sep}}$}
\State \Return $(f^{(r)},x^{(r)},t^{(r)},\kappa^{(r)})$, $\mathrm{LB}_r$, and $\mathrm{UB}_r$.
\EndIf
\EndFor
\end{algorithmic}
\end{algorithm}

\subsection{Proof of Proposition~\ref{prop:exchange_certificate}}
\begin{proof}{Proof.}
The restricted problem removes constraints from the full semi-infinite program, so its optimal value is a lower bound: $\mathrm{LB}_r\le V^\star$. Exact separation sets $\widehat s_i^{(r)}$ equal to the complete support supremum at the current $(f^{(r)},t^{(r)},\kappa^{(r)})$. Replacing $s_i^{(r)}$ by $\widehat s_i^{(r)}$ therefore creates a feasible point of the full semi-infinite program and gives $V^\star\le\mathrm{UB}_r$. The two candidate objectives differ only in the sample epigraph terms, so
\begin{equation}
\begin{aligned}
\mathrm{UB}_r-\mathrm{LB}_r
=\E_{\widehat P_N}[\widehat s^{(r)}(\xi)-s^{(r)}(\xi)]
&=\int_{\Xi}[\widehat s^{(r)}(\zeta)-s^{(r)}(\zeta)]\widehat P_N(\mathrm d\zeta)
=\frac1N\sum_{i=1}^N(\widehat s_i^{(r)}-s_i^{(r)})
=\E_{\widehat P_N}[\Delta^{(r)}(\xi)]\\
&=\int_{\Xi}\Delta^{(r)}(\zeta)\widehat P_N(\mathrm d\zeta)
=\frac1N\sum_{i=1}^N\Delta_i^{(r)}.
\end{aligned}
\end{equation}
At a restricted-master optimum, the positive coefficient of $s_i$ in the objective forces $s_i^{(r)}$ to equal the maximum of the currently imposed cut values. Since the current cut set is contained in the full support, $\widehat s_i^{(r)}\ge s_i^{(r)}$, and every violation is nonnegative. If their maximum is at most $\varepsilon_{\mathrm{sep}}$, their average is at most that tolerance. At zero tolerance the lower and upper bounds coincide with $V^\star$.
\end{proof}

\subsection{Proof of Proposition~\ref{prop:exchange_convergence_new}}
\begin{proof}{Proof.}
This is the standard nested-discretization argument underlying exchange methods; the details are included to verify the assumptions for the present program. Throughout this proof, the separation tolerance is zero, so every strictly violated constraint is added.

Because cut sets are nested, each restricted feasible region is contained in the preceding one, so $\mathrm{LB}_{r+1}\ge\mathrm{LB}_r$. Every restricted problem is a relaxation of the full semi-infinite program, hence $\mathrm{LB}_r\le V^\star$.

The variables $f$ and $t$ lie in compact sets. A feasible full point with $\kappa=0$ and $s_i=\sup_{\xi\in\Xi}g_t(f,\xi)$ gives a finite common upper bound on all restricted optimal values. Initialization with $\xi^i\in\mathcal X_i^{(0)}$ imposes
\begin{equation}
s_i\ge g_t(f,\xi^i)
\end{equation}
independently of $\kappa$. Since $g_t$ and the entropy are bounded on the compact $(f,t,\xi)$ domain, the restricted objective upper bound and $\rho>0$ give a uniform upper bound on $\kappa$. At an optimum, each $s_i$ can be chosen equal to the maximum of finitely many continuous cut values; boundedness of $(f,t,\kappa)$ and compactness of $\Xi$ therefore bound every $s_i$. The complete master-solution sequence is bounded.

Take a convergent subsequence with limit $(\bar f,\bar t,\bar\kappa,\bar s)$. Suppose this limit violates a full constraint for some sample $i$: there are $\bar\xi\in\Xi$ and $\epsilon>0$ such that
\begin{equation}
g_{\bar t}(\bar f,\bar\xi)-\bar\kappa d_\xi(\bar\xi,\xi^i)\ge\bar s_i+3\epsilon.
\end{equation}
By continuity, the exact separation violation along all sufficiently late members of the subsequence exceeds $2\epsilon$. The corresponding maximizers lie in compact $\Xi$; pass to a further subsequence on which they converge to $\xi^\dagger$. The generated cut at each such iteration is retained in every later master. Continuity then implies that the limit must satisfy
\begin{equation}
\bar s_i\ge g_{\bar t}(\bar f,\xi^\dagger)-\bar\kappa d_\xi(\xi^\dagger,\xi^i),
\end{equation}
while the maximizing property and the assumed violation imply the reverse inequality with a positive gap, a contradiction. Hence every accumulation point is feasible for the full program.

Let $\bar V$ be the limit of the nondecreasing lower bounds. Along the convergent subsequence, continuity of the restricted objective gives the objective value of the feasible accumulation point as $\bar V$. Therefore
\begin{equation}
V^\star\le\bar V\le V^\star,
\end{equation}
so the accumulation point is optimal and $\bar V=V^\star$. If $\Xi$ is finite, each nonterminal exact-separation iteration adds at least one previously absent sample--support cut. Only finitely many such cuts exist, so the method terminates finitely.
\end{proof}

\subsection{Safe implementation accelerations}
Separation problems are independent across empirical samples and can be solved in parallel. The hinge epigraph associated with a support point can be shared across samples because $(L(f,\zeta)-t)_+$ does not depend on the empirical index. Warm starts and retained cut pools are safe. Partial separation can be used between iterations, but termination must follow an exact or otherwise certified full audit. Heuristic maximizers may add valid cuts, yet they do not produce a certified upper bound. Arbitrary cut deletion can invalidate the lower-bound sequence; cuts should be removed only when their redundancy is certified. Stabilization should use a level or bundle construction that preserves an explicit lower bound for the original restricted problem.

\section{Additional Sioux Falls computational results}
\label{app:sioux_additional_results}
\label{app:sioux_base_line}

This appendix reports the Sioux Falls details that are omitted from the compact main-text illustration.  It contains the detailed scenario construction, scenario-conditioned TSUE benchmarks, and Model A--Model B diagnostics.

\subsection{Detailed scenario construction}
\label{subsec:sioux_scenarios}

The finite scenario library contains six joint full-network states, summarized in Table~\ref{tab:sioux_scenario_summary_main}. Detailed scenario-construction rules, spatial exposure groups, scenario-conditioned TSUE benchmarks, and Model A--Model B diagnostics are reported in Appendix~\ref{app:sioux_additional_results}. The complete
\(76\times6=456\)-row edge-level scenario table is available in the project repository as \href{https://github.com/bwc021600/Risk-Averse-Stochastic-User-Equilibrium-on-Uncertain-Transportation-Networks}
{\texttt{Github}}.

\begin{table}[htbp]
\centering
\scriptsize
\setlength{\tabcolsep}{3.0pt}
\renewcommand{\arraystretch}{1.12}
\caption{Sioux Falls scenario probabilities, affected directed links, and
directed-link modification rules.  Link travel times are evaluated using
\eqref{eq:bpr_new}.}
\label{tab:sioux_scenario_rules}
\begin{tabularx}{\textwidth}{ll>{\raggedright\arraybackslash}p{0.27\textwidth}X}
\toprule
Scenario & Probability & Affected directed links & Directed-link modification rule \\
\midrule

Normal
& 0.550
& None.
& No modification.  Capacity, free-flow time, BPR coefficient, exponent, and
additive delay remain at their baseline Sioux Falls values. \\

Light rain
& 0.200
& All 76 directed links.  The largest severities occur on
links \(1,3,2,4,5,\) and \(14\).
& Mild network-wide deterioration with a north--south gradient.  Capacity is
multiplied by \(\mathrm{clip}(1-0.10h_a^s,0.80,1.00)\), free-flow time by
\(1+0.10h_a^s\), and the BPR coefficient by \(1+0.15h_a^s\).  The additive
delay is \(0.40+1.10h_a^s\). \\

Minor incident
& 0.150
& All 76 directed links, with the strongest deterioration near the
middle-east incident center.  The largest severities occur on
links \(30,51,29,48,49,\) and \(52\).
& Localized incident deterioration.  Capacity is multiplied by
\(\mathrm{clip}(1-0.28h_a^s,0.65,1.00)\), free-flow time by
\(1+0.12h_a^s\), and the BPR coefficient by \(1+0.40h_a^s\).  The additive
delay is \(0.20+4.50h_a^s\). \\

Flood upper
& 0.040
& All 76 directed links.  The largest severities occur in the upper network
band, especially on links \(6,8,9,11,12,\) and \(15\).
& Flooding deterioration with the strongest impact in the upper part of the
network.  The flooding field has a positive full-network floor.  Capacity is
multiplied by \(\mathrm{clip}(1-0.82h_a^s,0.18,0.58)\), free-flow time by
\(1+0.72h_a^s\), and the BPR coefficient by \(1+1.45h_a^s\).  The additive
delay is \(6.00+24.00h_a^s\). \\

Flood middle
& 0.035
& All 76 directed links.  The largest severities occur in the middle network
band, especially on links \(30,49,51,52,27,\) and \(29\).
& Same flooding rule as Flood upper, but the largest severity is centered in
the middle part of the network.  Every link still receives the positive
full-network flooding floor. \\

Flood lower
& 0.025
& All 76 directed links.  The largest severities occur in the lower network
band, especially on links \(70,72,59,61,42,\) and \(46\).
& Same flooding rule as Flood upper, but the largest severity is centered in
the lower part of the network.  Every link still receives the positive
full-network flooding floor. \\

\bottomrule
\end{tabularx}
\end{table}

Figure~\ref{fig:sioux_link_groups} maps the link groups used in Table~\ref{tab:sioux_scenario_rules}. The figure labels every directed link and highlights the upper, middle, middle-east, middle-west, and lower exposure
regions. 


\begin{figure}[htbp]
\centering
\includegraphics[width=0.86\textwidth]{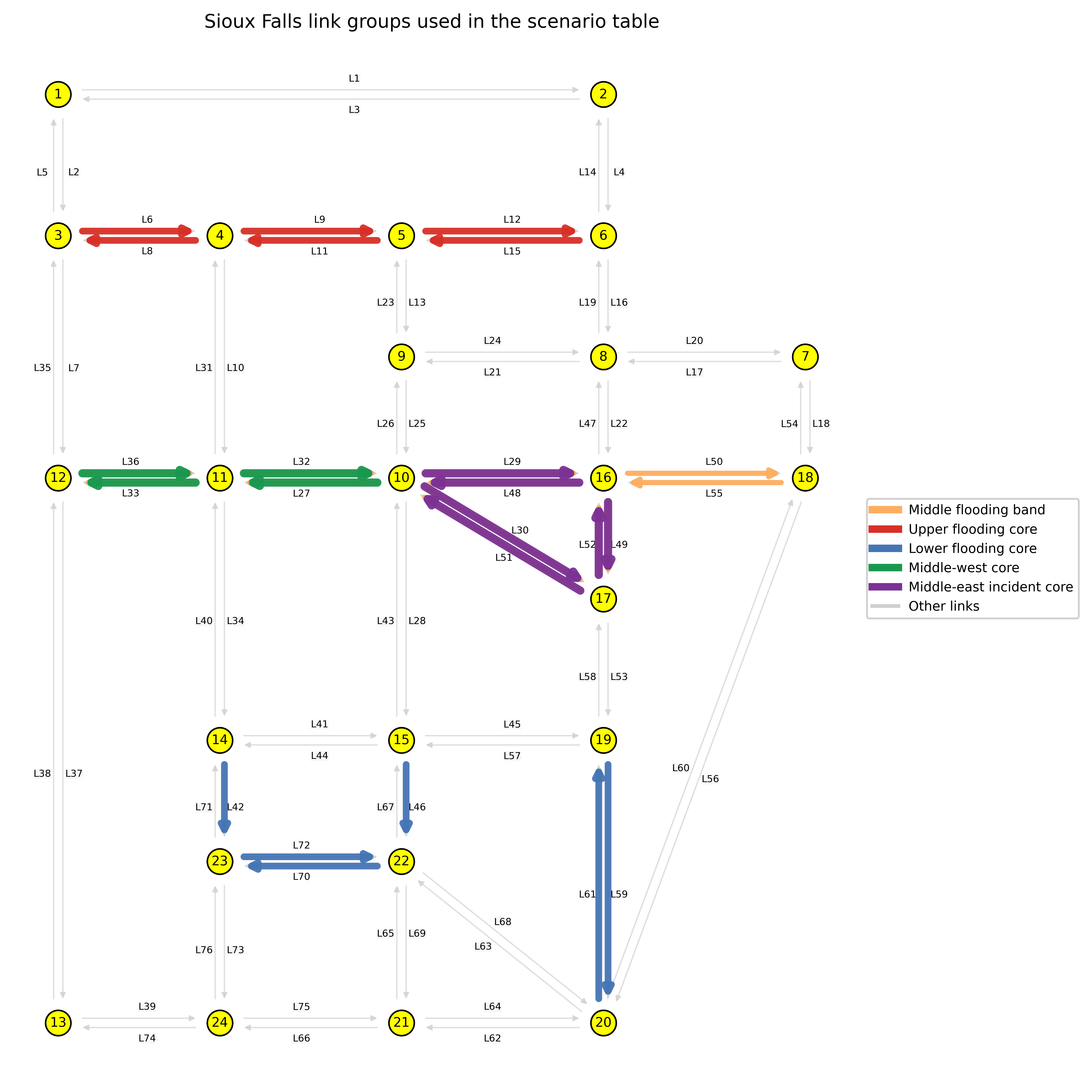}
\caption{Directed-link groups used to interpret the Sioux Falls scenario
library.  The highlighted links are the strongest-exposure links reported in
Table~\ref{tab:sioux_scenario_rules}; all directed links are labeled by link
ID.}
\label{fig:sioux_link_groups}
\end{figure}

\subsection{Scenario-conditioned TSUE benchmarks}
\label{subsec:sioux_scenario_tsue}

We solve one scenario-conditioned truncated SUE for each physical state in
Table~\ref{tab:sioux_scenario_rules}.  These six benchmark assignments use the
same feasible set and truncated kernel as the main models, but evaluate costs
under a single known scenario.  They are not risk-sensitive or robust; they
represent the flow pattern that would arise if the scenario were known before
route choice.  These conditional flows provide a baseline for comparing the
ex ante Model A, Model B, and DRO assignments.

\begin{figure}[htbp]
\centering
\includegraphics[width=0.4\textwidth]{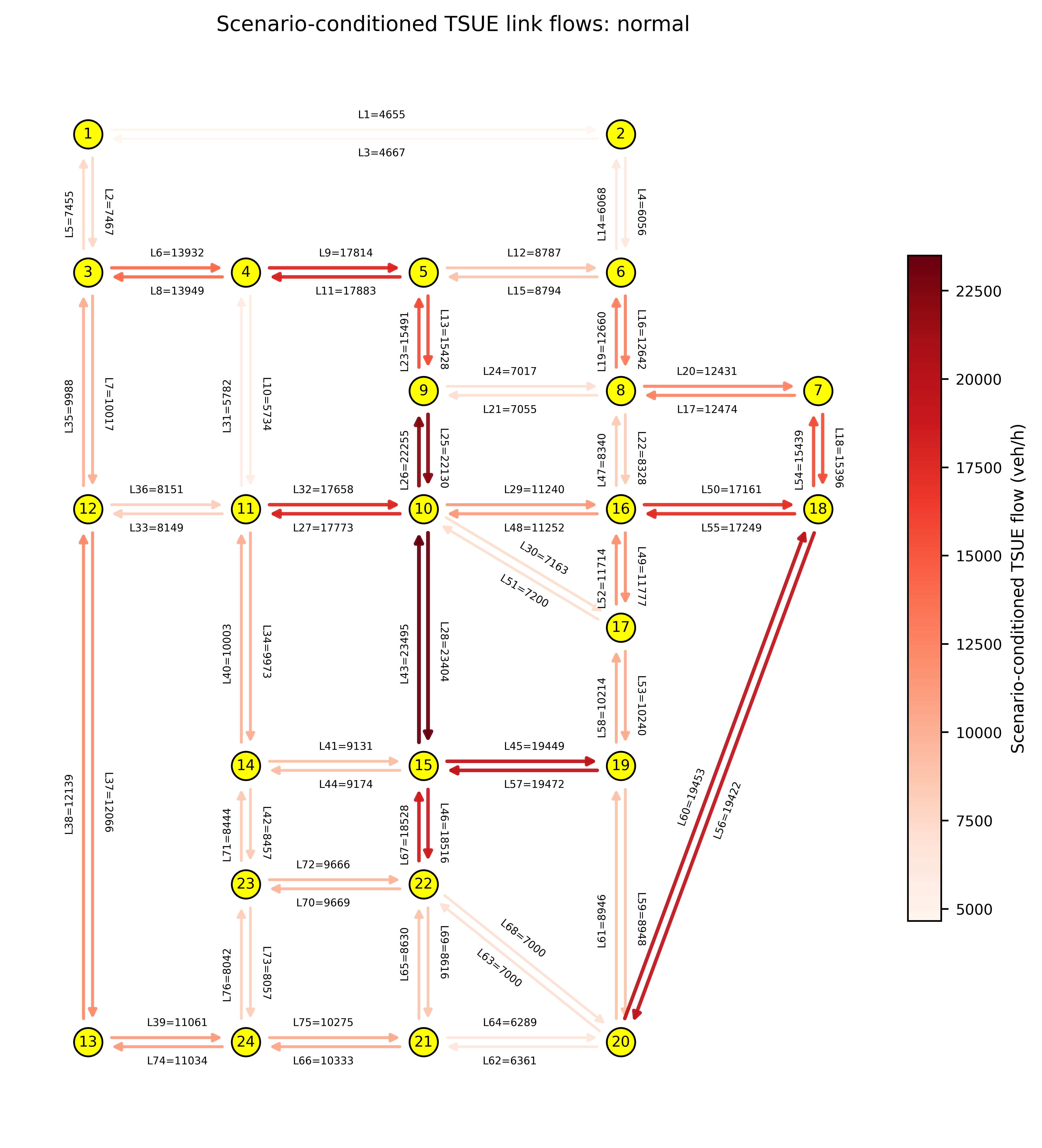}
\includegraphics[width=0.4\textwidth]{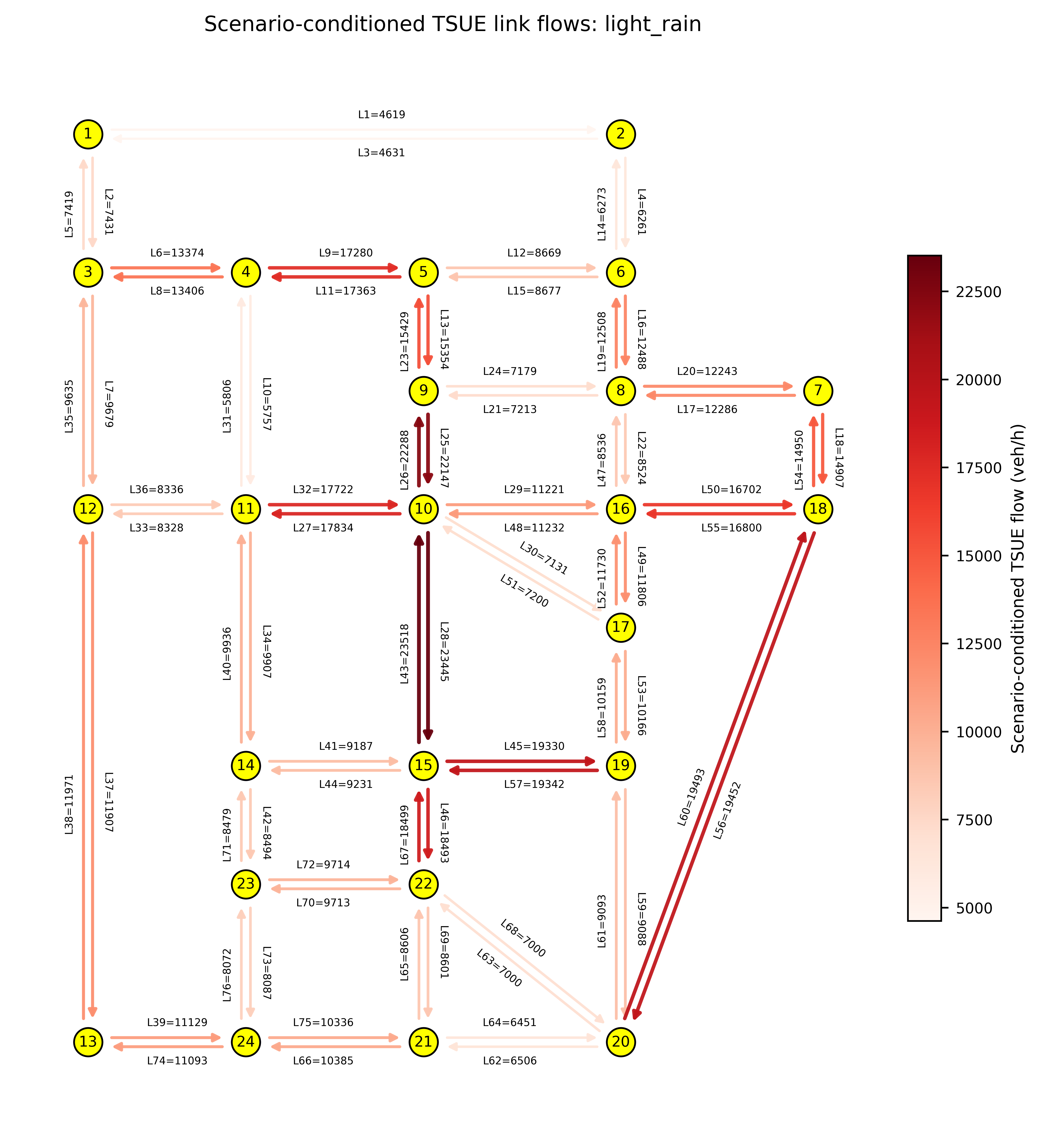}
\includegraphics[width=0.4\textwidth]{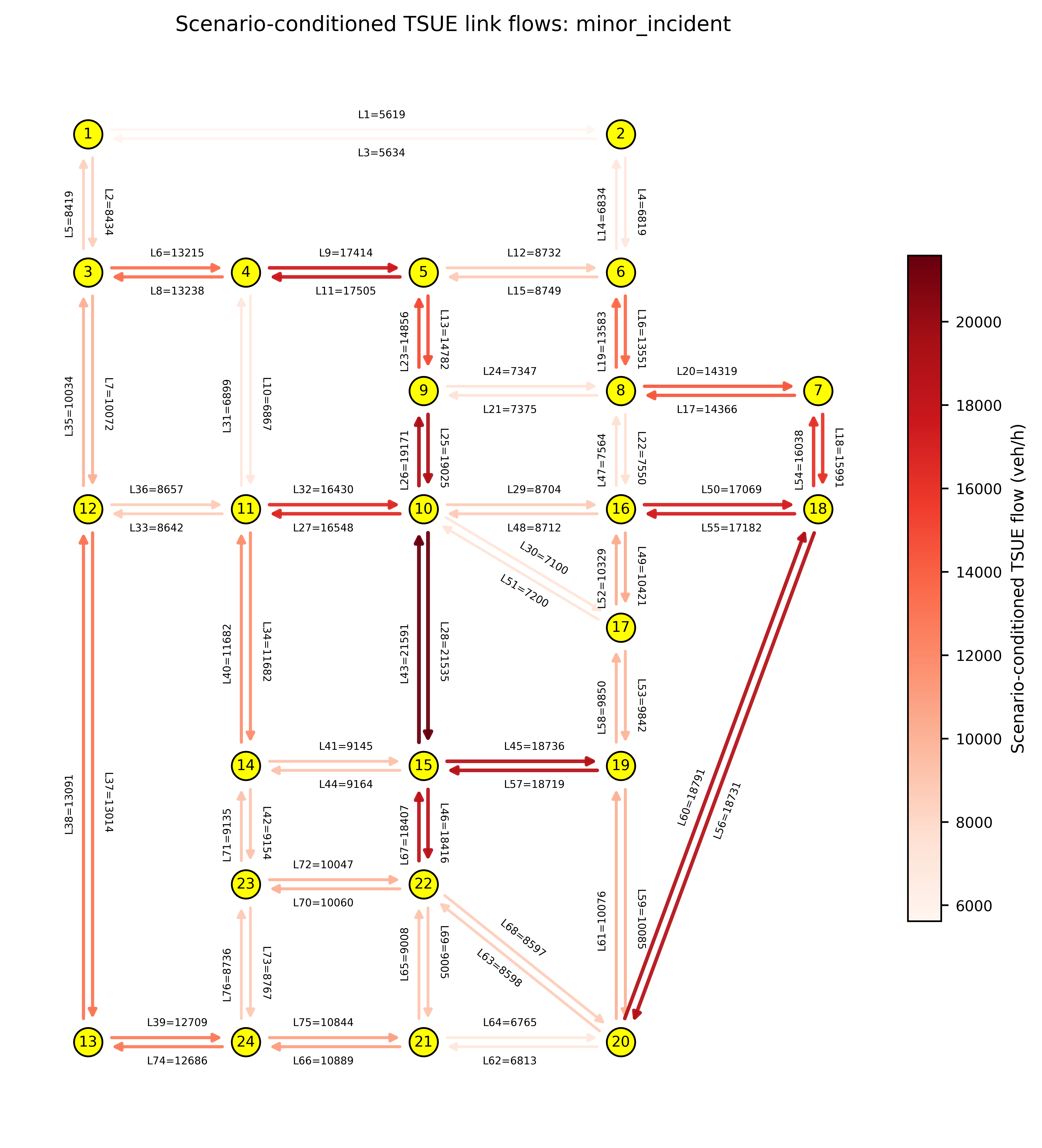}
\includegraphics[width=0.4\textwidth]{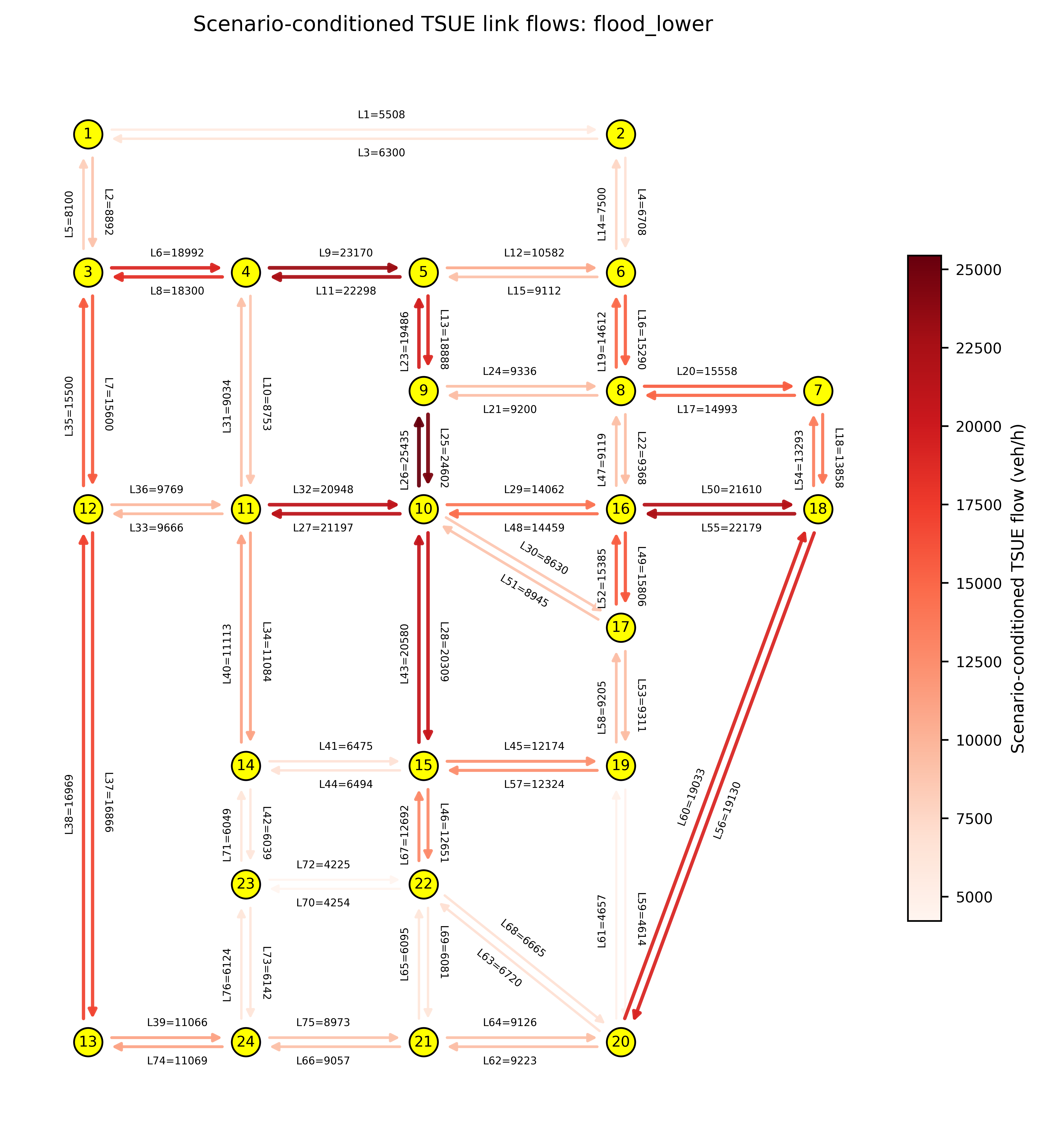}
\includegraphics[width=0.4\textwidth]{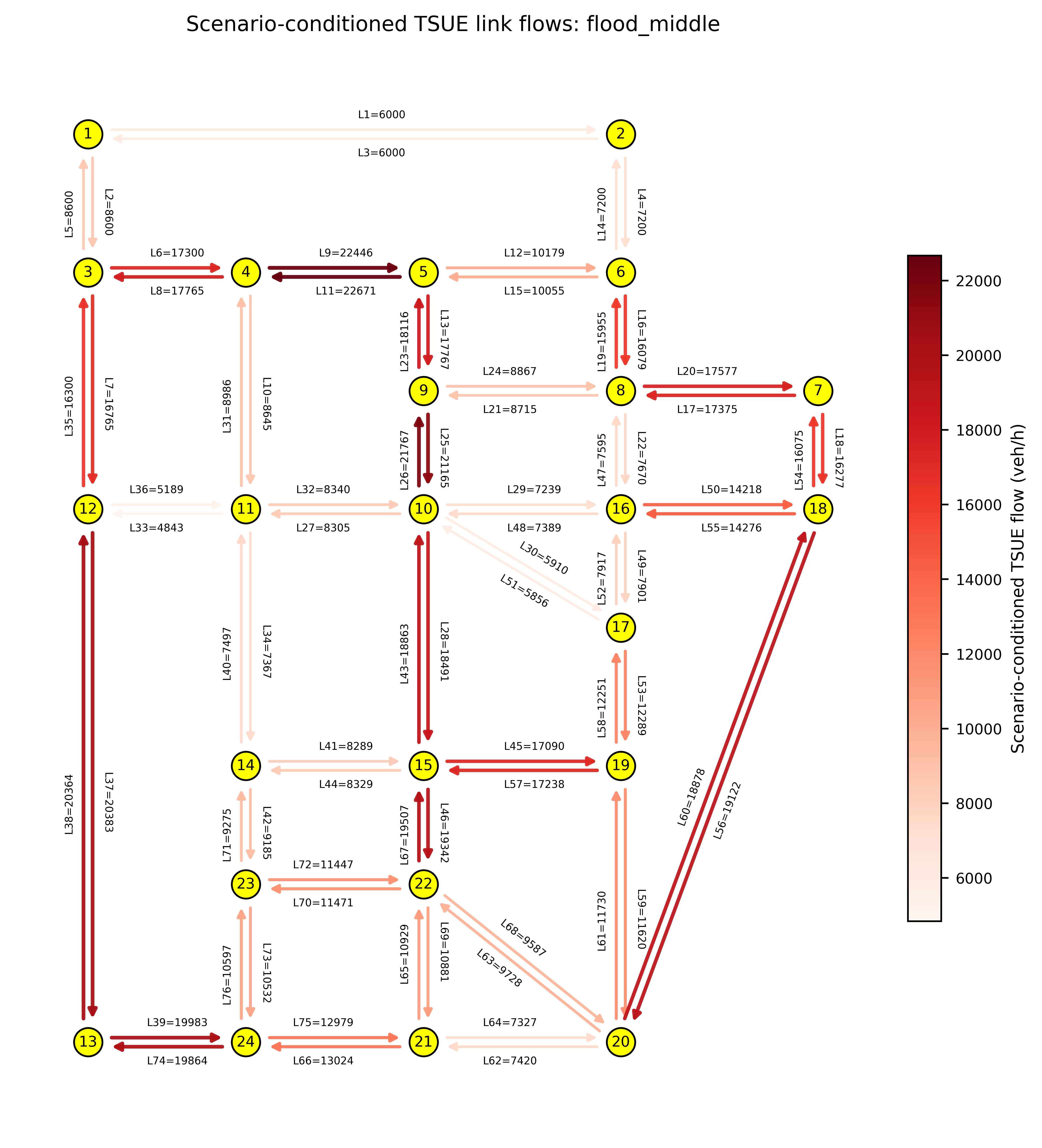}
\includegraphics[width=0.4\textwidth]{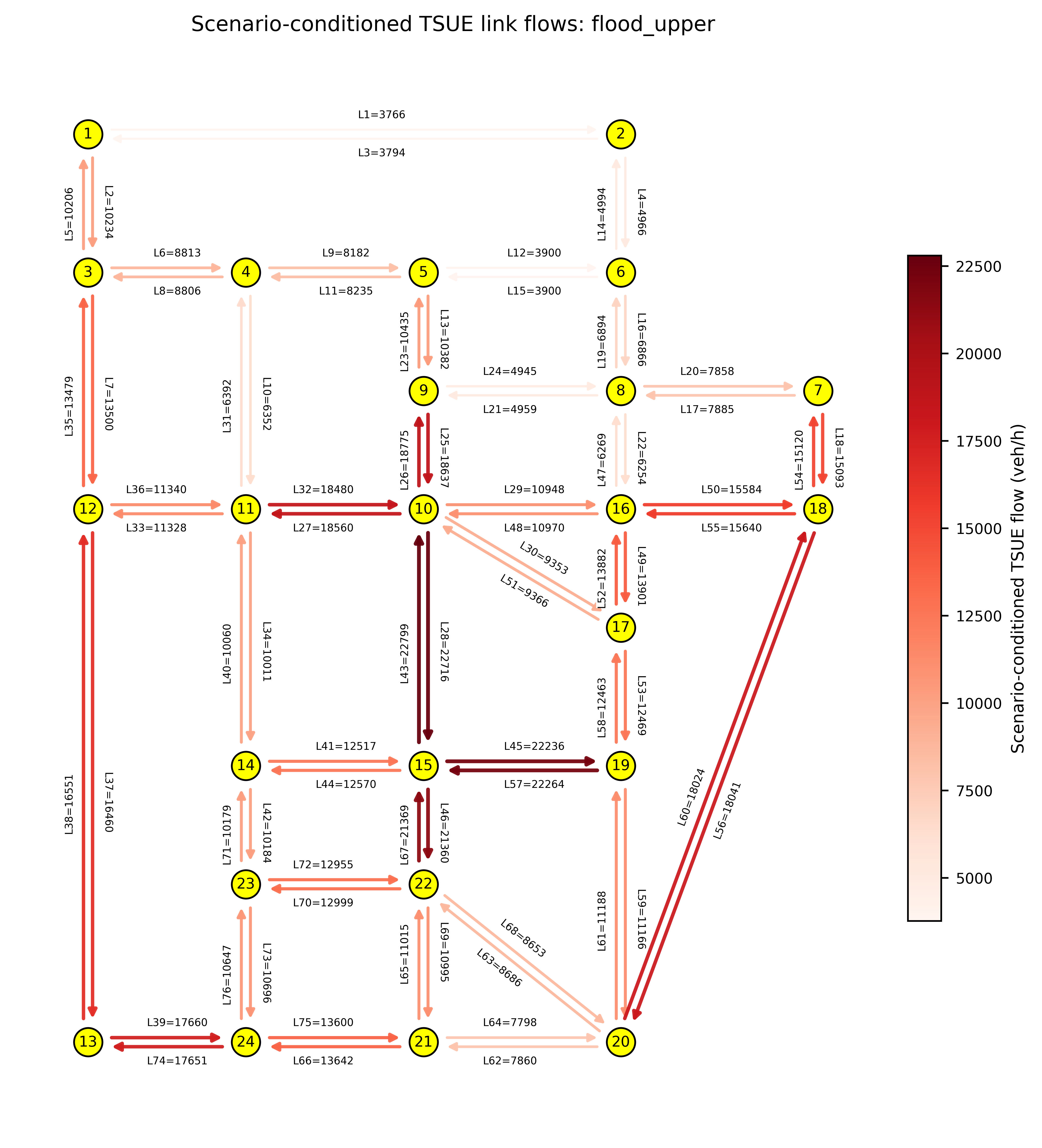}
\caption{Scenario-conditioned truncated stochastic user equilibrium link flows
under the six Sioux Falls physical states.}
\label{fig:sioux_scenario_tsue_benchmarks}
\end{figure}

\subsection{Model A--Model B benchmark relation}
\label{subsec:sioux_modelAB_relation}

We solve Model A and Model B on the full Sioux Falls network over
\[
\alpha\in\{0.80,0.85,0.90,0.96,0.965,0.975\},
\qquad
\eta\in\{0,0.2,0.4,0.6,0.8\}.
\]
For Model B at high confidence levels, we solve the finite-scenario
Rockafellar--Uryasev convex formulation in \eqref{eq:modelB_ru}, rather than
iterating on the selected tail law.  This avoids oscillations caused by changes
in the active CVaR tail selector when \(\alpha\) is close to one.  In the
common-tail case \(\alpha=0.90\), we verify the equality certificate implied by
Proposition~\ref{prop:common_tail}; when the certificate holds, the
Model A--Model B comparison uses the certified equality \(f^A=f^B\), rather
than a residual from independent fixed-point iterations.

Figure~\ref{fig:sioux_modelAB_heatmaps} summarizes the full-network relation
between the two models.  At \(\alpha=0.90\), the three flooding states form one
complete severe block, so \(D_f^{AB}\), \(D_p^{AB}\), and
\(\mathrm{Gap}_A(f^B)\) vanish up to numerical tolerance.  Below this cutoff,
the distances remain small because the tail includes a broader mixture of
states.  Above this cutoff, the tail cuts inside the heterogeneous flooding
block; route-specific Model A selectors and the common Model B selector can
then differ, producing larger \(D_p^{AB}\), larger \(D_f^{AB}\), and a positive
Model A residual at the Model B flow.

\begin{figure}[htbp]
\centering
\begin{subfigure}{0.6\textwidth}
\centering
\includegraphics[width=\textwidth]{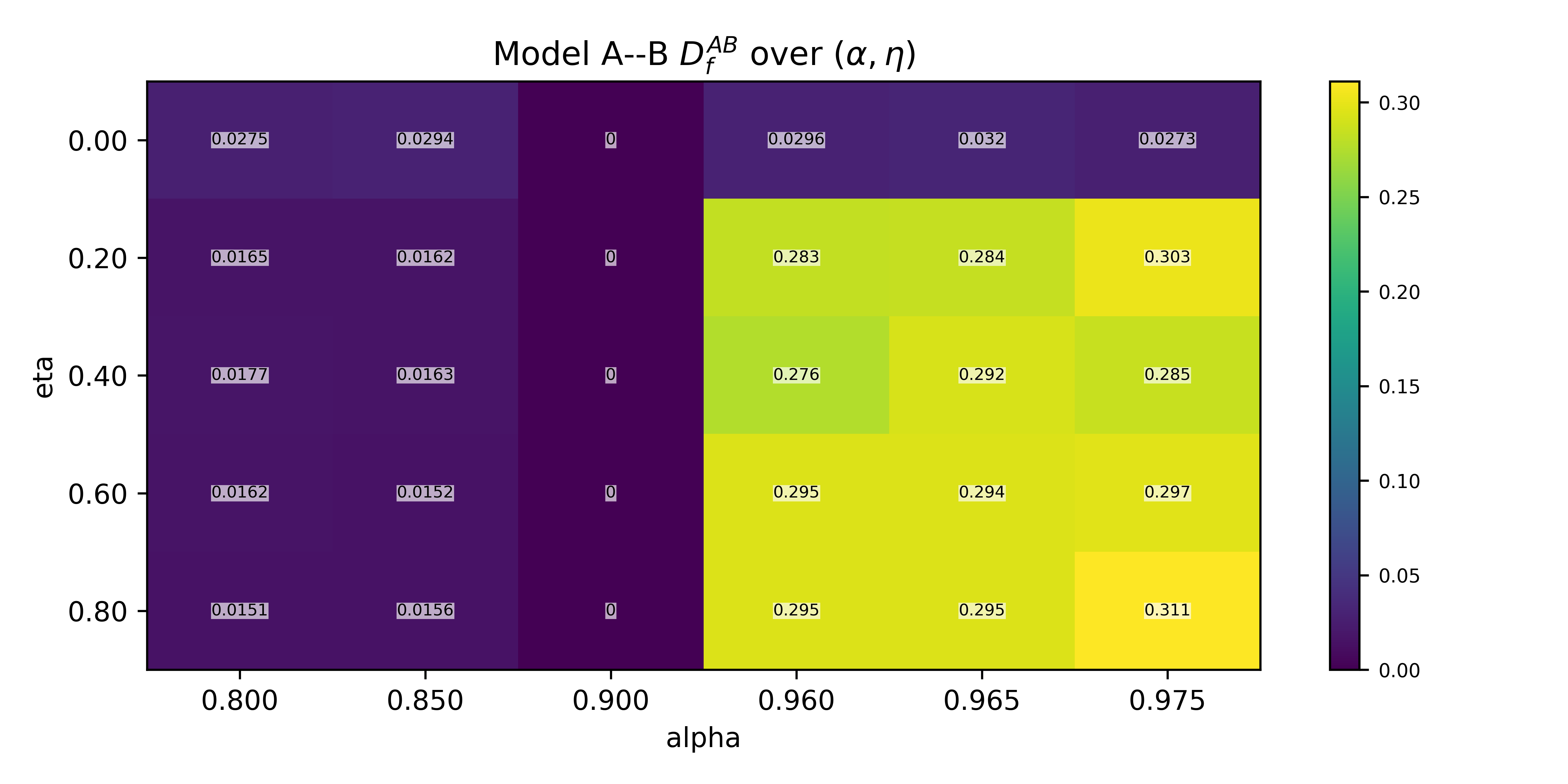}
\caption{\(D_f^{AB}\)}
\end{subfigure}

\begin{subfigure}{0.6\textwidth}
\centering
\includegraphics[width=\textwidth]{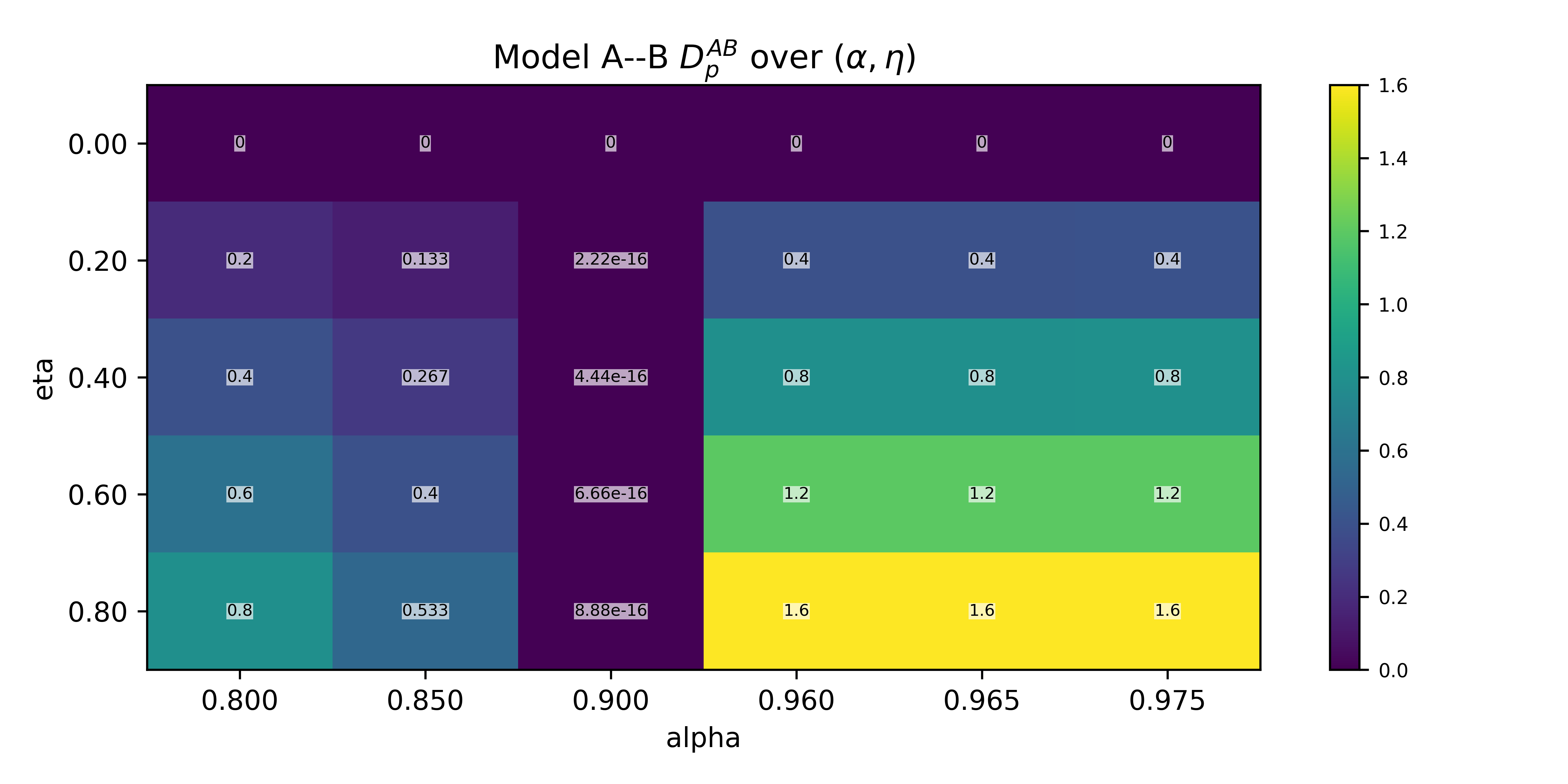}
\caption{\(D_p^{AB}\)}
\end{subfigure}

\begin{subfigure}{0.6\textwidth}
\centering
\includegraphics[width=\textwidth]{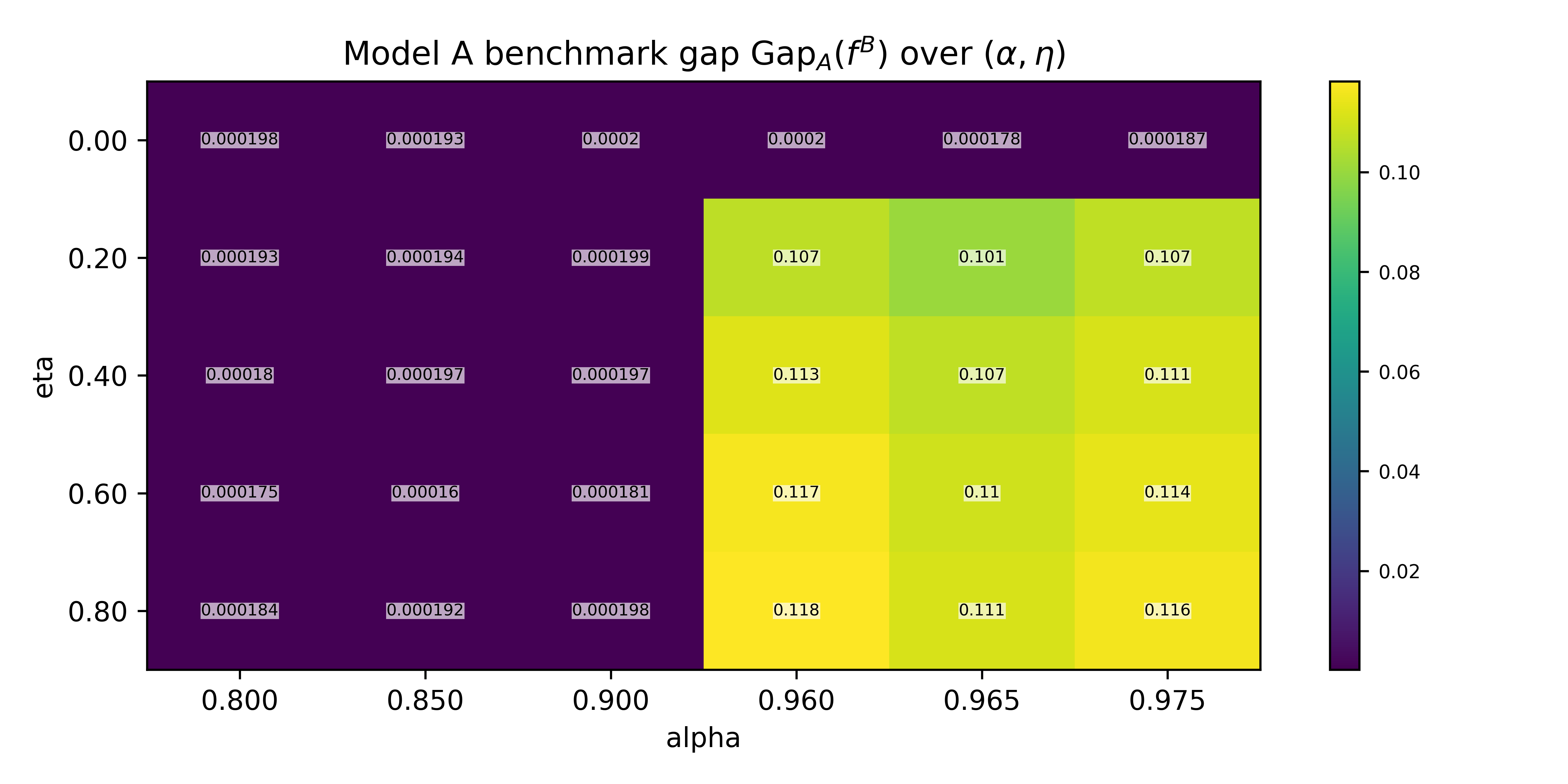}
\caption{\(\mathrm{Gap}_A(f^B)\)}
\end{subfigure}
\caption{Full-network Model A--Model B diagnostics over the
\((\alpha,\eta)\) grid.  Model A and Model B coincide in the common flooding
block case \(\alpha=0.90\); they separate when the CVaR cutoff falls inside the
heterogeneous flooding block.}
\label{fig:sioux_modelAB_heatmaps}
\end{figure}

\end{appendices}
\end{document}